\documentclass{proc-l}
\usepackage{amsfonts,latexsym,amsthm,amssymb,amsmath,amscd,euscript,tikz, tikz-cd}
\usepackage[alphabetic, msc-links, bibtex-style, nobysame]{amsrefs}
\usepackage{stackengine}
\usepackage{framed}
\usepackage{xfrac}
\usepackage[makeroom]{cancel}
\usepackage{faktor}
\usepackage{braket}
\usepackage{pgf,tikz,pgfplots}
\pgfplotsset{compat=1.17}
\usepgfplotslibrary{fillbetween} 
\usepackage{mathrsfs}
\usetikzlibrary{arrows}
\definecolor{wrwrwr}{rgb}{0.3803921568627451,0.3803921568627451,0.3803921568627451}
\definecolor{rvwvcq}{rgb}{0.08235294117647059,0.396078431372549,0.7529411764705882}
\definecolor{mblue}{rgb}{0.2, 0.3, 0.8}
\definecolor{morange}{rgb}{1, 0.5, 0}
\definecolor{mgreen}{rgb}{0.1, 0.4, 0.2}
\definecolor{mred}{rgb}{0.5, 0, 0}

\usepackage{hyperref}
    \hypersetup{colorlinks=true,citecolor=blue,linkcolor = blue,urlcolor =black,linkbordercolor={1 0 0}}
\usepackage{float}

\numberwithin{equation}{section}

\usepackage{lmodern}
\usepackage{supertabular}
\usepackage{verbatim}
\usepackage{amssymb}
\usepackage{enumerate}
\usepackage{stmaryrd}
\usepackage{bbm}
\usepackage{mathtools}
\usepackage{microtype}
\usepackage{setspace}
\usepackage{tikz,tikz-cd}
\usetikzlibrary{matrix,calc,positioning,arrows,decorations.pathreplacing,patterns,knots}

\newcommand{\la}{\langle}
\newcommand{\rg}{\rangle}

\newtheorem{theorem}{{Theorem}}[section]
\newtheorem*{theorem*}{Theorem}
\newtheorem{lemma}[theorem]{Lemma}
\newtheorem{proposition}[theorem]{Proposition}
\newtheorem{claim}[theorem]{Claim}
\newtheorem{example}[theorem]{Example}
\newtheorem{corollary}[theorem]{Corollary}
\newtheorem*{corollary*}{Corollary}

\usepackage{lmodern,url,enumerate,mathtools,microtype}
\usepackage[hmargin = 1in,vmargin=1in]{geometry}
\usepackage{graphicx}
\usepackage{subcaption}

\newcommand{\ve}{\varepsilon}

\newcommand{\mr}[1]{{\rm #1}}
\newcommand{\cE}{\mathcal{E}}\newcommand{\cF}{\mathcal{F}}
\newcommand{\cH}{\mathcal{H}}
\newcommand{\cI}{\mathcal{I}}\newcommand{\cJ}{\mathcal{J}}
\newcommand{\cL}{\mathcal{L}}
\newcommand{\cN}{\mathcal{N}}

\newcommand{\cR}{\mathcal{R}}
\newcommand{\cT}{\mathcal{T}}
\newcommand{\cV}{\mathcal{V}}

\newcommand{\bC}{\mathbb{C}}

\newcommand{\bH}{\mathbb{H}}

\newcommand{\bN}{\mathbb{N}}
\newcommand{\bP}{\mathbb{P}}
\newcommand{\bR}{\mathbb{R}}
\newcommand{\bS}{\mathbb{S}}\newcommand{\bT}{\mathbb{T}}

\newcommand{\bZ}{\mathbb{Z}}

\newcommand{\fp}{\mathfrak{p}}

\newcommand{\nc}{\newcommand}

\nc{\on}{\operatorname}
\nc{\p}{\partial}
\nc{\ol}{\overline}
\nc{\ul}{\underline}
\nc{\pa}{\partial}

\nc{\pb}{\partial_b}
\nc{\pc}{\partial_c}
\nc{\pd}{\partial_d}
\nc{\pe}{\partial_e}
\nc{\pf}{\partial_f}
\nc{\pg}{\partial_g}
\nc{\ph}{\partial_h}
\nc{\pari}{\partial_i}
\nc{\pj}{\partial_j}
\nc{\pk}{\partial_k}
\nc{\pl}{\partial_l}
\nc{\pell}{\partial_\ell}
\nc{\parm}{\partial_m}
\nc{\pn}{\partial_n}
\nc{\po}{\partial_o}
\nc{\pp}{\partial_p}
\nc{\pq}{\partial_q}
\nc{\pr}{\partial_r}
\nc{\ps}{\partial_s}
\nc{\pt}{\partial_t}
\nc{\pu}{\partial_u}
\nc{\pv}{\partial_v}
\nc{\pw}{\partial_w}
\nc{\px}{\partial_x}
\nc{\py}{\partial_y}
\nc{\pz}{\partial_z}

\nc{\pabar}{\partial_{\ol{a}}}
\nc{\pbbar}{\partial_{\ol{b}}}
\nc{\pcbar}{\partial_{\ol{c}}}
\nc{\pdbar}{\partial_{\ol{d}}}
\nc{\pebar}{\partial_{\ol{e}}}
\nc{\pfbar}{\partial_{\ol{f}}}
\nc{\pgbar}{\partial_{\ol{g}}}
\nc{\phbar}{\partial_{\ol{h}}}
\nc{\pibar}{\partial_{\ol{i}}}
\nc{\pjbar}{\partial_{\ol{j}}}
\nc{\pkbar}{\partial_{\ol{k}}}
\nc{\plbar}{\partial_{\ol{l}}}
\nc{\pellbar}{\partial_{\ol{\ell}}}
\nc{\pmbar}{\partial_{\ol{m}}}
\nc{\pnbar}{\partial_{\ol{n}}}
\nc{\pobar}{\partial_{\ol{o}}}
\nc{\ppbar}{\partial_{\ol{p}}}
\nc{\pqbar}{\partial_{\ol{q}}}
\nc{\prbar}{\partial_{\ol{r}}}
\nc{\psbar}{\partial_{\ol{s}}}
\nc{\ptbar}{\partial_{\ol{t}}}
\nc{\pubar}{\partial_{\ol{u}}}
\nc{\pvbar}{\partial_{\ol{v}}}
\nc{\pwbar}{\partial_{\ol{w}}}
\nc{\pxbar}{\partial_{\ol{x}}}
\nc{\pybar}{\partial_{\ol{y}}}
\nc{\pzbar}{\partial_{\ol{z}}}

\nc{\nababar}{\nabla_{\ol{a}}}
\nc{\nabbbar}{\nabla_{\ol{b}}}
\nc{\nabcbar}{\nabla_{\ol{c}}}
\nc{\nabdbar}{\nabla_{\ol{d}}}
\nc{\nabebar}{\nabla_{\ol{e}}}
\nc{\nabfbar}{\nabla_{\ol{f}}}
\nc{\nabgbar}{\nabla_{\ol{g}}}
\nc{\nabhbar}{\nabla_{\ol{h}}}
\nc{\nabibar}{\nabla_{\ol{i}}}
\nc{\nabjbar}{\nabla_{\ol{j}}}
\nc{\nabkbar}{\nabla_{\ol{k}}}
\nc{\nablbar}{\nabla_{\ol{l}}}
\nc{\nabelllbar}{\nabla_{\ol{\ell}}}
\nc{\nabmbar}{\nabla_{\ol{m}}}
\nc{\nabnbar}{\nabla_{\ol{n}}}
\nc{\nabobar}{\nabla_{\ol{o}}}
\nc{\nabpbar}{\nabla_{\ol{p}}}
\nc{\nabqbar}{\nabla_{\ol{q}}}
\nc{\nabrbar}{\nabla_{\ol{r}}}
\nc{\nabsbar}{\nabla_{\ol{s}}}
\nc{\nabtbar}{\nabla_{\ol{t}}}
\nc{\nabubar}{\nabla_{\ol{u}}}
\nc{\nabvbar}{\nabla_{\ol{v}}}
\nc{\nabwbar}{\nabla_{\ol{w}}}
\nc{\nabxbar}{\nabla_{\ol{x}}}
\nc{\nabybar}{\nabla_{\ol{y}}}
\nc{\nabzbar}{\nabla_{\ol{z}}}

\nc{\naba}{\nabla_{a}}
\nc{\nabb}{\nabla_{b}}
\nc{\nabc}{\nabla_{c}}
\nc{\nabd}{\nabla_{d}}
\nc{\nabe}{\nabla_{e}}
\nc{\nabf}{\nabla_{f}}
\nc{\nabg}{\nabla_{g}}
\nc{\nabh}{\nabla_{h}}
\nc{\nabi}{\nabla_{i}}
\nc{\nabj}{\nabla_{j}}
\nc{\nabk}{\nabla_{k}}
\nc{\nabl}{\nabla_{l}}
\nc{\nabm}{\nabla_{m}}
\nc{\nabn}{\nabla_{n}}
\nc{\nabo}{\nabla_{o}}
\nc{\nabp}{\nabla_{p}}
\nc{\nabq}{\nabla_{q}}
\nc{\nabr}{\nabla_{r}}
\nc{\nabs}{\nabla_{s}}
\nc{\nabt}{\nabla_{t}}
\nc{\nabu}{\nabla_{u}}
\nc{\nabv}{\nabla_{v}}
\nc{\nabw}{\nabla_{w}}
\nc{\nabx}{\nabla_{x}}
\nc{\naby}{\nabla_{y}}
\nc{\nabz}{\nabla_{z}}

\nc{\ola}{\ol{a}}
\nc{\olb}{\ol{b}}
\nc{\olc}{\ol{c}}
\nc{\old}{\ol{d}}
\nc{\ole}{\ol{e}}
\nc{\olf}{\ol{f}}
\nc{\olg}{\ol{g}}
\nc{\olh}{\ol{h}}
\nc{\oli}{\ol{i}}
\nc{\olj}{\ol{j}}
\nc{\olk}{\ol{k}}
\nc{\oll}{\ol{l}}
\nc{\olm}{\ol{m}}
\nc{\oln}{\ol{n}}
\nc{\olo}{\ol{o}}
\nc{\olp}{\ol{p}}
\nc{\olq}{\ol{q}}
\nc{\olr}{\ol{r}}
\nc{\ols}{\ol{s}}
\nc{\olt}{\ol{t}}
\nc{\olu}{\ol{u}}
\nc{\olv}{\ol{v}}
\nc{\olw}{\ol{w}}
\nc{\olx}{\ol{x}}
\nc{\oly}{\ol{y}}
\nc{\olz}{\ol{z}}

\title{Cohomogeneity two Ricci solitons with sub-Euclidean volume}
\date{\today}
\author[B. Firester]{Benjy Firester}
\address{{\href{mailto:benjyfir@mit.edu}{benjyfir@mit.edu}}}
\author[R. Tsiamis]{Raphael Tsiamis}
\address{
\href{mailto:r.tsiamis@columbia.edu}{r.tsiamis@columbia.edu}
}

\expandafter\def\expandafter\normalsize\expandafter{%
    \normalsize%
    \setlength\abovedisplayskip{1pt}%
    \setlength\belowdisplayskip{4pt}%
    \setlength\abovedisplayshortskip{-4pt}%
    \setlength\belowdisplayshortskip{1pt}%
}
\usepackage{xpatch}
\makeatletter
\xpatchcmd{\proof}{\topsep6\p@\@plus6\p@\relax}{}{}{}
\makeatother

\nc{\pbnski}{Pleba\'nski–Demia\'nski }

\nc{\metricQIsSqrtYSqXSq}{S8}
\nc{\metricSexBuffet}{S9}
\nc{\metricFaCK}{S7}
\nc{\metricStarStar}{S2}
\nc{\metricQIsSqrtXPlusY}{S3}
\nc{\metricStarStarBar}{S4}
\nc{\metricConfProdCigars}{S5}
\nc{\metricfifthOrderh}{S1}
\nc{\metricCfCf}{S6}

\begin{document}
\allowdisplaybreaks
\vspace{-0.2cm}
\begin{abstract}
    We introduce new families of four-dimensional Ricci solitons of cohomogeneity two with volume collapsing ends.
    In a local presentation of the metric conformal to a product, we reduce the soliton equation to a degenerate Monge-Amp\`ere equation for the conformal factor coupled with ODEs. 
    We obtain explicit complete expanding solitons as well as abstract existence results for shrinking and steady solitons with boundary.
    These families of Ricci solitons specialize to classical examples of Einstein and soliton metrics.
    We also classify local solutions of this Monge-Amp\`ere equation to prove rigidity for these solitons. 
\end{abstract}

\maketitle
\vspace{-0.75cm}
\section{Introduction}

Ricci solitons generalize the notion of Einstein metrics as self-similar solutions of the Ricci flow.
They arise as singularity models of the Ricci flow and are critical to understanding the landscape of geometric structures, a major open question in four-dimensional geometry. 
The soliton equation $\on{Ric} + \frac12 \cL_V g = \lambda g$ is a system of non-linear PDEs, for which most explicit solutions come from imposing a high degree of symmetry, such as metrics of cohomogeneity one or two, to simplify the equations.
We examine and construct a new class of four-dimensional metrics of cohomogeneity two specializing to known examples (see Table~\ref{fig:SolitonEncyclopedia}), and give conditions for the existence and rigidity of Ricci solitons.
These include families of complete expanding solitons as well as steady and shrinking solitons with boundaries along which the curvature blows up. 

Bryant developed a framework for studying gradient Ricci solitons with rotational symmetry \cite{bryant}, reducing the soliton equation to a second-order ODE system whose solutions have strong rigidity properties \cites{brendle-1, brendle-2}.
Koiso produced examples of solitons on compact K\"ahler manifolds, arising as cohomogeneity one toric K\"ahler metrics on $\bP^1$-bundles over $\bC \bP^n$ \cite{koiso}. 
Similar symmetries of the metric on appropriate bundles and blowups were employed by Cao and by Feldman-Ilmanen-Knopf to construct rotationally symmetric K\"ahler-Ricci solitons \cites{cao-rotational,fik}.
Bamler-Cifarelli-Conlon-Deruelle obtained another example in complex dimension two, namely a gradient K\"ahler-Ricci soliton with cohomogeneity two on $\mr{Bl}_p \bC \times \bP$ not arising via a reduction technique \cite{bccd}.
The classification of all K\"ahler-Ricci shrinker surfaces was completed by Li-Wang \cite{li-wang}.
Beyond the K\"ahler setting, the landscape of symmetric four-dimensional Ricci solitons is not fully understood and presents the natural next target.

Bamler's foundational work \cite{bamler-annals} showed that the pointed Gromov-Hausdorff limit of appropriate rescalings of any Ricci flow with a Type-I scalar curvature bound is a gradient singular Ricci shrinker with singularities of codimension four.
In four dimensions, the singularity models become more complicated and may have blowup singularities along different scales \cite{Appleton}, further motivating the study of Ricci solitons with incompleteness or singularities \cites{bcdmz, hallgren-1, hallgren-2}.
Alexakis-Chen-Fournodavlos produced singular Ricci solitons in all dimensions, whose metrics have curvature blowing up at a point and become instantaneously incomplete (but non-singular) under the Ricci flow diffeomorphisms \cite{singular-solitons}.
Hui exhibited unique global, singular, rotationally symmetric steady and expanding Ricci solitons of the form $\frac{da^2}{h(a^2)} + a^2 g_{\bS^n}$ \cite{hui}.
Singular metrics have also been employed as models to construct smooth Einstein metrics \cites{Hein+EtAl, Fu+EtAl}.

We study cohomogeneity two metrics satisfying the following geometric properties:  
\begin{equation}\label{eqn:metric-assumptions}
    \begin{tabular}{ll}
        (i) & the metric $g$ is toric of cohomogeneity two, \\
        (ii) &  the self-dual and anti self-dual Weyl curvature tensors $W^{\pm}$ have the same spectrum, and \\
        (iii) & the spectrum of $W^+$ has two distinct eigenvalues, one with multiplicity two.
    \end{tabular}
\end{equation}
The Ricci flow preserves the isometry group, and thus the toric symmetry (i).
The second condition, called pure curvature, is not preserved under the Ricci flow; however, it is found in conformal product metrics, which produce many known examples of Einstein metrics and Ricci solitons, including static solutions from general relativity such as the Schwarzschild metric.
The third property is conformally K\"ahler, which contains most examples of Einstein metrics and Ricci solitons \cites{bccd, cao-rotational, fik}.
A metric satisfying the properties~\eqref{eqn:metric-assumptions} can be locally given as conformal to a product of surfaces with metrics $dr_i^2 + f_i(r_i)^2 d\theta_i^2$ and radial variables $(r_1, r_2)$.
The conformal factor is invariant under the $\bT^2$-action, thus expressible as $q(r_1,r_2)^{-2}$.
These coordinates express the manifold as a torus fibration over a totally geodesic surface $\Omega$ given by coordinates $(r_1,r_2)$; $g$ further descends to a well-defined metric on $\Omega$, which we call the base metric. 

Under a natural coordinate change $(r_1,r_2)\rightsquigarrow (x,y)$ given in equation~\eqref{eqn:generalized-cylinder-A}, the determinant of the Euclidean Hessian $D^2 q(x,y)$ becomes an obstruction to the soliton equation, so $q$ satisfies the homogeneous Monge-Amp\`ere equation $\det D^2 q = 0$.
In Theorem~\ref{thm:solutions-to-monge-ampere}, we classify all local solutions to this equation, which has wider applications.
We say that the base metric is conformally cylindrical if $q(x,y)$ is homogeneous, i.e., $q(\tau x, \tau y) = \tau q(x,y)$. 
This condition eliminates the Monge-Amp\`ere obstruction and makes the origin an end of the manifold where the base metric is asymptotically cylindrical and the total metric is asymptotic to a cone over a two-torus. 

A sub-Euclidean volume growth property, such as conformally cylindrical, reduces the complexity of non-compact manifolds, as the non-compact pieces (ends) can be modeled on lower-dimensional examples.
Many soliton examples exhibit volume collapsing behavior: the Bryant soliton, the cigar soliton, as well as common Einstein metrics such as the Taub-NUT, Taub-Bolt, Schwarzschild, and Tian-Yau metrics \cite{Tian-Yau}. 
Using the conformally cylindrical property, we reduce the soliton equation to an ODE (see Section~\ref{sec:homogeneous}), and produce families of complete expanders as well as shrinking and steady solitons with boundary (see Table~\ref{fig:SolitonEncyclopedia}).

All previously known soliton metrics satisfying~\eqref{eqn:metric-assumptions} have a conformally cylindrical base; in the Einstein case, the conformal factor is the inverse of the scalar curvature squared \cite{derdzinski}.
This property is also implied by either of the surface factors having constant non-zero Gauss curvature.
We classify Ricci solitons with conformally cylindrical base and present new examples of solitons of any $\lambda$ exhibiting various geometric behaviors.
These include a new family of explicit complete expanding solitons~\eqref{metric:q=sqrtx2+y2} as well as shrinking and steady solitons with boundary $\mr{(}$\hyperref[cor:SexBuffet]{$\mr{\metricSexBuffet}$}$\mr{)}$, along which the curvature blows up.

\begin{theorem}\label{thm:main-theorem}
A Ricci soliton $g$ of the form~\eqref{eqn:metric-assumptions} that is conformally cylindrical is one of:
\begin{enumerate}[(i)]
    \item Einstein, specifically one of Schwarzschild, Pleba\'nski-Demia\'nski~\eqref{metric:pbnski},
    \item a product metric of compatible Einstein or soliton surfaces (including cylinders $\bR^2 \times \bS^2)$,
    \item locally conformal to $\bS^2 \times \bH^2$ or $\bR^4$ with conformal factor given by solving the ODE~\eqref{eqn:resulting-ODE}, explicitly given in~\eqref{metric:SStar} (cf.~Table~\ref{fig:SolitonEncyclopedia}~\eqref{metric:q=sqrtx2+y2} and $\mr{(}$\hyperref[cor:SexBuffet]{$\mr{\metricSexBuffet}$}$\mr{)}$), or
    \item a member of the family~\eqref{eqn:singular-soliton-metrics} of singular steady solitons degenerating to the Schwarzschild metric. 
\end{enumerate}
\end{theorem}
We further classify Ricci solitons where one of the surface factors has constant Gauss curvature, which, if the curvature is non-zero, implies the metric is conformally cylindrical.
\begin{theorem}\label{thm:secondary-thm}
    Suppose $g$ is a metric of the form~\eqref{eqn:metric-assumptions} in which one of the local surface factors has constant Gauss curvature $K$.
    If $g$ is a Ricci soliton, then
\begin{enumerate}[(i)]
    \item for $K = 0$, the metric is: Einstein, a product of Einstein/Ricci soliton surfaces, \eqref{eqn:warped-product-metric} with $b_2 = 0$, $\mr{(}$\hyperref[prop:metrics-s2-through-s5]{\mr{\metricStarStar}}$\mr{)}$, \eqref{eqn:new-soliton}, $\mr{(}$\hyperref[prop:metrics-s2-through-s5]{\mr{\metricStarStarBar}}$\mr{)}$, \eqref{eqn:Soliton-q=exp(x+y)-metric} with $k_2 = 0$, or \eqref{eqn:conformally-flat-metric}; and
    \item for $K \neq 0$, the factor is a piece of a hyperbolic cusp, the soliton is \eqref{eqn:warped-product-metric} with $C_2 = 0$, or the metric is conformally cylindrical and therefore characterized by Theorem~\ref{thm:main-theorem}.
\end{enumerate}
\end{theorem}
The conformally cylindrical base property means that the torus directions grow linearly near the origin $(x,y) = (0,0)$.
Complete metrics of the form~\eqref{metric:SStar}, including~\eqref{metric:q=sqrtx2+y2} with $b_0 > 0$, come in two types: conformally scalar flat when $c \neq 0$, and conformally flat when $c = 0$. 
In the former case, the origin is the only end, which is asymptotic to a cone over a two-torus. 
In the latter case, there is a second end, when $(x,y) \to \infty$.
On this end, the tori collapse, so the metric is asymptotically cylindrical, converging to the cylindrical base metric. 
The same holds for metrics with a boundary, such as $\mr{(}$\hyperref[cor:SexBuffet]{$\mr{\metricSexBuffet}$}$\mr{)}$, and the boundary extends to infinity.

Metrics of the form~\eqref{eqn:metric-assumptions} are actively studied and contain key geometric examples.
Similar cohomogeneity two constructions have been studied in the Einstein as well as in the K\"ahler-Einstein and K\"ahler-Ricci soliton settings, to disprove a conjecture of Cao \cites{ambitoric, apostolov-cifarelli}.
In addition to generalized cylinders, the class~\eqref{eqn:metric-assumptions} contains classical Einstein metrics such as the Riemannian Schwarzschild manifold and the \pbnski metric. 
The latter was recently employed by Alvarado-Ozuch-Santiago to produce degenerating families of Einstein metrics on $\bR^4$ with ends of novel geometry \cite{alvarado-fams}.

As classified by \cite{catino-mantegazza-mazzieri}, a complete, locally conformally flat shrinking or steady Ricci soliton is isometric to a quotient of $\bS^n, \bS^{n-1}\times \bR, \bR^n$, or the Bryant soliton.
Expanding this framework to allow incompleteness, our construction obtains a range of possible singular and new volume collapsing behaviors, as well as Ricci solitons with boundary.
We additionally produce complete locally conformally flat expanding solitons~\eqref{metric:q=sqrtx2+y2} that do not fall under their classification. 
With the exception of the Einstein metrics and appropriate products of lower dimensional known solitons, all the metrics described are previously unknown. 
In addition to the new complete expanding soliton~\eqref{metric:q=sqrtx2+y2}, the $\mr{(}$\hyperref[cor:SexBuffet]{$\mr{\metricSexBuffet}$}$\mr{)}$ metrics are new singular metrics that exhibit a singularity type like that of the Ooguri-Vafa metric, except in cohomogeneity two instead of one.
It would be interesting to use such metrics in gluing-desingularization constructions of complete cohomogeneity two Ricci solitons, analogous to constructions of complete hyperk\"ahler manifolds by gluing in Ooguri-Vafa metrics, such as~\cite{cvz-collapsing-k3}.

\begin{table}
    \centering
    \begin{tabular}{|c||c|c|c|c|}\hline 
    & $A,B,q$&$\lambda$ & Geometry & Eqn. \\ \hline \hline

         \eqref{eqn:warped-product-metric}
        & \begin{tabular}{c}
            $A =$ \eqref{eqn:a-from-mathematica} \\
             $B = b_2 y^2 + b_1y +b_0$  \\
            $q = h(x)$ 
         \end{tabular} &

            $\lambda$

         & \begin{tabular}{c}A warped product soliton with a flat surface factor\\
         Corollary~\ref{cor:familyOfS1Metric} exhibits a family of such solitons
         \end{tabular}
         &\eqref{eqn:fifth-order-ODE-for-h} \\ \hline

         (\hyperref[prop:metrics-s2-through-s5]{\metricStarStar})& \begin{tabular}{c}
            $A = a_0$ \\
             $B = b_0$  \\
            $q = \theta(x+y)$   
         \end{tabular} &
         
        $\lambda$

         & \begin{tabular}{c}
               Conformally flat and geometry depending on $\theta $ \\
               For $\lambda \geq 0$, boundary, singular, or quotient of $\bR^{2+k} \times \bS^{2-k}$
         \end{tabular}&~\eqref{eqn:theta-ODE-1}\\ \hline
         
        \eqref{eqn:new-soliton} & \begin{tabular}{c}
            $A = x$ \\
             $B = y$  \\
            $q = \sqrt{x + y}$   
         \end{tabular} &
         
        $\frac12$
        
         & Singular and conformally flat & Explicit\\ \hline

        (\hyperref[prop:metrics-s2-through-s5]{\metricStarStarBar})& \begin{tabular}{c}
            $A = x$ \\
             $B = y$  \\
            $q = \theta(x+y)$   
         \end{tabular} &

            $\lambda$

         & 
         \begin{tabular}{c}
               Geometry depending on $\theta $ \\
               For $\lambda \geq 0$, boundary, singular, or quotient of $\bR^{2+k} \times \bS^{2-k}$ \\
               $\theta = t$ recovers~\eqref{metric:pbnski}; $\theta = \sqrt{t}$ recovers~\eqref{eqn:new-soliton}
         \end{tabular}
         &~\eqref{eqn:theta-ODE-2}\\ \hline

        \eqref{eqn:Soliton-q=exp(x+y)-metric}& \begin{tabular}{c}
            $A = k_1 e^x + k_0$ \\
             $B = k_2 e^y - k_0$  \\
            $q = e^{x+y}$   
         \end{tabular} &

            0

         & Singular metric conformal to product of cigars& Explicit\\ \hline

       \eqref{eqn:conformally-flat-metric}& \begin{tabular}{c}
            $A = 1$ \\
             $B = 1$  \\
            $q$ solves two PDEs
         \end{tabular} &

        $ \lambda$

         & \begin{tabular}{c}
        Pair of PDEs encompassing all conf.~flat toric metrics \\
        Recovers (\hyperref[prop:metrics-s2-through-s5]{\metricStarStar}) when \eqref{eqn:px-s-tilde} and \eqref{eqn:py-s-tilde} are equivalent\\
        Includes \eqref{metric:q=sqrtx2+y2} and $\mr{(}$\hyperref[cor:SexBuffet]{$\mr{\metricSexBuffet}$}$\mr{)}$ for $(a_0, b_0,c) = (1,1,0)$
         \end{tabular}
         & (\ref{eqn:px-s-tilde}, \ref{eqn:py-s-tilde})   \\ \hline

       \eqref{eqn:singular-soliton-metrics}& \begin{tabular}{c}
            $A =$ \hyperref[eqn:singular-soliton-metrics]{$F_{\alpha, c, k_1}(x)$} \\
             $B =$ \hyperref[eqn:singular-soliton-metrics]{$F_{1-\alpha, -c, k_2}(y)$}  \\
            $q = x^\alpha y^{1-\alpha}$   
         \end{tabular} &

        $ 0$

         & \begin{tabular}{c}
         Singular family specializing to Schwarzschild \\
         \end{tabular}& Explicit \\ \hline

        \eqref{metric:q=sqrtx2+y2} & \begin{tabular}{c}
            $A = a_0 - cx^2$ \\
             $B = b_0 + cy^2$  \\
            $q = \sqrt{a_0y^2 + b_0x^2}$   
         \end{tabular} &
        
            $-2a_0b_0$

         & 
         \begin{tabular}{c}
              Singular shrinking for $b_0 < 0$  \\
              Riemannian Schwarzschild when $b_0 = 0$ \\
              Complete volume collapsing expander $b_0 > 0$
         \end{tabular}

         & Explicit \\ \hline
         
         $\mr{(}$\hyperref[cor:SexBuffet]{$\mr{\metricSexBuffet}$}$\mr{)}$ 
         & \begin{tabular}{c}
            $A = a_0 - cx^2$ \\
             $B = b_0 + cy^2$  \\
            $q = yf(x/y)$   
         \end{tabular} &
        
            $\geq 0 $

         & 
         \begin{tabular}{c}
         Ricci solitons with boundary \\
         $c \neq 0$: conf.~scalar flat, boundary at $q=0$ \\
        $c = 0$: conf.~flat, asymptotic end of $c\neq0$
         \end{tabular}
         &~\eqref{eqn:resulting-ODE}\\ \hline

    \end{tabular}
    \caption{Ricci soliton properties determined by the conditions and equations they satisfy.
    The functions $A,B,q$ in the second column are the entries of~\eqref{eqn:PaperMetric} resulting in soliton metrics.
    }
    \label{fig:SolitonEncyclopedia}
\end{table}

\textbf{Acknowledgments.} We are thankful to Tristan Ozuch for suggesting this topic and for many helpful conversations.
Additionally, we are very grateful to the anonymous referee for the detailed feedback and the thoughtful suggestions.

\section{Reduction of the soliton equation}\label{sec:reduction}
Metrics satisfying the geometric conditions~\eqref{eqn:metric-assumptions} above can be expressed as locally conformal to a product of warped product surfaces.
These surfaces $(\Sigma, g)$ have metrics locally expressed as
\begin{equation}\label{eqn:generalized-cylinder}
    g = dr^2 + f(r)^2 ds^2
\end{equation}
where $s$ is the coordinate on a unit circle of intrinsic diameter $\pi$. 
Such metrics occur naturally in surfaces that have constant curvature or are extremal (in the Calabi sense).
The constant curvature cases have warping functions $f \in \left\{ r, R \sin (r/R), R \sinh (r/R) \right\}$ for flat, round, and hyperbolic metrics, respectively.
Since a generalized cylinder metric on $N \times \bR$ has the form $g = dr^2 + g_N$, we obtain a cylindrical metric with a volume collapsing end at $r \to \infty$ by considering  $dr^2 + f(r)^2 g_N$ where $f(r) = o(r)$.
If $g_N = \bS^{n-1}$ the round unit sphere, the cone angle of such a metric at $r = 0$ is given by $\Theta := \lim_{\rho \to 0^+} \frac{\cH^{n-1}(\p B_\rho)}{\rho^{n-1}}$, where $B_\rho$ consists of the points of distance within $\rho$ of the origin with respect to $g$, and $\cH^{n-1}$ denotes the $(n-2)$-dimensional Hausdorff measure.
The angle $\Theta$ must be $\omega_n = \mr{Vol}(\bS^{n-1})$ for the metric $g$ to extend smoothly to the origin.
When $f(0) = 0$, these coordinates can be viewed as polar coordinates, and the $s$-variable is in the unit $\bS^1$.
The cone angle is therefore given by $2\pi f'(0)$, so the smoothness requirement corresponds to $f = r + O(r^2)$ near $r = 0$ and $f^{(2k)}(0) = 0$ for $k \geq 1$, by \cite{verdiani-ziller}*{Theorem 2}.
These are satisfied by all the constant curvature metrics above. 
A singularity is conical if $f = \alpha r + O(r^2)$ where $\alpha \neq 1$, the classical cones are for $\alpha < 1$.

Replacing the coordinate $r$ by $x(r)$ such that $x' = f$ transforms the metric into
\begin{equation}\label{eqn:generalized-cylinder-A}
g_{\Sigma}(A) := \dfrac{dx^2}{A(x)} + A(x) \, ds^2, \qquad A := (f \circ x^{-1})^2
\end{equation}
which will be the standard form for our purposes.
Expression~\eqref{eqn:generalized-cylinder} is equivalent to $g = f_1(r)^2 \, dr^2 + f_2(r)^2 \, ds^2$ using a coordinate $\tilde{r}'(r) = f_1$.
The above coordinate transformation for $x$ shows that any metric of the form $A(x) \, dx^2 + C(x) \, ds^2$ is locally conformal to a generalized cylinder metric~\eqref{eqn:generalized-cylinder-A}.
A metric of the form~\eqref{eqn:generalized-cylinder-A} is determined by $A$ and denoted by $g_{\Sigma}(A)$.
In coordinates $(\tilde{x},\tilde{s}) := \left (a_1 x + a_0, \frac{s}{a_1 \lambda^2} \right)$ for $a_1 \neq 0$, $g_{\Sigma}(A)$ becomes
\begin{equation}\label{eqn:scale-rescale-translate}
\lambda^2 g_{\Sigma} (\tilde{A}) = \lambda^2 \left( \frac{1}{\tilde{A}(\tilde{x})^2} d \tilde{x}^2 + \tilde{A} (\tilde{x})^2 \, d \tilde{s}^2 \right), \qquad \text{where } \; \tilde{A}(\tau) := a_1 \lambda A \left( \frac{\tau-a_0}{a_1} \right).
\end{equation}
These transformations effectively scale the overall metric $g$, and rescale and translate $x$ and $y$.
 
The above constant curvature metrics are characterized by $A(x)$ being a quadratic with quadratic coefficient equal to $-1$ times the curvature, including $A(x)$ being linear or constant for flat metrics.
This property comes from transforming~\eqref{eqn:generalized-cylinder-A} into~\eqref{eqn:generalized-cylinder}, via
\begin{align}
    \frac{dx^2}{a_0 - x^2} + (a_0 - x^2) \, ds^2 &= d \rho^2 + a_0 \sin^2 \rho \, ds^2, \qquad & & \rho = \arccos ({x/\sqrt{a_0}}), \label{eqn:metric:a0-x2} \\
    \frac{dx^2}{x^2 - b_0} + (x^2 - b_0)ds^2 &= d\rho^2 + b_0 \sinh^2\rho \, ds^2, \qquad & & \rho = \on{arccosh} (x/\sqrt{b_0}), \label{eqn:metric:x2-b0} \\
    \frac{dx^2}{x^2} + x^2 \, ds^2 &= d \rho^2 + e^{2 \rho} \, ds^2, \qquad & & \rho = \log x, \label{eqn:metric:x2} \\
    \frac{dx^2}{x^2 + b_0} + (x^2 + b_0) \, ds^2 &= d \rho^2 + b_0 \cosh^2 \rho \, ds^2, \qquad & & \rho = \on{arcsinh}(x / \sqrt{b_0}). \label{eqn:metric:x2+b0}
\end{align}
for $a_0, b_0 > 0$.
As in \cite{verdiani-ziller}*{Theorem 2}, the smoothness of these metrics is characterized as follows: 
\begin{enumerate}[(i)]
    \item the metric~\eqref{eqn:metric:a0-x2} is spherical, of constant positive curvature. 
    The sphere closes up smoothly at the two points $x = \pm \sqrt{a_0}$ where $A$ vanishes, if and only if $a_0 = 1$.
    If $a_0 \neq 1$, those points are conical singularities, including the ``football'' metric.
    \item the metric~\eqref{eqn:metric:x2-b0} corresponds to hyperbolic space, with constant negative curvature, defined for either $x > \sqrt{b_0}$ or $x < - \sqrt{b_0}$.
    The surface closes up smoothly at $x = \pm \sqrt{b_0}$ if and only if $b_0 = 1$, otherwise, it contains a conical singularity.
    \item the metric~\eqref{eqn:metric:x2} is the hyperbolic cusp metric, expressed in horocyclic coordinates. 
    This space is $\bH^2 /\bZ$ expressible as the quotient of the upper half-space by the Kleinian group generated by $z \mapsto z + 1$. 
    \item the metric~\eqref{eqn:metric:x2+b0} represents a complete hyperbolic cylinder of infinite width, whose unique closed geodesic has length equal to $2\pi\sqrt{b_0}$, cf.~\cite{small-eigenvalues}*{Lemma 9}.
\end{enumerate}

The two hyperbolic cylinders in equations~\eqref{eqn:metric:x2-b0} and \eqref{eqn:metric:x2} can be formed by identifying a pair of two-sided geodesic rays in $\bH^2$.
If these geodesics are parallel, i.e., they share a boundary point in $\p_\infty \bH^2$, this quotient achieves the cusp metric \eqref{eqn:metric:x2}.
Otherwise, the distance between the two geodesics is well-defined and equal to $2\pi\sqrt{b_0}$, so the metric is \eqref{eqn:metric:x2}.
The $\bS^1$-action on these hyperbolic cylinders has no fixed points and is unique and rigid, unlike the rotational $\bS^1$-action in \eqref{eqn:metric:x2-b0} which has $\rho = 0$ as a fixed point and depends on a basepoint in $\bH^2$. 
In particular, annular regions $\{ R_1 \leq \rho \leq R_2\}$ for metrics \eqref{eqn:metric:x2-b0}, \eqref{eqn:metric:x2}, and \eqref{eqn:metric:x2+b0} are not isometric, so the property of being part of a hyperbolic cusp or cylinder is metrically distinguishable from being an arbitrary hyperbolic annulus.
This difference is detected in Theorem~\ref{thm:secondary-thm}, in which pieces of hyperbolic cusps are less coercive and do not force soliton metrics to be conformally cylindrical.

More complicated metrics also appear in the class~\eqref{eqn:generalized-cylinder-A}, for example, the cigar soliton, a steady soliton with $\lambda = 0$ and $V = -2u\pu - 2v\pv$.
The cigar metric is commonly expressed as
\begin{equation}\label{eqn:cigar-soliton}
    g = \frac{du^2 + dv^2}{1 + u^2 + v^2} = d r^2+ \tanh^2(r)\,ds^2.
\end{equation}
In the convention~\eqref{eqn:generalized-cylinder-A}, this metric corresponds to $A(x) = 1 - e^{-2x}$.

Metrics satisfying the properties~\eqref{eqn:metric-assumptions} above have an explicit local expression as
\begin{equation}\label{eqn:PaperMetric}
    g = \dfrac{1}{q(x,y)^2} \left( \dfrac{1}{A(x)} dx^2 + \dfrac{1}{B(y)} dy^2 + A(x) \, ds^2 + B(y) \, dt^2 \right).
\end{equation}
We refer to the restriction of~\eqref{eqn:PaperMetric} to the $(x,y)$-plane as the base metric, which is $\frac{1}{q(x,y)^2}\left(\frac{dx^2}{A(x)} + \frac{dy^2}{B(y)}\right)$ over a two-dimensional base $\Omega$.
If $A = B = 1$, then the base metric is simply $q^{-2} g_{\mr{Euc}}$.
In addition to the base surface with coordinates $(x,y)$ for fixed $(s,t)$, we will consider the two surface factors given by coordinates $(x,s)$ or $(y,t)$ while fixing the other pair.
Unlike the base metric, which is invariant under any change in $(s,t)$, these surface metrics vary due to the conformal coupling $q$.
In Section~\ref{subsec:singularityGeometry}, we study the asymptotics and completeness of the solitons by restricting the metric to these surface slices.

A metric satisfying~\eqref{eqn:metric-assumptions} has a maximal open dense chart where $g$ is given in coordinates by~\eqref{eqn:PaperMetric}. 
It is sufficient to understand the soliton behavior of metrics of the form~\eqref{eqn:PaperMetric} and then analyze the geometry at the boundary of the chart.
The metric can close up smoothly or can have an end, singularity, or boundary.
We denote partial derivatives of $q$ with subscripts.
The Ricci tensor of $g$ is diagonal, except for $R_{xy}$, with entries
\allowdisplaybreaks\begin{align}
R_{xx} &= -\dfrac{1}{A} \left( \dfrac{A''}{2} - \dfrac{2 \, q_x A' + 3 \, q_{xx}A + q_yB' + q_{yy}B}{q} + 3 \dfrac{q_x^2 A + q_y^2 B}{q^2} \right), \label{eqn:Rxx} \\
R_{xy} &= 2 \dfrac{q_{xy}}{q}
\label{eqn:Rxy} \\
R_{ss} &= 
- A \left( \dfrac{A''}{2} - \dfrac{2 \, q_xA' + q_{xx}A + q_y B' + q_{yy}B}{q} + 3\dfrac{q_x^2 A+ q_y^2 B}{q^2} \right), \label{eqn:Rss}
\end{align}
and $R_{yy}, R_{tt}$ obtained from $R_{xx}, R_{ss}$, respectively, using the symmetries $x \leftrightsquigarrow y$ and $A \leftrightsquigarrow B$.
\begin{example}\label{ex:Einstein}
    Einstein metrics of the form~\eqref{eqn:PaperMetric} include, up to rescaling and translations of $x,y$:
    \begin{enumerate}[(i)]
        \item Product metrics $g_{\Sigma_1} \oplus g_{\Sigma_2}$: when $q=1$ and $\Sigma_1, \Sigma_2$ are Einstein surfaces with the same Einstein constant, i.e., space forms with the same curvature.
        This corresponds to $A,B$ being quadratic polynomials with the same leading coefficient.
        \item Schwarzschild metric: when $q=x$ (or $q=y$), $A$ is a cubic polynomial, and $B$ is a quadratic polynomial.
        \item \pbnski metric, cf.~\cite{pbnski}: when $q = x \pm y$ and $A,B$ are cubic polynomials of the form
        \begin{equation}\label{metric:pbnski}
        \tag{S0}
            g = \frac{1}{(c_1 x + c_2 y)^2} \left( \frac{dx^2}{a_0 - P(x)} + (a_0 - P(x)) \, ds^2 + \frac{dy^2}{b_0 + P(y)} + (b_0 + P(y)) \, dt^2 \right), \quad P \; \text{cubic}.
        \end{equation}
        The Einstein constant is $\lambda = - 3(a_0 c_1^2 + b_0 c_2^2)$; for $c_1 c_2 =0$, this is a Schwarzschild metric (ii).
    \end{enumerate}
\end{example}
Related constructions of families of Einstein metrics modeled on $\bR^4$, particularly involving the Pleba\'nski - Demina\'nski metric, are studied in \cites{alvarado-fams, plebanski-fam}.
\begin{claim}\label{claim:Einstein}
The only Einstein metrics $g$ in our class of study are the ones in example~\ref{ex:Einstein}.
\end{claim}
\begin{proof}
Since the metric $g$ of~\eqref{eqn:PaperMetric} is diagonal, if it is Einstein, the Ricci tensor must be as well. 
This requires $R_{xy} = 0$, meaning $q_{xy} = 0$, so $q = f_1 (x) + f_2(y)$.
The Einstein equation forces $(g^{-1} \on{Ric})^x{}_x = (g^{-1} \on{Ric})^s{}_s$, which implies $q_{xx} = 0$.
The same argument shows that $q_{yy} = 0$.
Therefore, $q$ must be affine, and up to translating and rescaling, one of $1, x,y$, or $x+y$.
$A$ and $B$ are computed directly in all these cases.
\end{proof}
It follows that any soliton metric~\eqref{eqn:PaperMetric} with $q$ not of the form $ax+by+c$ will not be Einstein.
\begin{claim}\label{claim:appropriate-V}
If a metric $g$ as above is a Ricci soliton $\on{Ric} + \frac{1}{2} \cL_V g = \lambda g$, there exists some $\tilde{V}$ with local expression $\tilde{V} = \tilde{V}^x(x,y) \partial_x + \tilde{V}^y(x,y) \partial_y$ that satisfies the soliton equation.
\end{claim}
\begin{proof}
The conditions~\eqref{eqn:metric-assumptions} on the metric imply the existence of two orthogonal, commuting Killing vector fields $\xi_1, \xi_2$ that encode the cohomogeneity two condition and the toric structure of $g$.
Writing the metric as locally conformal to the product $\tilde{g} = g_{\Sigma_1} \oplus g_{\Sigma_2}$ of two warped product surfaces by $g = e^{2\varphi} \tilde{g}$ for $\varphi(x,y)$ a conformal change respecting the toric structure; in terms of the expression~\eqref{eqn:PaperMetric} for the metric, we have $\varphi(x,y) = -\log q(x,y)$.
This means that $\xi_1, \xi_2$ are Killing for $\tilde{g}$, as well.
For any vector field $V$,
\[
\cL_V g = \cL_V(e^{2\varphi} \tilde{g}) = e^{2\varphi} \left( 2V(r) \tilde{g} + \cL_V \tilde{g} \right)
\]
whereby $\cL_{\xi_i} g = \cL_{\xi_i} \tilde{g}=0$ implies that $\xi_i(\varphi) = 0$.
The Ricci tensors $\on{Ric}(g), \on{Ric}(\tilde{g})$ are related by
\begin{equation}\label{eqn:conformal-ricci}
\on{Ric}(g) = \on{Ric}(\tilde{g}) - 2 \left( \on{Hess}_{\tilde{g}} \varphi - d\varphi \otimes d\varphi \right) - \left( \Delta_{\tilde{g}} \varphi + 2 \, |d\varphi|_{\tilde{g}}^2 \right) \tilde{g}
\end{equation}
in dimension four, due to $g = e^{2\varphi} \tilde{g}$.
The Ricci soliton equation for $(g,V)$ in terms of $(\tilde{g},\varphi)$ becomes
\begin{equation}\label{eqn:soliton-expanded}
    V(\varphi) \tilde{g} + \frac{1}{2} \cL_V \tilde{g} = - e^{-2\varphi} \on{Ric}(\tilde{g}) + 2 e^{-2\varphi} \left( \on{Hess}_{\tilde{g}} \varphi - d\varphi \otimes d\varphi \right) + e^{-2\varphi}\left( \Delta \varphi + 2 \, |d\varphi|^2 + \lambda e^{2\varphi} \right) \tilde{g}.
\end{equation}

We decompose the vector field $V$ as $V = V_0 + a_1 \xi_1 + a_2 \xi_2$, where $a_i := g(V,\xi_i)$ and $g(V_0,\xi_i) = 0$, so $V_0$ is orthogonal to the Killing fields $\xi_i$.
We will produce from $V$ a new vector field $\tilde{V}$ satisfying the soliton equation~\eqref{eqn:soliton-expanded} and being orthogonal to the $\xi_i$, so $g(\tilde{V},\xi_i) = 0$.
The first property means
\begin{equation}\label{eqn:ODE-for-V}
    2 \tilde{V}(\varphi) \tilde{g} + \cL_{\tilde{V}} \tilde{g} = 2 V_0 (\varphi) \tilde{g} + \cL_V \tilde{g}, 
\end{equation}
where $V_0(\varphi) = V(\varphi)$ due to $\xi_i(\varphi) = 0$.
Equation~\eqref{eqn:ODE-for-V} is a first-order ODE in (the components of) $\tilde{V}$, so a local solution exists for given initial data, by the Picard-Lindel\"of theorem. 
We need only ensure that given initial data requiring $g(\tilde{V},\xi_i) = 0$, the (unique) resulting solution to~\eqref{eqn:ODE-for-V} preserves this property.

Since Killing vector fields are infinitesimal generators of isometries, which preserve the Ricci tensor, we also have $\cL_{\xi_i} \on{Ric}(\tilde{g}) = 0$.
Therefore, since $\xi_i(\varphi) = 0$, we can apply~\eqref{eqn:ODE-for-V}, showing
\[
2 \, \xi_i (V_0(\varphi)) \tilde{g} + \cL_{[\xi_i,V]} \tilde{g} = 0, \qquad i=1,2.
\]
In particular, applying $\cL_{\xi_j}$ to this shows that the vector fields $[\xi_i,[\xi_j,V]]$ are Killing for $\tilde{g}$.
Since $[\xi_i,V] = [\xi_i,V_0] + \xi_i(a_1)\xi_1 + \xi_i(a_2) \xi_2$, the symmetric tensors $\cT_i := \cL_{[\xi_i,V]} \tilde{g} - \cL_{[\xi_i,V_0]} \tilde{g}$ have
\[
\cT_i(X_1,X_2) = \tilde{g} \left(X_1, X_2(\xi_i(a_1)) \xi_1 + X_2(\xi_i(a_2)) \xi_2 \right) + \tilde{g} \left(X_1(\xi_i(a_1)) \xi_1 + X_1(\xi_i(a_2)) \xi_2, X_2 \right)
\]
so they vanish on $\{ \xi_1, \xi_2 \}^{\perp}$.
We only required that $[\xi_i,V] - [\xi_i,V_0] \in \on{span} \{ \xi_1, \xi_2\}$, which is also the case for $V - V_0$; therefore, the tensor $\cT := \cL_V \tilde{g} - \cL_{V_0} \tilde{g}$ has the same property.

For $\tilde{g}$ the product metric, all tensors on the right-hand side of equation~\eqref{eqn:soliton-expanded} are diagonal except possibly in $\{ \xi_1, \xi_2 \}^{\perp}$, so $\cL_V \tilde{g}$ is, too. 
Thus, there is a solution $\tilde{V}$ to~\eqref{eqn:ODE-for-V} with initial data $g(\tilde{V},\xi_i) = 0$ that preserves this property for short time, as desired.
Using the local conformal presentation of $g$, we can cover the manifold by local patches on which this result holds. 
When $g(\tilde{V},\xi_i) = 0$, we choose local coordinates $(x,y,s,t)$ for which $\xi_1, \xi_2$ are $\ps, \pt$.
We obtain a metric $g$ of local presentation~\eqref{eqn:PaperMetric}, so $\tilde{V} = \tilde{V}^x \px + \tilde{V}^y \py$.
The $(xs)$ component of the soliton equation gives $\ps \tilde{V}^x = 0$; likewise, $\tilde{V}^x, \tilde{V}^y$ have no $s$- or $t$-dependence.
The equation is thus satisfied locally for some $\tilde{V} = \tilde{V}^x(x,y) \px + \tilde{V}^y (x,y) \py$, as claimed.
\end{proof}
Without loss of generality, we may now assume that $V = V^x(x,y)\px + V^y(x,y)\py$ is of the simplified form.
For such a $V$, the tensor $\cL_V g$ becomes diagonal, except for $(\cL_V g)_{xy}$, with entries 
\allowdisplaybreaks\begin{align}
(\cL_V g)_{xx} &= \dfrac{2 (\partial_x V^x) A - V^x A'}{q^2 A^2} - \dfrac{2}{A}\dfrac{V^x q_x + V^y q_y}{q^3} \label{eqn:Lvg-xx}, \\
(\cL_V g)_{xy} &= \dfrac{(\partial_x V^y)A + (\partial_y V^x)B}{q^2 AB} \label{eqn:Lvg-xy}, \\
(\cL_V g)_{ss} &= \dfrac{V^x A'}{q^2} - 2A\dfrac{V^x q_x + V^y q_y}{q^3} \label{eqn:Lvg-ss},
\end{align}
and $(\cL_V)_{yy}, (\cL_V g)_{tt}$ obtained similarly using the symmetries $x \leftrightsquigarrow y, A \leftrightsquigarrow B$.

\subsection{Reduction via auxiliary functions}

We introduce the auxiliary functions
\begin{equation}\label{eqn:SxSy}
    S^x := 2 q q_x + \dfrac{V^x}{A}, \qquad S^y := 2 q  q_y + \dfrac{V^y}{B}.
\end{equation}
The $(xy)$-component of the soliton equation $\on{Ric} + \frac{1}{2} \cL_V g = \lambda g$ is equivalent to the relation
\begin{equation}\label{eqn:pySxpxSy}
    \partial_y S^x + \partial_x S^y = 4 q_xq_y.
\end{equation}
Rewriting the Ricci soliton equation in the form
\begin{equation}\label{eqn:g-inverse-Ricci-soliton}\tag{$\ddagger$}
g^{-1} \on{Ric} + \frac{1}{2} g^{-1} \cL_V g = \lambda,
\end{equation}
where the metric $g$ is diagonal, requires the equalities
\allowdisplaybreaks\begin{align}
    &(g^{-1} \on{Ric})^x{}_x + \frac{1}{2} \left( g^{-1} \cL_V g \right)^x{}_x = (g^{-1} \on{Ric})^s{}_s + \frac{1}{2} \left( g^{-1} \cL_V g \right)^s{}_s, &\iff \quad  \px S^x = 2 q_x^2, \label{eqn:pxSx} \\
    &(g^{-1} \on{Ric})^y{}_y + \frac{1}{2} \left( g^{-1} \cL_V g \right)^y{}_y = (g^{-1} \on{Ric})^t{}_t + \frac{1}{2} \left( g^{-1} \cL_V g \right)^t{}_t, &\iff \quad  \py S^y = 2 q_y^2. \label{eqn:pySy}
\end{align}
Equating the $(ss)$ and $(tt)$ components of equation~\eqref{eqn:g-inverse-Ricci-soliton} is equivalent to
\begin{equation}\label{eqn:A''B''-PDE}
    A'' - \dfrac{S^x}{q^2} A' = B'' - \dfrac{S^y}{q^2} B' =: w(x,y).
\end{equation}
Provided that the above equations hold, the Ricci soliton equation is reduced to the one for the $(xx)$-component; in terms of $S^x, S^y$, this is
\begin{equation}\label{eqn:RicciSoliton-Last-Sx}
    - \dfrac{\lambda}{q^2} = \dfrac{1}{2} w(x,y) - \dfrac{q_xA' + q_{xx}A + q_y B' + q_{yy}B}{q} + \dfrac{q_x^2 A+ q_y^2 B}{q^2} + \dfrac{S^x q_x A + S^y q_y B}{q^3}.
\end{equation}
The soliton equation therefore reduces to solving equation~\eqref{eqn:RicciSoliton-Last-Sx}.
This equivalent representation of the soliton PDE is summarized in the following corollary.
\begin{corollary}\label{corollary:exist-Sx-Sy}
    A metric $g$ of the form~\eqref{eqn:PaperMetric} is a Ricci soliton if and only if there exist functions $S^x, S^y$ such that the system of equations~\eqref{eqn:pySxpxSy} --~\eqref{eqn:RicciSoliton-Last-Sx} is satisfied.
\end{corollary}
Differentiating~\eqref{eqn:pySxpxSy} with $\py$ (or $\px$) and substituting $\py \px S^y$ from~\eqref{eqn:pySy} (likewise for $S^x$) implies
\begin{equation}\label{eqn:p2ySx}
\partial_i \partial_j S^k = 4 \, q_k q_{ij}, \qquad \text{for } \; i,j,k \in \{ x,y \}.
\end{equation}
For example, $\py^2 S^x = 4q_x q_{yy}$.
Differentiating~\eqref{eqn:p2ySx} and using~\eqref{eqn:pxSx} (resp.~\eqref{eqn:pySy}) shows that
\begin{equation}\label{eqn:HessianCompatibility}
    q_{xx}q_{yy} = (q_{xy})^2.
\end{equation}
Therefore, the conformal factor $q$ for a Ricci soliton metric~\eqref{eqn:PaperMetric} has everywhere singular Hessian.
Let $D^2q$ denote the Euclidean Hessian of $q$ whose entries are the ordinary second derivatives $q_{xx}, q_{xy}, q_{yy}$ of $q$. 
We say that $D^2q$ has zeroes if any second-order derivative vanishes.
Otherwise, we write $D^2 q \neq 0$ if no second-order derivative is zero.

We can consider $S^x, S^y$ as solutions to the system~\eqref{eqn:pySxpxSy} --~\eqref{eqn:pySy} for a given $q$.
Given such a solution, any other solution $(\tilde{S}^x, \tilde{S}^y)$ to this system (for fixed $q$) is expressed as
\begin{equation}\label{eqn:S-S-tilde}
    \tilde{S}^x = S^x + C_0 y + C_1, \qquad \tilde{S}^y = S^y - C_0 x - C_2
\end{equation}
for constants $C_i$. 
By~\eqref{eqn:pxSx} and~\eqref{eqn:p2ySx}, the function $\tilde{S}^x - S^x$ is constant in $x$ and at most linear in $y$, yielding the first equality, similarly for $y$ using~\eqref{eqn:pySxpxSy}.
When a solution to this system exists, the auxiliary functions $S^x, S^y$ will be determined among the admissible representatives $\tilde{S}^x, \tilde{S}^y$ by equations~\eqref{eqn:A''B''-PDE} and~\eqref{eqn:RicciSoliton-Last-Sx}.

\subsection{ODE system for the soliton vector field}

The system of equations~\eqref{eqn:pySxpxSy} --~\eqref{eqn:RicciSoliton-Last-Sx} can be viewed as a coupled system of equations for $q,A,B$, since the functions $S^x, S^y$ are completely determined by $q$, up to a three-dimensional vector space of linear functions described in~\eqref{eqn:S-S-tilde}.
We exploit the coupling of the equation~\eqref{eqn:HessianCompatibility} for $q$ with those for the single-variable functions $A$ and $B$ in~\eqref{eqn:A''B''-PDE} and~\eqref{eqn:RicciSoliton-Last-Sx}.
The equations~\eqref{eqn:A''B''-PDE} and~\eqref{eqn:RicciSoliton-Last-Sx} impose an additional compatibility condition on the functions $q, A, B$ comprising the metric, derived by viewing~\eqref{eqn:A''B''-PDE} --~\eqref{eqn:RicciSoliton-Last-Sx} as a system with unknowns $S^x, S^y$.
We first prove general results describing the solvability of a coupled system of ODEs, which can be applied to the system obtained from $S^x, S^y$ as well as to the classification of warped product metrics of the form~\eqref{eqn:PaperMetric}.

\begin{lemma}\label{lemma:system-alpha-beta-h}
Given two-variable functions $\alpha_i(x,y), \beta_i(x,y)$, consider the following ODEs for $F(x)$:
\begin{align}
        \alpha_2 F' + \alpha_1 F + \alpha_0 &= 0 , \label{eqn:decouple-1} \\
        \beta_3 F'' + \beta_2 F' + \beta_1 F + \beta_0 &= 0 \label{eqn:decouple-2}.
    \end{align}
\begin{enumerate}[(i)]
    \item There exists a function $F(x)$ satisfying~\eqref{eqn:decouple-1} if and only if
    \begin{equation}\label{eqn:A-existence}
        \py \left( \frac{\py (\alpha_0 / \alpha_2)}{\py (\alpha_1 / \alpha_2)} \right) = 0 \qquad \text{and} \qquad \px \left( \frac{\py (\alpha_0 / \alpha_2)}{\py (\alpha_1 / \alpha_2)} \right) + \frac{\alpha_1}{\alpha_2} \frac{\py (\alpha_0 / \alpha_2)}{\py (\alpha_1 / \alpha_2)} = \frac{\alpha_0}{\alpha_2},
    \end{equation}
    in which case $F(x)$ is given by
    \begin{equation}\label{eqn:F(x)-solution-ODE}
        F(x) = - \frac{\py (\alpha_0 / \alpha_2)}{\py (\alpha_1 / \alpha_2)}
    \end{equation}
    provided that this is non-degenerate.
     \item The ODE system~\eqref{eqn:decouple-1} and~\eqref{eqn:decouple-2} admits a solution $F(x)$ if and only if in addition to condition~\eqref{eqn:A-existence}, the solution to $F(x)$ in~\eqref{eqn:F(x)-solution-ODE} agrees with
      \begin{equation}\label{eqn:couple-decouple-single-variable}
    F(x) = - \frac{\alpha_0 \alpha_1 \beta_3 + \alpha_0 (\px \alpha_2) \beta_3 - (\px \alpha_0) \alpha_2 \beta_3 - \alpha_0 \alpha_2 \beta_2 + \alpha_2^2 \beta_0}{\alpha_1^2 \beta_3 + \alpha_1 (\px \alpha_2) \beta_3 - (\px \alpha_1) \alpha_2 \beta_3 - \alpha_1 \alpha_2 \beta_2 + \alpha_2^2 \beta_1}.
    \end{equation}
\end{enumerate}
\end{lemma}
\begin{proof}
Part (i) comes from applying $\py$ to
\eqref{eqn:decouple-1} and solving the resulting $2 \times 2$ system with unknowns $F,F'$:
\[
\begin{cases}
    0&= \alpha_2 F' + \alpha_1 F + \alpha_0 \\
    0&= (\py \alpha_2) F' + (\py \alpha_1) F +\py \alpha_0 
\end{cases}
\implies 
    \begin{cases}
    F(x) &= - \frac{\py(\alpha_0/\alpha_2)}{\py(\alpha_1/\alpha_2)} \\ 
    F'(x) &= - \frac{\py(\alpha_0/\alpha_1)}{\py(\alpha_2/\alpha_1)} \\ 
    \end{cases}.
\]
We see that the conditions~\eqref{eqn:A-existence} are equivalent to $F$ being a function of only $x$, and its derivative satisfying its equation from solving the system.
Furthermore, a solution $F(x)$ to the ODE~\eqref{eqn:decouple-1} exists even if the expression~\eqref{eqn:F(x)-solution-ODE} is degenerate from these formulas.

Part (ii) comes from differentiating the expression of $F'$ from part (i) and plugging the result into~\eqref{eqn:decouple-2}.
This yields an equation of the form $\gamma_1 F + \gamma_0 = 0$ whose solution gives \eqref{eqn:couple-decouple-single-variable}, which must be compatible with part (i).
Conversely, the equivalence of part (i) guarantees that the expression $F(x)$ indeed produces a single-variable function satisfying~\eqref{eqn:decouple-1}; by construction, it also satisfies~\eqref{eqn:decouple-2}.
\end{proof}
\begin{corollary}\label{cor:FGsystem}
    Consider two-variable functions $u(x,y), v(x,y)$.
    The existence of non-zero functions $F(x), G(y)$ satisfying the equality
    \begin{equation}\label{eqn:A-B-u-v-system}
        F'(x) - u F(x) = G'(y) - v G(y)
    \end{equation}
    is equivalent to a triple of fourth-order differential equations in $u,v$.
    For each expression, there is a fixed $N \in \bN^*$ such that all coefficients involve summands of the form $\Pi_{r=1}^k \partial_x^{p_r} \partial^{q_r}_y u \cdot \Pi_{s=1}^{\ell} \partial_x^{p'_s} \partial_y^{q'_s} v$, with $\sum_r (p_r + q_r) + \sum_s (p'_s + q'_s) + (k+\ell) = N$.
\end{corollary}
\begin{proof}
The pair of equations
\[
(a): \quad \px \left( \frac{\px (F' - u F)}{\px v} \right) = 0, \qquad (b): \quad - \py \left( \frac{\px (F' - u F)}{\px v} \right) = F' - u F - v \frac{\px (F' - u F)}{\px v}
\]
is equivalent to the existence of functions $G(y), \tilde{G}(y)$ such that
\[
(a): \quad F'(x) - u F(x) = \tilde{G}(y) - v G(y), \qquad (b): \quad \tilde{G}(y) =  G'(y), \quad \iff~\eqref{eqn:A-B-u-v-system},
\]
with $G = - \frac{\px (G' - u G)}{\px v}$.
Equation~\eqref{eqn:A-B-u-v-system} is therefore equivalent to the solvability (in $F$) of a system
\begin{equation}\label{eqn:solvability-new-system}
(a): \quad \beta_3 F^{(3)} + \beta_2 F'' + \beta_1 F' + \beta_0 F = 0, \qquad (b): \quad \alpha_2 F'' + \alpha_1 F' + \alpha_0 F = 0.
\end{equation}
For each $j$, the coefficients $\alpha_j, \beta_j$ of~\eqref{eqn:solvability-new-system} are sums of terms of the form $\Pi_{r=1}^k \partial_x^{p_r} \partial^{q_r}_y u \cdot \Pi_{s=1}^{\ell} \partial_x^{p'_s} \partial_y^{q'_s} v$, which satisfy the equation $\sum_r (p_r + q_r) + \sum_s (p'_s + q'_s) + (k+\ell) = 5-j$.
This means that the sum of the total number of factors and the total number of derivatives in each of the products comprising the coefficients $\alpha_j, \beta_j$ of~\eqref{eqn:solvability-new-system} is fixed and equal to $5-j$.
For instance, $\alpha_2 = \px \py v + v (\px v)$ and $\beta_3 = \px v$.
An analogous argument as in Lemma~\ref{lemma:system-alpha-beta-h} shows that the solvability of~\eqref{eqn:solvability-new-system} is equivalent to three second-order equations in $\alpha_i, \beta_j$ with terms $\Pi_{r=1}^k \partial_x^{p_r} \partial_y^{q_r} \alpha_{i_r} \cdot \Pi_{s=1}^{\ell} \partial_x^{p'_s} \partial_y^{q'_s} \beta_{j_s}$ such that in each expression, $k+\ell$ (the total number of factors) and $\sum_r (p_r + q_r) + \sum_s (p'_s + q'_s)$ (the total number of derivatives) are fixed.
\end{proof}

Following Corollary~\ref{corollary:exist-Sx-Sy}, we view the soliton equations as a system for $S^x, S^y$ via equations~\eqref{eqn:A''B''-PDE} and~\eqref{eqn:RicciSoliton-Last-Sx}. 
We rewrite the latter for brevity as
\begin{align}
     A q_x S^x &+ \left( Bq_y - \frac{1}{2} B' q \right) S^y = q \left( \gamma - \frac{1}{2} B'' q^2 \right), \qquad \text{where} \label{eqn:Last-Sx-Sy-gamma} \\
     \gamma(q,A,B, \lambda) &:= q \left( q_xA' + q_{xx}A + q_y B' + q_{yy}B \right) - \left( q_x^2 A+ q_y^2 B \right) - \lambda. \label{eqn:gamma-A,B,q}
\end{align}
In what follows, the dependence of $\gamma$ on $q,A,B,\lambda$ will be suppressed.

The solution $(S^x, S^y)$ of this $2 \times 2$ system can be computed as
\begin{equation}\label{eqn:Sx-system-solution-general}
S^x = q \frac{ \left( Bq_y - \frac{1}{2} B' q \right) (A'' - B'') q + B' \left( \gamma - \frac{1}{2} B'' q^2 \right)}{AB' q_x + A'B q_y - \frac{1}{2} A' B' q}, \qquad S^y = q \frac{A' \left( \gamma - \frac{1}{2} B'' q^2 \right) - A (A'' - B'') q q_x}{AB' q_x + A'B q_y - \frac{1}{2} A' B' q}.
\end{equation}
We denote by $D$ the denominator $D := AB' q_x + A'B q_y - \frac{1}{2} A' B' q$.
The functions $S^x, S^y$ are determined uniquely unless $D = 0$.
By Corollary~\ref{corollary:exist-Sx-Sy}, the metric $g$ is a soliton if and only if equations~\eqref{eqn:pySxpxSy} --~\eqref{eqn:pySy} are satisfied.
Imposing $\px S^x = 2 q_x^2$ and $\py S^y = 2 q_y^2$, we reduce~\eqref{eqn:pySxpxSy} to a second-order condition in $q$ by computing the terms of
\begin{align*}
    \frac{2}{A' B'} (A' q_x - B' q_y)^2 &= 2 \frac{A'}{B'} q_x^2 + 2 \frac{B'}{A'} q_y^2 - 4 q_x q_y = \frac{A'}{B'} \px S^x + \frac{B'}{A'} \py S^y - (\py S^x + \px S^y).
\end{align*}
Computing the right-hand side using expressions~\eqref{eqn:Sx-system-solution-general} and multiplying by $\frac{D}{q}$, we arrive at
\allowdisplaybreaks{
\begin{equation}\label{eqn:second-compatibility-equation}
\begin{split}
    &\left( \frac{A^{(3)}}{B'} + \frac{B^{(3)}}{A'} \right) q D + \frac{1}{2} q^2 \left( (A'')^2 + (B'')^2 \right) - (A'' + B'') ( q q_{xx} A + q q_{yy} B - \lambda) \\
    & + (3 A A'' - A B'') q_x^2 + (3 B B'' - A'' B) q_y^2 + 2 (A'' - B'') \left( \frac{B}{B'} A' -  \frac{A}{A'} B'\right) q_x q_y - 2 q (A' A'' q_x + B' B'' q_y) \\
    &= \frac{2}{A' B'} \left( A' q_x - B' q_y \right)^2 \left( A B' \frac{q_x}{q} + A' B \frac{q_y}{q} - \frac{1}{2} A' B' \right).
\end{split}
\end{equation}}
Together with the conditions $\px S^x = 2 q_x^2$ and $\py S^y = 2 q_y^2$ computed using~\eqref{eqn:Sx-system-solution-general}, the relation~\eqref{eqn:second-compatibility-equation} gives the system of equations that is equivalent to the Ricci soliton property of the metric~\eqref{eqn:PaperMetric}, by Corollary~\ref{corollary:exist-Sx-Sy}.

\subsection{Proof outline}

The paper is organized around the proof of Theorems~\ref{thm:main-theorem} and~\ref{thm:secondary-thm} as follows.
We will first examine the situation where the Hessian $D^2q$ has zeroes in Section~\ref{subsec:D2qhasZeroes}, where Proposition~\ref{prop:d2q=0} classifies such solitons.
In Section~\ref{subsection:degeneracy-resolvent-system}, we treat the degenerate case of expression~\eqref{eqn:Last-Sx-Sy-gamma} when the denominator $D := A B' q_x + A' B q_y - \frac{1}{2} A' B' q$ vanishes.

Next, Section~\ref{sec:homogeneous} classifies the conformally cylindrical solitons, which correspond to the conformal factor $q$ being a homogeneous function of degree $1$, so the condition $\det D^2q = 0$ is automatically satisfied. 
Our classification proceeds in two steps, depending on whether or not the surface factors $\Sigma_i$ have constant curvature; we deduce that either neither or both factors have constant curvature.
In the former case, the soliton metrics are classified in Proposition~\ref{prop:homogeneous-rigidity}, while in the latter, their Gauss curvatures must be opposite, and the only possible soliton metrics are the conformally flat and scalar flat families from Section~\ref{subsect:solitons-conformal-S2H2}.

Our results show that in many settings, the Ricci solitons produced in Section~\ref{sec:homogeneous} are essentially the only ones of the form~\eqref{eqn:PaperMetric}.
We collect these steps and complete the proofs of Theorems~\ref{thm:main-theorem} and~\ref{thm:secondary-thm} in Section~\ref{section:proofs}.

\section{Rigidity results for cohomogeneity two solitons}\label{section:rigidity-results}

We begin by obtaining various restrictions on the behaviors of a cohomogeneity two Ricci soliton metric satisfying the properties~\eqref{eqn:metric-assumptions}.
For the results that follow, no assumption of completeness is required for the metric $g$; the only properties used are the local analytic solvability of the expression~\eqref{eqn:PaperMetric} for the metric.

\subsection{Degenerate Monge-Amp\`ere equation}

Recall the domain of definition of the metric $g = g(q,A,B)$ is $\Omega \subseteq \bR^2$ given as a maximal connected component of $\left\{ A(x) > 0 \right\} \cap \left\{ B(y) > 0 \right\} \cap \left\{ q(x,y) \neq 0 \right\}$.
The intersection of the first two sets forms the product of two intervals, possibly (half-)infinite; equation~\eqref{eqn:HessianCompatibility} then requires $q$ to solve the degenerate homogeneous Monge-Amp\`ere equation on $\Omega$, which can be viewed as a free boundary-type condition.
By itself, this equation is flexible and allows for various parametric families of solutions that describe developable surfaces, which we obtain as the main result of this subsection.

The connection between the Monge-Amp\`ere equation and the geometry of warped products was previously observed in three dimensions, in the work of G\'alvez, Hauswirth, Jim\'enez, and Mira \cites{galvez-hauswirth-mira, galvez-jimenez-mira}.
In the case of the finitely punctured plane $\Omega$, (i.e., $A, B$ are everywhere positive), the possibilities for $q$ were studied by G\'alvez and Nelli \cite{galvez-nelli}, advancing classical results of Pogorelov and Hartman--Nirenberg.

We first characterize all local solutions to the PDE $\det D^2 q = 0$. 
For this, we start by recalling some standard facts from first-order equations.
\begin{lemma}\label{lemma:level-set-composition}
   Let $u,v$ be real-valued functions defined on a domain $\Omega \subseteq \bR^2$.
    Suppose that
    \begin{equation}\label{eqn:u-v-gradient-condition}
        \frac{\py u}{\px u} = \frac{\py v}{\px v}.
    \end{equation}
    If $\nabla u \neq (0,0)$ in a punctured neighborhood of $p$, then there exists some $B_r(p) \subseteq \Omega$ and a single-variable function $\theta$ such that $v = \theta \circ u$ on $B_r(p)$.
\end{lemma}
This happens because if the functions have everywhere parallel gradients, their level sets agree locally.
\begin{corollary}\label{corollary:apxr+bpyr=c}
Suppose that the function $u(x,y)$ satisfies 
\begin{equation}\label{eqn:r-a-b-c}
    a(x) \, \px u + b(y) \, \py u = c(x)
\end{equation}
for some single-variable functions $a,b \neq 0$ and $c$. 
Then, $u$ is given by
\begin{equation}\label{eqn:r-solution-a-b-c}
   u = \theta \left( f(x) - g(y) \right) + \int \frac{c(x)}{a(x)} \, dx, \qquad \text{where } \; f(x) := \int \frac{1}{a} \, dx, \quad g(y) := \int \frac{1}{b} \, dy, 
\end{equation}
for some single-variable function $\theta$. 
\end{corollary}
Indeed, the function $\tilde{u} = u - \int \frac{c}{a} \, dx$ satisfies $a \, \px \tilde{u} + b \, \py \tilde{u} = 0$, so Lemma~\ref{lemma:level-set-composition} can be applied.
    \begin{theorem}\label{thm:solutions-to-monge-ampere}
        Let $q(x,y)$ satisfy the homogeneous Monge-Amp\`ere equation $\det D^2 q = 0$ on an open set $\Omega \subseteq \bR^2$.
        For any point $p \in \Omega$, there exists a ball $B_r(p) \subseteq \Omega$ such that $q$ has one of the following forms:
        \begin{enumerate}[(i)]
            \item $q(x,y) = cx + h(y)$ or $q(x,y) = h(x) + cy$, for a single-variable function $h$ and a constant $c$;
            \item $q(x,y) = \theta(x+y) + cx$, for a single-variable function $\theta$ and a constant $c$, possibly after rescaling $x,y$; or
            \item $q(x,y) = x \Theta(v) - ( \int \psi' \Theta' )(v)$,
            for single-variable functions $\psi, \Theta$ and a function $v(x,y)$ defined by
            \begin{equation}\label{eqn:q-defined-implicitly}
                \psi (v) = y + x v, \qquad \text{and} \qquad v = \frac{q_{xy}}{q_{yy}} = \frac{q_{xx}}{q_{xy}}.
            \end{equation}
        \end{enumerate}
    \end{theorem}
    \begin{proof}
    If some second-order derivative $\partial^2_{ij} q$ of $q$ vanishes, then $\det D^2q = 0$ implies that at least two such derivatives vanish.
    This means that either $q(x,y) = h(x) + cy$, or $q(x,y) = cx + h(y)$, for some function $h$ and constant $c$.
    Conversely, any such function satisfies $\det D^2 q = 0$, so in what follows we assume that $D^2 q \neq 0$.
    Defining the function $v = \frac{q_{xy}}{q_{yy}} = \frac{q_{xx}}{q_{xy}}$, we differentiate the equality $q_{xy}^2 = q_{xx} q_{yy}$ to obtain
        \begin{equation}\label{eqn:qxy-qyy-v-transport}
        \px \left( \frac{q_{xy}}{q_{yy}} \right) = \frac{q_{xy}}{q_{yy}} \cdot \py \left( \frac{q_{xy}}{q_{yy}} \right), \qquad \text{so} \qquad v_x = v v_y.
        \end{equation}
        The relation~\eqref{eqn:qxy-qyy-v-transport} implies that the function $v$ satisfies the inviscid Burgers' equation $v_x - v v_y = 0$.
        The method of characteristics shows that $v$ is either constant or is implicitly defined by
        \begin{equation}\label{eqn:v-implicit-transport}
        \psi(v) = y + x v
        \end{equation}
        for a single-variable function $\psi$.
        We first consider the situation when $q_x = v q_y$, meaning that $\frac{q_{xy}}{q_x} = \frac{q_{yy}}{q_y}$.
        In this case, we could write $\frac{q_x}{q_y} = a(x)$ for some single-variable function, whereby Lemma~\ref{lemma:level-set-composition} would imply that $q = \theta (\int a(x) \, dx + y)$, for some single-variable function $\theta$.
        Then, $q_x = \theta' a$ and $q_y = \theta'$, with $q_{yy} = \theta''$ so the Monge-Amp\`ere condition $q^2_{xy} = q_{xx} q_{yy}$ becomes $\theta' \theta'' a' = 0$.
        Since $D^2 q \neq 0$, we have $\theta' \theta'' \neq 0$, which forces $a = C$, meaning that $q(x,y) = \theta(Cx+y)$.
        
        In all other cases, we have $D^2 q \neq 0$ and $q_x - vq_y \neq 0$, so we can produce a well-defined function
        \begin{equation}\label{eqn:homo-first-step}
        u(x,y) := \dfrac{q q_{yy}}{q_x q_{yy} - q_yq_{xy}}, \qquad \text{hence } \qquad q = u \, q_x - uv \, q_y.
        \end{equation}
        Differentiating~\eqref{eqn:homo-first-step} and using $q_{xy}^2 = q_{xx} q_{yy}$, we obtain
        \begin{equation}\label{eqn:differentiated-eqns}
            q_x = u_x q_x - \px (uv) q_y \qquad \text{and} \qquad
            q_y = u_y q_x - \py (uv) q_y.
        \end{equation}
        Therefore, multiplying the second equation of~\eqref{eqn:differentiated-eqns} by $v$ and subtracting the resulting relations gives
        \[
        (u_x - v u_y - 1) (q_x - v q_y) = 0.
        \]
        Our previous discussion shows that $q_x - v q_y \neq 0$ due to $D^2 q \neq 0$, so either $u_x - 1 = u_y = 0$, or we can write
        \begin{equation}\label{eqn:qxy-qyy-ux-uy}
            \frac{\px (u-x)}{\py (u-x)} = \frac{u_x - 1}{u_y} = v(x,y) = \frac{q_{xy}}{q_{yy}} = \frac{v_x}{v_y}.
        \end{equation}
        If $v = c$ is constant, then $\frac{q_{xy}}{q_{yy}} = c \neq 0$ since $D^2 q \neq 0$, so Lemma~\ref{lemma:level-set-composition} implies that $q_y = \theta_2 (cx + y)$, for some single-variable function $\theta_2$.
        Applying Lemma \ref{lemma:level-set-composition} to the equality $\frac{q_{yy}}{q_{xy}} = \frac{q_{xy}}{q_{xx}}$ enables us to write $q_x = \phi(q_y)$, so $q_x = \theta_1 (cx+y)$ for a single-variable function $\theta_1$.
        Since $q_{xy} = \py \theta_1 = \px \theta_2$, we have $\frac{\theta'_1}{\theta'_2} = c$, so $\theta_1(t) = c \theta_2(t) + c_0$ for some constant $c_0$.
        Since $\frac{\px (u - x)}{\py (u-x)} = \frac{u_x - 1}{u_y} = c$, we also obtain $u(x,y) = x + \theta(cx + y)$ for a single-variable function $\theta$.
        Plugging these equalities into~\eqref{eqn:homo-first-step} yields
        \[
            q(x,y) = u(x,y) (q_x - v q_y) = c_0 \left( x + \theta(cx + y) \right),
        \]
        whereby $q$ has the form described in case (ii).
        For $v$ a non-constant function satisfying the equality $\frac{\px (u-x)}{\py(u-x)} = \frac{v_x}{v_y}$ from~\eqref{eqn:v-implicit-transport}, Lemma~\ref{lemma:level-set-composition} implies that $u=x - \theta(v)$, for a single-variable function $\theta$.
        Also, $\frac{q_{xy}}{q_{yy}} = v = \frac{v_x}{v_y}$ implies that $q_x = \theta_1(v)$ and $q_y = - \Theta'(v)$, for single-variable functions $\theta_1, \Theta$.
        Now, $q_{xy} = \py q_x = \px q_y$ gives $\frac{\theta'_1(v)}{- \Theta''(v)} = v$, so $\theta'_1(t) = - t \Theta''(t) = (\Theta - t \Theta')'$ and $\theta_1(t) = \Theta - t \Theta'$.
        Using these relations in~\eqref{eqn:homo-first-step}, we obtain
        \begin{equation}\label{eqn:q-plug-in}
            q = (x - \theta(v)) \Theta(v).
        \end{equation}
        Computing $q_y = - \Theta'(v)$ from~\eqref{eqn:q-plug-in} and writing $v_y = \frac{1}{\psi'(v) - x}$ from~\eqref{eqn:v-implicit-transport}, we obtain $\psi' \Theta' =(\theta \Theta)'$.
        Upon integration, $\theta \Theta = \int \psi' \Theta'$ and $q = x \Theta(v) - (\int \psi' \Theta') (v)$, so $q$ has the form described in case (iii) with $v = \frac{q_{xy}}{q_{yy}}$.

        It remains to consider the case when~\eqref{eqn:qxy-qyy-ux-uy} is not well-defined, meaning that $u_x - 1 = u_y =0$.
        We deduce that $u(x,y) = x  - x_0$ and equation~\eqref{eqn:differentiated-eqns} implies that $\py \left( \frac{q_{xy}}{q_{yy}} \right) = - \frac{1}{x-x_0}$.
        This means that $\frac{q_{xy}}{q_{yy}} = - \frac{a(x) + y}{x-x_0}$ for a single-variable function $a(x)$, and using this result in~\eqref{eqn:homo-first-step} gives
        \[
        q = (x-x_0) q_x + (a(x) + y) q_y
        \]
        so $(x- x_0) q_{xx} + a'(x) q_y + (a(x) + y) q_{xy} = 0$.
        Since $\frac{q_{xx}}{q_{xy}} = \frac{q_{xy}}{q_{yy}} = - \frac{a(x) + y}{x-x_0}$, we also have $(x - x_0) q_{xx} + (a(x) + y) q_{xy} = 0$, whereby combining the two equations gives $a' = 0$.
        This means that $q = (x - x_0) q_x + (y - y_0) q_y$ for some $x_0, y_0$, so it is homogeneous after translating $x,y$.
        Note that this class of functions belongs in case (iii): when $\Phi = x_0 \Theta$, we have $\Phi'/\Theta' = x_0$, so $\int \Phi'/\Theta' = x_0 t + y_0$ for some $y_0$.
        Thus, equation~\eqref{eqn:q-defined-implicitly} becomes $v = - \frac{y-y_0}{x-x_0}$, making the expression for $q$ into $q(x,y) = x \Theta(v) - \Phi(v) = (x-x_0) \Theta ( - \frac{y-y_0}{x-x_0} )$.
        This is precisely the homogeneity condition after translation.
    \end{proof}

    Using the framework of case (iii) for solutions of the Monge-Amp\`ere equation, we can write
    \begin{align}
        q_x &= \Theta(v) - v \Theta'(v), \qquad q_y = - \Theta'(v) \label{eqn:qx,qy-from-MA-1} \\
        q_{xx} &= -\Theta''(v) v^2 v_y, \qquad q_{xy} = - \Theta''(v) v v_y, \qquad q_{yy} = - \Theta''(v) v_y. \label{eqn:qx,qy-from-MA-2}
    \end{align}
    By~\eqref{eqn:qx,qy-from-MA-1}, we have $q(x,y) = C(x-x_0)$ (resp.~$q(x,y) = C(y-y_0)$) if and only if $\Theta$ is constant (resp.~linear).
    Case (iii) includes the homogeneous $q$, motivating the classification of conformally cylindrical solitons in Section~\ref{sec:homogeneous}, but is too broad of a condition to fully classify as currently understood.
\subsection{Zeroes in \texorpdfstring{$D^2 q$}{Hess q} and warped products}\label{subsec:D2qhasZeroes}

Motivated by the first case of Theorem~\ref{thm:solutions-to-monge-ampere} on solutions of the homogeneous Monge-Amp\`ere equation $\det D^2 q = 0$, we first consider Ricci soliton metrics $g$ for which $D^2 q$ has zeroes. 
This includes the warped product metrics of the form $g_{\Sigma_1} \oplus \phi(x)^2 g_{\Sigma_2}$.
In subsequent sections, we will assume that none of the second derivatives of $q$ identically vanish. 
\begin{proposition}\label{prop:d2q=0}
    A Ricci soliton metric $g$ of the form~\eqref{eqn:PaperMetric} where $q$ has some vanishing second-order partial derivative must be, up to rescaling and translation of $x,y$, one of:
    \begin{enumerate}[(a)]
        \item $q=1$ and $g$ is a product of Ricci soliton surface metrics $g_{\Sigma_1} \oplus g_{\Sigma_2}$ with the same constant $\lambda$, or of a flat factor with an Einstein (constant sectional curvature) surface;
        \item $q=x$ (or $q=y$) and $g$ is the Schwarzschild metric;
        \item $q=x \pm y$ and $g$ is the \pbnski metric~\eqref{metric:pbnski};
        \item the warped product metric with a $(y,s)$-factor of constant curvature, given by
        \begin{equation}\label{eqn:warped-product-metric}\tag{\metricfifthOrderh}
            g(x,y) = \frac{1}{h(x)^2} \left( \frac{dx^2}{A(x)} \, dx^2 + A(x) \, ds^2 + \frac{dy^2}{b_2 y^2 + b_1 y + b_0} + (b_2 y^2 + b_1 y + b_0) \, ds^2 \right),
        \end{equation}
        where we define $H(x) := \int (h')^2 \, dx$ and 
        \begin{equation}\label{eqn:a-from-mathematica}\tag{$\mathbf{A}_{b_i, C_i, \lambda}$}
A(x) := \frac{b_1 C_2 (h(h')^2 + h^2 h'') + 2 h^2 h'' (\lambda + b_2 h^2)}{- 2h^3 (h'')^2 + h (h')^2 (- 2 (h')^2 + h h'') + (2H + C_1) ( (h')^3 + h h' h'') + h^3 h' h^{(3)} }
\end{equation}
    for constants $C_1, C_2$ with $b_2 C_2 = 0$, and $h(x)$ satisfying the equation
    \begin{equation}\label{eqn:fifth-order-ODE-for-h}\tag{$\mathbf{h}_\lambda$}
        A' = \left( \frac{2H + C_1}{h^2} + \frac{h'}{h} - \frac{h''}{h'} \right) A + \frac{\lambda + \frac{1}{2} b_1 C_2 + b_2 h^2}{h h'}
    \end{equation}
    for $A(x)$ defined in~\eqref{eqn:a-from-mathematica}.
    For any given choice of the constants $\{ b_1, b_2, C_1, C_2, \lambda \}$ with $b_2 C_2 = 0$, the resulting equation~\eqref{eqn:fifth-order-ODE-for-h} is equivalent to a fifth-order ODE for $h$. 
    \end{enumerate}
\end{proposition}
\begin{proof}
If some second-order derivative $\p^2_{ij} q$ of $q$ vanishes, equation~\eqref{eqn:HessianCompatibility} implies that at least two such derivatives vanish. 
Assuming by symmetry that $q_{yy} = q_{xy} = 0$, this requires $q = h(x) + cy$ for some function $h$ and constant $c$.
By rescaling and translating $x,y$ as in~\eqref{eqn:generalized-cylinder-A}, if $A$ is linear or quadratic, in all cases it is sufficient to examine $A(x) = \pm x^2 + a_0$ or $x$, likewise for $B$.

Multiplying both sides of~\eqref{eqn:A''B''-PDE} by $q^2 = (h(x) + cy)^2$ gives the explicit equation
\begin{equation}\label{eqn:A''B''-h+cy}
    (h(x) + cy)^2 A'' - S^x A' = (h(x) + cy)^2 B'' - S^y B'.
\end{equation}
The relations~\eqref{eqn:pySxpxSy} can be integrated to solve for $S^x, S^y$, for some constants $C_i$ to be determined, as
\begin{equation}\label{eqn:SxSy-h(x)+cy}
S^x = 2 H(x) + C_0 y + C_1 \quad \text{where } \; H(x) := \int (h')^2 \, dx, \qquad  S^y = 2c^2 y + 4 c h - C_0 x - C_2.
\end{equation}
\noindent\textbf{Step 1: when $h$ is linear.} 
In this case, the metric $g$ is one of (a) -- (c) above: it is a product of soliton surfaces, or it is Einstein and does not support Ricci solitons with non-Killing vector field $V$. 
Let $h(x) = a_1 x + a_0$.

If $a_1 = c = 0$, then $q=1$ and $g$ is a product metric, so $g_{\Sigma_1}, g_{\Sigma_2}$ are any two soliton surface metrics of the form~\eqref{eqn:generalized-cylinder-A} with the same $\lambda$, which we now classify.
First,~\eqref{eqn:SxSy-h(x)+cy} gives $S^x = C_0 y + C_1, S^y = - C_0 x - C_2$.
The operation $\px^2$ applied to~\eqref{eqn:A''B''-h+cy} yields the formula for $A(x)$ as $C'_3 e^{(C_0 y + C_1)x}$ plus a quadratic term, and similarly applying $\py^2$ yields the formula for $B(y)$ as $C''_3 e^{- (C_0 x + C_2)y}$ plus a quadratic term, by the $x \leftrightsquigarrow y$ symmetry. 
Therefore, $C_0 \neq 0$ and $C'_3 = C''_3 = 0$, or $C_0 = 0$.
In the former case, matching coefficients for the quadratics in~\eqref{eqn:A''B''-h+cy} shows that $A,B$ must be constant, so $g$ is the flat metric and $\lambda = 0$, by~\eqref{eqn:RicciSoliton-Last-Sx}. 
If $C_0=0$, applying $\px$ in~\eqref{eqn:A''B''-h+cy} implies $A(x) = k_1 e^{C_1 x} + a_1 x + a_0$, likewise for $B$, and $2 \lambda = a_1 C_1 = b_1 C_2$. 
Note that $C_i =0$ means replacing the corresponding exponential by a quadratic, so also $a_1 C_1$ by $- 2a_2$, likewise for $B$.
The equation~\eqref{eqn:RicciSoliton-Last-Sx} is reduced to $A'' - C_1 A' = - 2 \lambda$, so it is always satisfied.
The resulting solutions are:
\begin{equation}\label{eqn:q=1-soliton}
\begin{split}
    g &= g_{\Sigma_1} (A_{C_1,\lambda}(x)) \oplus g_{\Sigma_2}(B_{C_2, \lambda}(y)), \qquad V = C_1 A_{C_1,\lambda}(x) \partial_x + C_2 B_{C_2, \lambda}(y) \partial_y, \\
    A_{C_1,\lambda}(x) &= \begin{cases}
    k_1 e^{C_1 x} + 2 \lambda x/C_1  + a_0, & C_1 \neq 0, \\
    - \lambda x^2 + a_1 x + a_0, & C_1 = 0
    \end{cases}, \qquad \text{likewise for } \; B_{C_2,\lambda}(y).
\end{split}
\end{equation}
The factor $\Sigma_i$ is, respectively: (i) Einstein (constant curvature) if $C_i = 0$; (ii) non-trivial Ricci soliton if $C_i k_i \neq 0$; (iii) flat if $C_i \neq 0, k_i = 0$, reparametrized as $\frac{1}{x}dx^2 + x \, ds^2$ with radial vector field $\px$.
The product (i)$\times$(i) gives the only Einstein metric, as the product of Einstein surfaces of the same sectional curvature; the product (iii)$\times$(iii) gives the Gaussian soliton on $\bR^4$ for arbitrary $\lambda$.
The metrics of case (ii) form a family of incomplete solitons specializing, for $\lambda = 0$, to the steady cigar soliton of~\eqref{eqn:cigar-soliton}.

If $(a_1,c) \neq (0,0)$, assume by symmetry that $a_1 \neq 0$; then,~\eqref{eqn:A''B''-h+cy} becomes the variable-coefficient ODE
\begin{equation}\label{eqn:dalembert-form}
\left( x + b_0(y) \right)^2 A''(x) - \left( 2 x + b_1(y) \right) A'(x) = \tilde{b}_2(y) x^2 + \tilde{b}_1(y) x + \tilde{b}_0(y). 
\end{equation}
This has a one-dimensional family of solutions $A'(x)$ expressed as a fixed quadratic polynomial plus any scalar multiple of the solution of the homogeneous equation $(x+ b_0(y))^2 A''(x) = (2x + b_1(y)) A'(x)$.
The latter solution is of the form $k(y) (x+b_0(y))^2 \exp \left( \frac{2 b_0(y) - b_1(y)}{x+b_0(y)} \right)$; here $b_0(y) = \frac{cy+a_0}{a_1}$ and $b_1(y) = \frac{C_0 y + C_1}{a_1^2}$, since $a_1 \neq 0$.
Eliminating the $y$-dependence requires $k=0$ or $2 b_0(y) = b_1(y)$; this makes the homogeneous solution for $A'$ quadratic, so $A$ is a cubic. 
If $c \neq 0$, by symmetry $B$ is a cubic with coefficients determined by~\eqref{eqn:dalembert-form} as corresponding to the \pbnski metric.
If $c=0$ (but $a_1 \neq 0$), the same argument for $B$ applies, but now the left side of~\eqref{eqn:SxSy-h(x)+cy} is linear in $y$, so $B$ is quadratic and $g$ is the Schwarzschild metric.
Computing $V^x, V^y$ from~\eqref{eqn:SxSy-h(x)+cy} and~\eqref{eqn:SxSy} implies $\cL_V g = 0$, so the solitons are trivial in these Einstein cases.

\smallskip
\noindent\textbf{Step 2: when $h'' \neq 0$.}
Applying $\px^2 \py^2$ to~\eqref{eqn:A''B''-h+cy} and using~\eqref{eqn:p2ySx}, we obtain
\begin{equation}\label{eqn:h(x)+cy-2}
    \dfrac{c^2}{h''(x)} A^{(4)}(x) = \left( c y + \frac{\px^2 (h(x)^2)}{2 h''(x)} \right) B^{(4)}(y), \qquad \text{since } \; h'' \neq 0.
\end{equation}
If $B^{(4)}(y)$ is constant, then $c=0$, hence $B$ is a cubic or $\px^2 (h(x)^2) = 0$ as well. 

\smallskip
\noindent\textbf{Step 2(a): if $c=0$.}
Equation~\eqref{eqn:h(x)+cy-2} becomes $\frac{1}{2h''} \px^2 (h^2) B^{(4)} = 0$, meaning that $h(x) = \sqrt{D_1 x + D_0}$ with $D_1 \neq 0$, or $B^{(4)} = 0$.
In the first sub-case, applying $\px \py^2$ to~\eqref{eqn:A''B''-h+cy} gives $D_1 B^{(4)} + C_0 B^{(3)} = 0$, so $B(y) = k \exp \left( - \frac{C_0}{D_1} y \right) + \sum_{i=0}^2 b_i y^i$.
If $k \neq 0$, matching coefficients in~\eqref{eqn:A''B''-h+cy} gives $C_0 D_0 = C_2 D_1$ and $A$ quadratic as above.
Equation~\eqref{eqn:A''B''-h+cy} has a term $\frac{D_1}{2} \log |D_1 x + D_0|$ in $S^x$, with $D_1 \neq 0$, that cannot be matched in~\eqref{eqn:A''B''-h+cy}, so there is no solution.
If $k=0$, we again have that $A$ is a quadratic.
Applying $\px^2 \py$ to~\eqref{eqn:A''B''-h+cy} produces a logarithm term that cannot be matched if $D_1 \neq 0$, so there are again no solutions.

In the second sub-case, $B^{(4)} = 0$ implies that $B = \sum_{i=0}^3 b_i y^i$ is a cubic. 
Using~\eqref{eqn:SxSy-h(x)+cy}, we can view~\eqref{eqn:A''B''-h+cy} as an equation in $y$, which implies that
\begin{equation}\label{eqn:match-bs-2}
3 b_3 (C_0 x + C_2) y^2 = 0 \qquad \text{and} \qquad \left( 6 b_3 h^2 + C_0 A' + 2 b_2 (C_0 x + C_2) \right) y = 0.
\end{equation}
This forces $B$ to be a quadratic, since either $b_3 =0$, or $C_0 = C_2 = 0$ and the second equality gives $b_3 = 0$.
The remaining terms imply that $A$ is also a quadratic, or $C_0 = 0$.
In the former case, we must have $A(x) = - b_2 x^2 + a_1 x + a_0$ and $B(y) = b_2 y^2 + b_1 y + b_0$, so $C_2 = 0$ and~\eqref{eqn:RicciSoliton-Last-Sx} gives $h(x) = k \sqrt{b_2 x^2 - a_0}$ where $h'' \neq 0$ means $k a_0 b_2 \neq 0$.
However,~\eqref{eqn:A''B''-h+cy} requires $(h')^2 = \px (h^2/x)$, which is never satisfied for such $h(x)$, so we do not have solutions.
Finally, if $C_0 = 0$, we use~\eqref{eqn:match-bs-2} to see that $b_2 C_2 = 0$, so $B$ is linear or $S^y = 0$.
The equations 
\eqref{eqn:A''B''-h+cy},~\eqref{eqn:SxSy-h(x)+cy}, and~\eqref{eqn:RicciSoliton-Last-Sx} give coupled ODEs for $A,h$:
\allowdisplaybreaks\begin{align}
    A'' &= \frac{2H + C_1}{h^2} A' + 2 b_2 + \frac{b_1 C_2}{h^2}, \quad & & \text{where } \; H(x) := \int (h')^2 \, dx, \label{eqn:coupled-1} \\
    A' &= \left( \frac{2 H + C_1}{h^2} + \frac{h'}{h} - \frac{h''}{h'} \right) A + \frac{\lambda + \frac{1}{2} b_1 C_2 + b_2 h^2}{h h'}, \label{eqn:coupled-2}
\end{align}
where $b_2 C_2 = 0$.
Differentiating equation~\eqref{eqn:coupled-2} and combining it with~\eqref{eqn:coupled-1} and~\eqref{eqn:coupled-2} results in the expression~\eqref{eqn:a-from-mathematica} for $A(x)$. 
As analyzed in Lemma~\ref{lemma:system-alpha-beta-h}, the pair $(A,h)$ solves this system if and only if the function $A(x)$ from~\eqref{eqn:a-from-mathematica} satisfies equation~\eqref{eqn:coupled-2}; this is the same condition as~\eqref{eqn:fifth-order-ODE-for-h}.
The resulting equation is expressible as a quadratic in $H$ with coefficients in $h,h',\dots,h^{(4)}$; solving by the quadratic formula and enforcing $H' = (h')^2$ results in the equivalent fifth-order ODE for $h$.
Any solution $(A,h)$ of this system produces a non-trivial Ricci soliton, provided that $h'' \neq 0$ and $A > 0$ on an interval.

\smallskip
\noindent\textbf{Step 2(b): if $c \neq 0$.}
Equation~\eqref{eqn:h(x)+cy-2} requires $B^{(4)}(y)$ to be non-constant: it solves a differential equation in $y$ with $x$-variable coefficients. 
This equation must be expressible as $B^{(4)}(y) = \frac{1}{k_1 y + k_0}$ for constants $k_1 \neq 0, k_0$ with $[1 : k_1 : k_0] = \left[ c^2 A^{(4)} : ch'' : \frac{1}{2} \px^2 (h^2) \right]$.
We obtain $B$ as $Q_1 \log |k_1 y + k_0| + Q_0$ for cubics $Q_i$.
The relations $\px^2 (h^2) = 2 c k_0 h''/k_1$ and $\px^2 (h^2) = 2 k_0 c^2 A^{(4)}$ imply
\[
h(x) = ck_0 / k_1 \pm \sqrt{s_1 x + s_0}, \qquad A^{(2)} = \ell_1 x + \ell_0 \pm (k_1 c)^{-1} \sqrt{s_1 x + s_0}, \qquad s_1 \neq 0.
\]
By the argument used for $h(x) = \sqrt{D_1 x + D_0}$ in step 2(a) above, there is no solution in this case.
\end{proof}
We note that there exist many families of soliton metrics of the form described in Proposition~\hyperref[prop:d2q=0]{\ref{prop:d2q=0}(d)}, both explicit ones as well as abstract solutions of the fifth-order ODE for $h$. 
We highlight one such family.
\begin{corollary}\label{cor:familyOfS1Metric}
    The family of functions $h(x)$ such that $\{ \exists \; c \in \bR : hh'' = c (h')^2 \}$ coincides with the family 
    \begin{equation}\label{eqn:family-for-h(x)}
    \{ s_0 e^{s_1 x} \mid s_0, s_1 \in \bR \} \cup \{ s_0 (x + s_1)^a \mid s_0, s_1, a \in \bR \}.
    \end{equation}
    These functions produce Ricci soliton metrics described in Proposition~\hyperref[prop:d2q=0]{\ref{prop:d2q=0}(d)} for all $c \neq -1$ (i.e., for $h(x) \neq$ $\sqrt{D_1 x + D_0}$) if and only if up to scaling $g$ and rescaling and translating $x,y$, they are one of the following:
    \begin{enumerate}[(i)]
        \item $C_1 = 0$, $h(x) = x, A(x)$ is a cubic, $B(y)$ is a quadratic, and $g$ is the Schwarzschild metric;
        \item $C_1 = b_2 = 0$ and $\lambda = - b_1 C_2$, with $h(x) = e^x$ and $A(x) = \frac{1}{6} b_1 C_2 e^{-2x} + k e^x$, for some $k \in \bR$; or
        \item $C_1 = b_2 = 0$ and $\lambda = - \frac{2a-1}{2(a-1)} b_1 C_2$, for $a \not\in \{ \frac{1}{2},1\}$, with $h(x) = x^a$ and
        \begin{equation}\label{eqn:A(x)-given-a}
            A(x) = \frac{a - \frac{1}{2} b_1 C_2}{6 a^2 - 4a + 1} x^{2 - 2a} + k x^{\frac{2a^2 + 2a-1}{2a-1}}, \qquad \text{for some } \; k \in \bR.
        \end{equation}
    \end{enumerate}
\end{corollary}
\begin{proof}
    Writing the condition on the family as $\frac{h''}{h'} = c \frac{h'}{h}$ implies that $h^{-c} h'$ is a constant function, so either $c \neq 1$ and $h(x)^{1-c} = s_0^{1-c} (x + s_1)$ or $h(x) = s_0 e^{s_1 x}$.
    In the former case, we can write $h(x) = s_0 (x + s_1)^a$ for $a = \frac{1}{1-c}$.
    The condition on $h$ implies that $(h^2)'' = 2 (hh')' = 2 (1+c) (h')^2$, so the function $H(x) = \int (h')^2 \, dx = \frac{1}{1+c} h h'$, if $c \neq -1$.
    The case $c = -1$ leads to $(h^2)'' = 0$ and $h(x) = \sqrt{D_1 x + D_0}$, which was showed to produce no solutions in step 2(a) of Proposition~\ref{prop:d2q=0}.
    For $c \neq -1$, using the equality for $H$ together with $h h'' = c (h')^2$ transforms equation~\eqref{eqn:fifth-order-ODE-for-h} into
    \begin{equation}\label{eqn:fifth-order-simplified}
        A'(x) = \left( \frac{3-c^2}{1+c} \frac{h'}{h} + \frac{C_1}{h^2} \right) A(x) + \frac{\lambda + \frac{1}{2} b_1 C_2 + b_2 h^2}{h h'}.
    \end{equation}
    If $C_1 \neq 0$, using $h' h'' + h h^{(3)} = 2c h' h''$, computes expression~\eqref{eqn:a-from-mathematica} for $A(x)$ as
    \[
A(x) = h \frac{b_1 C_2  + 2 \frac{c}{1+c} (\lambda + b_2 h^2)}{C_1 h'}, \qquad A'(x) = (3-c) \frac{b_1 C_2 + 2 \frac{c}{1+c} (\lambda + b_2 h^2)}{C_1} - 2 \frac{b_1 C_2 + 2 \frac{c}{1+c}\lambda}{C_1}.
\]
Using the fact that $h(x)$ is an element of the family~\eqref{eqn:family-for-h(x)}, we verify that the two expressions for $A'(x)$ cannot be equal for $C_1 \neq 0$, unless $b_1 C_2 + 2 \frac{c}{1+c} \lambda = b_2 = 0$ or $c=0$.
The former case would give $A=0$, which is impossible; if $c=0$, then $h(x) = s_0 x, A(x)$ is a cubic, $B(y)$ is a quadratic, and $g$ is the Schwarzschild metric.
Otherwise, $C_1 = 0$ and the compatibility of equations~\eqref{eqn:coupled-1} and~\eqref{eqn:coupled-2} requires the numerator of~\eqref{eqn:a-from-mathematica} to also vanish, hence $b_2 = 0$ and $\lambda = - \frac{1+c}{2c} b_1 C_2$.
If $c=1$, we have $h(x) = e^x$ up to rescaling and translating $x$; we check that the equation resulting from~\eqref{eqn:fifth-order-simplified} has the solution $A(x)$ given in part (ii) above.

Finally, for $c \neq \{ -1, 0 , 1 \}$, we use $a = \frac{1}{1-c}$ to write $h(x) = x^a$, up to rescaling and translating $x,y$. 
Using these relations simplifies the remaining equation~\eqref{eqn:fifth-order-simplified} to
\[
A'(x) = \frac{2a^2 + 2a - 1}{2a-1} x^{-1} A - \frac{b_1 C_2}{2(a-1)} x^{1-2a}
\]
which has the solution $A(x)$ given in~\eqref{eqn:A(x)-given-a}. 
\end{proof}

\subsection{Degeneracy in the resolvent system and in Monge-Amp\`ere}\label{subsection:degeneracy-resolvent-system}

We next treat the degenerate cases occurring in the system~\eqref{eqn:Sx-system-solution-general} for $(S^x, S^y)$ when their denominator for their expressions vanishes, meaning
\begin{equation}\label{eqn:denominator-vanishes-general}
AB' q_x + A'B q_y - \frac{1}{2} A' B' q = 0.
\end{equation}
We also classify soliton metrics with $q = \theta (f_1(x) + f_2(y))$ or $q = \theta(x + y) + cx$.
The latter class of conformal factors corresponds to Theorem~\hyperref[thm:solutions-to-monge-ampere]{\ref{thm:solutions-to-monge-ampere}(ii)} for solutions of the degenerate Monge-Amp\`ere equation $\det D^2 q = 0$.

\begin{claim}\label{claim:vanishing-2x2}
If the vanishing~\eqref{eqn:denominator-vanishes-general} occurs, then one of the following holds:
\begin{enumerate}[(i)]
\item $A$ and $B$ are both constant;
\item $q$ is a single-variable function, namely $q = \sqrt{A}$ or $q = \sqrt{B}$;
\item $q = \sqrt{b_0 x^2 + a_0 y^2}$ and $A(x) = a_0 - c x^2, B(y) = b_0 + c y^2$, for any $c \in \bR$.
\end{enumerate}
\end{claim}
\begin{proof}
    If $A$ or $B$ is constant, then the condition $D^2 q \neq 0$ implies that both are constant. 
    Otherwise, letting $q = e^r$, we can rearrange~\eqref{eqn:denominator-vanishes-general} into $\frac{\px r}{(\log A)'} + \frac{\py r}{(\log B)'} = \frac{1}{2}$, whereby Corollary~\ref{corollary:apxr+bpyr=c} shows that $r$, and so $q$, must have the form
    \[
    r = \tilde{\theta} \left( \log (A/B) \right) + \frac{1}{2} \log A \quad \implies\quad q = A^{1/2} \theta(A/B). 
    \]
    We use this form of $q$ in~\eqref{eqn:HessianCompatibility}.
    For $\theta' \neq 0$, we can divide both sides by $(\theta')^2 B^{-3} (A' B')^2$, collect terms, and let $t := \dfrac{A}{B},  \alpha(x) := \frac{A A''}{(A')^2},  \beta(y) := \frac{B B''}{(B')^2}$.
    Then, the Monge-Amp\`ere equation for $q$ is homogenized to a polynomial in $(t, \alpha, \beta, \frac{\theta}{\theta'}, \frac{\theta''}{\theta'})$.

    Using the compatibility of the system resulting in~\eqref{eqn:Sx-system-solution-general} and $q_y = - A^{3/2} B^{-2} B' \theta'$, we obtain
    \[
    B (A'' - B'') q q_y + B' \gamma = 0, \quad \implies - B(A'' - B'') \theta \theta' t^2 + \gamma = 0.
    \]
    Solving the resulting $ 2 \times 2$ system for $\frac{\theta'}{\theta}, \frac{\theta''}{\theta}$ in terms of $A^{(i)}, B^{(i)}$ and using the property that they are functions of $t$ allows us to deduce that the resulting second-order differential equation in $\theta$ has solution $\theta(t) = \sqrt{ |1-t^{-1}| }$.
    The functions relating $\alpha, \beta$ to $A,B$ are $\alpha(A) = \pm \frac{A}{2(1-A)}$, likewise for $\beta$.
    Combined with the definition of $\alpha, \beta$, and the above systems, we solve for $A,B$ (up to rescaling and translating $x,y$) to obtain $A = a_0 \mp x^2$ and $B = b_0 \pm y^2$ in this situation.
    Thus, we arrive at case (iii), which is the final possibility.
\end{proof}
Having used the system~\eqref{eqn:Sx-system-solution-general} to decouple the soliton equation, we treat the final possibly degenerate case of Theorem~\ref{thm:main-theorem}.
\begin{proposition}\label{prop:metrics-s2-through-s5}
Suppose that the conformal factor has the form $q(x,y) = \theta(x+y) + cx$ for a single-variable function $\theta$ and a constant $c$.
Then, up to translating and rescaling $x,y$, either $\theta$ is linear and $g$ is the \pbnski metric~\eqref{metric:pbnski}, or $c=0$ and $q = \theta(x+y)$ has one of the following forms:
\begin{enumerate}
    \item[$\mr{(\metricStarStar)}$] $(A,B) = (a_0,b_0)$ and $\theta(t)$ satisfies the third-order ODE~\eqref{eqn:theta-ODE-1}; 
    \item[$\mr{(\metricQIsSqrtXPlusY)}$] $(A,B) = (x,y)$, $\theta(t) = \sqrt{t}$ and $\lambda = \frac{1}{2}$, resulting in a singular shrinking soliton;
    \item[$\mr{(\metricStarStarBar)}$] $(A,B) = (x,y)$ and $\theta(t) \neq C \sqrt{t}$ satisfies the third-order ODE~\eqref{eqn:theta-ODE-2};
    \item[$\mr{(\metricConfProdCigars)}$] $q = e^{x+y}$ and $g$ is an incomplete steady soliton~\eqref{eqn:Soliton-q=exp(x+y)-metric} conformal to the product of two cigar metrics.
\end{enumerate}
\end{proposition}
It is interesting to observe that this classification covers all the Ricci solitons presented in Table \ref{fig:SolitonEncyclopedia} that are not conformally cylindrical nor conformally products.
In this sense, this result can be viewed as complementary to Section \ref{sec:homogeneous} in terms of proving Theorem \ref{thm:secondary-thm}.
\begin{proof}
We define the function $\Theta(t) := \int (\theta')^2 \, dt$, so integrating equations~\eqref{eqn:pySxpxSy} --~\eqref{eqn:pySy} gives
        \begin{equation}\label{eqn:sx,sy-formulas-for-x+y}
            S^x = 2 \Theta(x+y) + 4 c \theta(x+y) + 2 c^2 x + C_0 y + C_1, \qquad S^y = 2 \Theta(x+y) - C_0 x - C_2
        \end{equation}
        for constants $C_i$.
        In what follows, we will suppress the argument $x+y$ in $\Theta$ and $\theta$.
        Using the expressions for $q$ and $S^x, S^y$ from~\eqref{eqn:sx,sy-formulas-for-x+y}, we write~\eqref{eqn:A''B''-PDE} --~\eqref{eqn:RicciSoliton-Last-Sx} as
\begin{align}
    0 &= (\theta + cx)^2 (A'' - B'') - 2 \Theta (A' - B') -(4 c \theta + 2 c^2x + C_0 y + C_1) A' - (C_0 x + C_2) B', \label{eqn:A''B''-PDE-q(x+y)} \\
    \begin{split}\label{eqn:Last-Sx-x+y}
    0 &= \lambda (\theta + cx) + \frac{1}{2} (\theta + cx) \left( (\theta + cx)^2 B'' - S^y B' \right) - (\theta + cx)^2 \theta' (A' + B') - c (\theta + cx)^2 A' \\
    & \; \; \; + (\theta + cx) ( (\theta')^2 - (\theta + cx) \theta'') (A + B) + (\theta + cx) (c^2 + 2c \theta') A +  (\theta' + c) (S^x A + S^y B).
    \end{split}
\end{align}
Unless noted otherwise, we will be differentiating equations~\eqref{eqn:A''B''-PDE-q(x+y)} and~\eqref{eqn:Last-Sx-x+y} by the operator $\px - \py$, which satisfies the Leibniz rule.
For example, $(\px - \py) (\tilde{\theta} \cdot u) = \tilde{\theta} \cdot (\px u - \py u)$ for any function $\tilde{\theta}(x+y)$.

If $(A,B) = (a_0, b_0)$ are constant, expanding equation~\eqref{eqn:RicciSoliton-Last-Sx} using $q$ and $S^x, S^y$ from~\eqref{eqn:sx,sy-formulas-for-x+y} shows that $c=0$, otherwise equation~\eqref{eqn:Last-Sx-x+y} cannot be satisfied due to the additional terms involving $c$ that depend only on $x$.
If $c = 0$, then also $C_0 = 0$ and $(\theta, \Theta)$ only need to satisfy~\eqref{eqn:Last-Sx-x+y}.
Differentiating by $\px - \py$ then produces the equation
\begin{equation}\label{eqn:theta-ODE-1}\tag{$\star \star_{\lambda}$}
0 = \theta^2 \theta' \theta^{(3)} - \theta^2 (\theta'')^2 + \theta \left( \frac{\lambda}{a_0 + b_0} + (\theta')^2 \right) \theta'' - 3 (\theta')^4 - \frac{\lambda}{a_0 + b_0} (\theta')^2. 
\end{equation} 
For non-constant $A,B$, equation~\eqref{eqn:A''B''-PDE-q(x+y)} is non-trivial. 
We also note that for functions $\phi, \psi$ of $t$, and $\theta, \Theta$ as defined at the beginning of the proof, we may rearrange
\begin{equation}\label{eqn:theta-phi-psi}
\theta^2 - 2 \Theta \phi + \psi = 0 \implies \theta' = \frac{1}{2 \phi} \left( \theta \pm \sqrt{\theta^2 (1 - 2 \phi') + 2 (\psi' \phi - \psi \phi')} \right)
\end{equation}
by differentiation and the quadratic formula.

If one of $A,B$ is linear, then applying $(\px - \py)$ to~\eqref{eqn:A''B''-PDE-q(x+y)} shows that either $c=0$, or $\theta$ is linear, so $g$ is the \pbnski metric~\eqref{metric:pbnski}.
For $c=0$,~\eqref{eqn:A''B''-PDE-q(x+y)} gives $(A,B) = (x,y)$ and $C_i = 0$ and we can solve~\eqref{eqn:Last-Sx-x+y} for $\Theta$, provided that $\theta(t) \neq C \sqrt{t}$. 
If $\theta(t) = \sqrt{t}$, the equality~\eqref{eqn:Last-Sx-x+y} holds for $\lambda = \frac{1}{2}$, yielding the metric
\begin{equation}\label{eqn:new-soliton}\tag{\metricQIsSqrtXPlusY}
    g = \frac{1}{x+y} \left( \frac{1}{x} dx^2 + \frac{1}{y} dy^2 + x \, ds^2 + y \, dt^2 \right), \qquad V = \left( \frac{1}{2} \log(x+y) - 1 \right) (x \px + y \py ).
\end{equation}
This corresponds to a conformally flat shrinking soliton ($\lambda = \frac{1}{2}$) with ends at the axes. 
At any point along the axes, for example $(0,y)$, the length of the horizontal curve from $(1,y)$ is $\int_0^1 \frac{dx}{\sqrt{x^2 + xy}} \asymp y^{-\frac12}\int_0^1 \frac{dx}{\sqrt{x}} < \infty$, so for $y \neq 0$, these points are singularities.

When $(A,B) = (x,y)$ and $\theta \neq C \sqrt{t}$, solving~\eqref{eqn:Last-Sx-x+y} for $\Theta$ and enforcing $\Theta' = (\theta')^2$ amounts to the ODE
\begin{equation}\label{eqn:theta-ODE-2}\tag{$\overline{\star \star}_{\lambda}$}
\begin{split}
   0 &= t \theta^2 \left( \theta - 2 t \theta' \right) \theta^{(3)} + 2 \, t^2 \theta^2 (\theta'')^2 + \theta \left( 3 \theta^2 - t \theta \theta' - 2 \, t^2 (\theta')^2  - 2  \lambda  t \right) \theta'' + 6 (\theta')^2 ( \theta - \theta' t)^2 - 2 \lambda \theta' (\theta - \theta' t) .
\end{split}
\end{equation}
In all other cases, $A'' - B'' \neq 0$ in~\eqref{eqn:A''B''-PDE-q(x+y)}.
If one of $A,B$ is quadratic or cubic, of the form $A = \sum_i a_i x^i, B = \sum_j b_j y^j$, differentiating equation~\eqref{eqn:A''B''-PDE} by $(\px - \py)^{\deg B}$ shows that either $c=0$, or $\theta'' = 0$, or one of $a_3 x^2 + b_3 y^2$ (in the cubic case) or $a_2 x - b_2y$ (in the quadratic case) is a function of $x+y$.
        Since the latter property cannot hold, we deduce that $\theta$ is linear. 
        If $c=0$, then~\eqref{eqn:A''B''-PDE-q(x+y)} reduces to an equation of the form~\eqref{eqn:theta-phi-psi} that yields $\theta' = \frac{\theta}{t}$, so $\theta$ is linear and again $g$ is the \pbnski metric~\eqref{metric:pbnski}.
        
        In all other cases, we now assume $A^{(4)}, B^{(4)} \neq 0$. 
Note that
\begin{equation}\label{eqn:sx-sy-a-b-derivatives}
    (\px - \py)(S^x A + S^y B) = 2 c^2 A - C_0(A+B) + (\theta + cx)^2 (A'' - B'')
\end{equation}
by using the expressions~\eqref{eqn:sx,sy-formulas-for-x+y} and~\eqref{eqn:A''B''-PDE-q(x+y)}.
Applying $\py$ to~\eqref{eqn:A''B''-PDE-q(x+y)} then gives
\begin{equation}\label{eqn:b-triple-prime-sub}
- (\theta + cx)^2 B''' + S^y B'' = 2 (\theta + cx) \theta' B'' - 2\theta'^2 B' - 2 (\theta + cx) \theta' A'' + (2 \theta'^2 + 4 c \theta' + C_0) A'. 
\end{equation}
Applying $(\px - \py)$ to~\eqref{eqn:Last-Sx-x+y} and using~\eqref{eqn:b-triple-prime-sub} as well as $S^y = 2 \Theta - C_0 x - C_2$ from~\eqref{eqn:sx,sy-formulas-for-x+y}, we obtain 
\begin{align*}
    c \Theta B' &= \lambda c + \frac{1}{2} c (C_0 x + C_2) B' + \frac{1}{2} (\theta + cx) ( (C_0 - 2c^2) A' + C_0 B') - \frac{1}{2} (\theta + cx)^2 ( 2 \theta' A''- (2 \theta'+c) B'') \\
    & \; \; \; + c^2 (3c + 4 \theta') A  + c \left( (\theta')^2 - 2 (\theta + cx) \theta'' \right) (A+B) + (\theta + cx) ( 2 (\theta')^2 + 2 c \theta' - (\theta + cx) \theta'') (A' - B') .
\end{align*}
First, if $c \neq 0$, we can divide the right-hand side by $cB'$ to solve for $\Theta$.
Setting $\beta(y) := \log B(y)$ here, since $B > 0$, we have $\frac{B}{B'} = \frac{1}{\beta'}$ and $\frac{B''}{B'} = \frac{\beta''}{\beta'} + \beta'$.
Applying $\py \Theta = (\theta')^2$ and combining it with the equation for $\Theta$ derived from~\eqref{eqn:A''B''-PDE-q(x+y)} produces a system of equations as functions of $x+y$, which can be related via $\Theta' = (\theta')^2 = \theta^2 \cdot \left( \frac{1}{2} \frac{(\theta^2)'}{\theta^2} \right)^2$.
Solving for $\theta'$ by the formula~\eqref{eqn:theta-phi-psi}, we see that the resulting system admits no solutions for $c \neq 0$, due to the additional terms involving $c$ that only depend on $x$.

Next, if $c=0$, the above equation involving $c \Theta B'$ simplifies to
\begin{equation}\label{eqn:remaining-equation-c=0}
C_0(A'+B') - 2 \theta \theta' (A''- B'') + 2 (2 \theta'^2 - \theta \theta'')(A' - B') = 0. 
\end{equation}
Applying $(\px-\py)$ to equation~\eqref{eqn:A''B''-PDE-q(x+y)} and combining it with~\eqref{eqn:remaining-equation-c=0}, we deduce that $C_0 = 0$, and so equation~\eqref{eqn:remaining-equation-c=0} implies that $\frac{A'' - B''}{A' - B'} = \phi(x+y)$ for some $\phi$. 
Applying $(\px - \py)$ to this relation, we obtain
\[
\frac{A'' - B''}{A' - B'} = \phi = \frac{A^{(3)} + B^{(3)}}{A'' + B''} \implies (A'')^2 - (B'')^2 = (A^{(3)} + B^{(3)}) (A' - B').
\]
Differentiating this relation in $x$ gives $A'' A^{(3)} = - B' A^{(4)} + A'' B^{(3)} + A' A^{(4)}$, which can be viewed as an ODE in $B(y)$, of the form $c_1(x) B^{(3)} - c_2(x) B' = c_3(x)$.
This ODE requires $B(y) = P(y) + b_{\pm} \exp \left( \pm \sqrt{c_2/c_1} y \right)$, for $P$ a cubic polynomial; by symmetry, the same holds for $A$.
Substituting into~\eqref{eqn:A''B''-PDE-q(x+y)} shows that one of the summands is trivial; up to rescaling, we may consider $A,B$ to be exponential, trigonometric, or hyperbolic.

For $(A,B) = (\pm \sin x, \pm \sin y)$, the sum-to-product formulas reduce~\eqref{eqn:A''B''-PDE-q(x+y)} to an equation of the form~\eqref{eqn:theta-phi-psi} with $\psi = 0$ and $\phi = - \frac{1}{\tan(t/2)}$.
Here, $1-2 \phi' < 0$ except at $t= (2k+1)\pi$, so~\eqref{eqn:theta-phi-psi} cannot be satisfied; the computations involving $\cos(-)$ are analogous.
For hyperbolic functions, the same argument allows us to complete the square in the region $\cosh^2 t \geq 2$ and thus compute $\theta$; substituting all the terms again produces no solution.
In the exponential case,~\eqref{eqn:A''B''-PDE-q(x+y)} yields $\theta(t) = e^t$, so $q = e^{x+y}$, and $C_0 = C_1 = 0$, whereby $A(x) = k_{A,1} e^x - cx + k_{A,0}$, likewise for $B$.
Computing all the terms of~\eqref{eqn:RicciSoliton-Last-Sx} shows that this is equation satisfied if and only if $c=\lambda=0$ and $k_{A,0} = - k_{B,0} = k_0$, so upon rescaling, 
\begin{equation}\label{eqn:Soliton-q=exp(x+y)-metric}\tag{\metricConfProdCigars}
\begin{split}
    g &= e^{-2(x+y)} \left( \dfrac{1}{k_1 e^{x} + k_0} dx^2 + \dfrac{1}{k_2 e^y - k_0} dy^2 + \left( k_1 e^x + k_0 \right) ds^2 + \left( k_2 e^y - k_0 \right) dt^2  \right), \\
    V &= - e^{-2 (x+y)} \left( (k_1 e^x + k_0) \px + (k_2 e^y - k_0) \py \right).
    \end{split}
    \end{equation}
This soliton is conformal to~\eqref{eqn:q=1-soliton} (where $q=1$), specialized to $\lambda =0$.
As noted in Proposition~\ref{prop:d2q=0}, the latter metric is the product of two steady cigar solitons.
The metric~\eqref{eqn:Soliton-q=exp(x+y)-metric} has maximal domain $(x,y) \in \bR \times \left( \log (k_0/k_2),\infty \right)$.
As $x \to +\infty$, it is approximated by $g \approx \hat{g} := e^{-2y}\left(k_1e^{-3x}\, dx^2 + k_1e^{-x}\,ds^2\right) + e^{-2x}\left(\dfrac{1}{k_2 e^y - k_0} \,dy^2+ \left( k_2 e^y - k_0 \right)\, dt^2 \right)$.
Taking $x \to + \infty$ for fixed $y_0$ shows that $\hat{g}$, and so $g$, is incomplete, as the $(y,t)$ directions collapse rapidly.
\end{proof}

\subsection{Conformally flat solitons}\label{subsect:conformally-flat}

When $A,B$ in~\eqref{eqn:PaperMetric} are constants, we may rescale to express $g = q^{-2} g_{\on{Euc}}$.
By the work of Cao-Chen \cite{cao-chen}, a complete (non-compact) locally conformally flat shrinking or steady gradient Ricci soliton is one of: Einstein (including the Gaussian shrinkers), cylindrical ($\bS^{n-1} \times \bR$), the Bryant soliton, or a quotient thereof.
Proposition~\ref{prop:d2q=0} recovers some solitons of this form, namely the Einstein or Gaussian soliton structures on $\bR^4$, when $q=1$, and the degeneration of the \pbnski metric, when $q=x+y$.
On the other hand, the result of \cite{cao-chen} does not rule out the existence of expanding Ricci solitons; indeed, our analysis in the following lemma obtains an explicit complete Ricci expander. 

It may be possible to obtain other complete expanding Ricci solitons from abstract existence results for the PDEs presented below.
The flexibility of conformally flat Ricci soliton metrics is expected due to case (d) of Proposition~\ref{prop:d2q=0}, where the flat geometry in the $y$-direction led to great flexibility in the $x$-direction.
Taking $A=1$ in Proposition~\hyperref[prop:d2q=0]{\ref{prop:d2q=0}(d)} implies $B=1$ and reduces the condition to a third-order equation for $h$ coming from~\eqref{eqn:coupled-2}, with many explicit solutions. 
We generalize this class of solutions in the following lemma.
\begin{lemma}\label{lemma:conformally-flat}
    Suppose that $D^2 q \neq 0$.
    The metric
    \begin{equation}\label{eqn:conformally-flat-metric}\tag{\metricCfCf}
        g = q^{-2} g_{\on{Euc}} = \frac{1}{q(x,y)^2} (dx^2 + dy^2 + ds^2 + dt^2)
    \end{equation}
    is a Ricci soliton if and only if one of the following holds:
    \begin{enumerate}[(i)]
        \item it is an expanding soliton with $q = C \sqrt{x^2 + y^2}$ and $\lambda = - 2 C$;
        \item $q = \theta(x+y)$, where $\theta$ satisfies the third-order ODE~\eqref{eqn:theta-ODE-1};
        \item $q$ satisfies $\det D^2 q = 0$ together with the pair of fourth-order differential equations given by 
    \begin{align}
            \px \tilde{S} &= 2 q_x^2, \label{eqn:px-s-tilde} \tag{$\mathbf{CF}_1$} \\
            \tag{$\mathbf{CF}_2$} \begin{split} 
    q_x q_y \py \tilde{S} &=  (q_x q_{yy} - q_y q_{xy}) \tilde{S} + q^2 q_y (q_{xxy} + q_{yyy}) - 2 q q_y ( q_x q_{xy} - q_y q_{xx}) - q_y^2 (q_x^2 + 3 q_y^2 + \lambda) \\
    & \; \; \; + q q_{yy} \left( - q (q_{xx} + q_{yy}) + q_x^2 + q_y^2 \right) - \lambda q q_{yy}, \label{eqn:py-s-tilde}  
\end{split}
    \end{align}
    where we define
    \begin{equation}\label{eqn:S-tilde-x}
\begin{split}
\tilde{S} &:= q q_{xy} \frac{q (q_{xx} + q_{yy}) - (q_x^2 + q_y^2 + \lambda)}{q_x q_{xy} - q_y q_{xx}} - 2 q q_y \frac{q_x q_{xy} - q_y q_{xx}}{q_x q_{xx} + q_y q_{xy}} \\ 
& \;\;\quad - q_y q_{xx} \frac{q^2 \left( q_x (q_{xxx} + q_{xyy}) + q_y (q_{xxy} + q_{yyy}) \right) - 3 (q_x^2 + q_y^2)^2 - \lambda (q_x^2 + q_y^2)}{(q_x q_{xx} + q_y q_{xy}) (q_x q_{xy} - q_y q_{xx})}.
\end{split}
\end{equation}
\end{enumerate}
\end{lemma}
\begin{proof}
The conditions of Corollary~\ref{corollary:exist-Sx-Sy} simplify to the differential system~\eqref{eqn:pySxpxSy} --~\eqref{eqn:pySy} together with
\begin{equation}\label{eqn:conformally-flat-last-sx}
    q_xS^x + q_y S^y = q \gamma,
\end{equation}
where $\gamma := \gamma(q,1,1,\lambda) = q (q_{xx} + q_{yy}) - (q_x^2 + q_y^2 + \lambda)$ is the function defined in~\eqref{eqn:gamma-A,B,q}, when $A = B = 1$.
Equation~\eqref{eqn:conformally-flat-last-sx} is obtained from~\eqref{eqn:Last-Sx-Sy-gamma} upon setting $A=B=1$.
The equality~\eqref{eqn:conformally-flat-last-sx} implies that
\begin{equation}\label{eqn:setup-system-qxsx+qysy}
q_x \px \left( q_x S^x + q_y S^y \right) + q_y \py \left( q_x S^x + q_y S^y \right) = (q_x^2 + q_y^2) \gamma + q (q_x \gamma_x +  q_y \gamma_y)
\end{equation}
which, upon using~\eqref{eqn:pySxpxSy} --~\eqref{eqn:pySy}, takes the form
\begin{equation}\label{eqn:sx-sy-second-equation}
\begin{split}
& q_{xx} S^x + q_{xy} S^y \\
&= \frac{q^2 q_{xx} ( q_x (q_{xxx} + q_{xyy}) + q_y (q_{xxy} + q_{yyy})) - q_{xx} (3 (q_x^2 + q_y^2)^2 + \lambda (q_x^2 + q_y^2)) + 2q (q_x q_{xy} - q_y q_{xx})^2}{q_x q_{xx} + q_y q_{xy}}
\end{split}
\end{equation}
upon applying~\eqref{eqn:HessianCompatibility}.
We may view the pair of equations~\eqref{eqn:conformally-flat-last-sx} and~\eqref{eqn:sx-sy-second-equation} as a $2 \times 2$ system with unknowns $S^x, S^y$, which has unique solution $(S^x, S^y) = \left(\tilde{S}, \frac{q \gamma}{q_y} - \frac{q_x}{q_y} \tilde{S} \right)$ for $\tilde{S}$ given by~\eqref{eqn:S-tilde-x}.
The expression for $\tilde{S}$ is well-defined unless
\[
q_x q_{xy} - q_y q_{xx} = 0 \qquad \text{or} \qquad q_x q_{xx} + q_y q_{xy} = 0 .
\]
The first case implies that $\frac{q_x}{q_y} = a(x)$, while the second one is equivalent to$q_x^2 + q_y^2 = 2 b(y)$, for single-variable functions $a,b$.
In the first case, the argument of Theorem~\ref{thm:solutions-to-monge-ampere} (an application of Lemma~\ref{lemma:level-set-composition} following~\ref{eqn:homo-first-step}) shows that, in fact, $\frac{q_x}{q_y} = C_2$ and $q = \theta(x+y)$ up to rescaling and translating $x,y$.
This possibility was treated in Proposition~\ref{prop:metrics-s2-through-s5} and amounts to equation~\eqref{eqn:theta-ODE-1}.
In the second case, we may use $D^2 q \neq 0$ to write
\[
q_x q_{xx} + q_y q_{xy} = 0, \qquad \frac{q_{xx}}{q_{xy}} = \frac{q_{xy}}{q_{yy}}, \qquad \implies b'(y) = \frac{1}{2} \partial_y (q_x^2 + q_y^2) = q_x q_{xy} + q_y q_{yy} = 0,
\]
whereby $q$ satisfies the eikonal equation $q_x^2 + q_y^2 =C_1$.
We may scale $q$ to take $C_1 = 1$ and obtain the equalities
\begin{equation}\label{eqn:eikonal-relations}
q_{xy} = - \frac{q_x}{q_y} q_{xx}, \qquad q_{yy} = - \frac{q_x}{q_y} q_{xy} =  \frac{q_x^2}{q_y^2} q_{xx}, \qquad \frac{q_x}{q_y} q_{xxx} + q_{xxy} = -\frac{q_{xx}^2}{q_y^3},
\end{equation}
where the last one is derived by differentiating the first two.
In this situation, the coefficients of $S^x$ and $S^y$ in equation~\eqref{eqn:setup-system-qxsx+qysy} vanish, while the eikonal relations~\eqref{eqn:eikonal-relations} transform the equation~\eqref{eqn:setup-system-qxsx+qysy} into 
\[
2 q q_y^2 q_{xx} - q^2 q_{xx}^2 = (1 - (\alpha - 1)^2) q_y^4, \qquad \text{where } \; \lambda = - (\alpha - 1)^2 - 2. 
\]
Rearranging the terms in this equation implies 
\[
\alpha^{-1} q q_{xx} = q_y^2, \quad \implies \frac{q_{xy}}{q_y} = - \alpha \frac{q_x}{q}, \quad \implies \px \log q_y = - \alpha \px \log q, 
\]
by using the first relation from~\eqref{eqn:eikonal-relations}. This implies that the function $q^{\alpha}q_y$ is constant in $x$, so $\py (q^{\alpha+1}) = G'_2(y)$ and $q^{\alpha +1} = G_1(x) + G_2(y)$.
The eikonal equation for $q$ therefore takes the form
\[
1 = q_x^2 + q_y^2 = \frac{1}{(\alpha + 1)^2} (G'_1(x)^2 + G'_2(y)^2) (G_1(x) + G_2(y))^{- \frac{2 \alpha}{\alpha + 1}}
\]
which for $\frac{2 \alpha}{\alpha +1 } \not\in \{ 0,1 \}$ has only the trivial solution with $G_i$ constant.
When $\alpha = 0$, the $G_i$ must be linear; when $\alpha = 1$, we obtain $\lambda = - 2$ and $G_1(x) + G_2(y) = x^2 + y^2$, so $g = \frac{1}{x^2 + y^2} g_{\on{Euc}}$.

In all other cases, when the expression $\tilde{S}$ given by~\eqref{eqn:S-tilde-x} is well-defined, we claim that the soliton conditions of~\ref{corollary:exist-Sx-Sy} become equivalent to the pair of equations~\eqref{eqn:px-s-tilde} and~\eqref{eqn:py-s-tilde}, i.e.,~\eqref{eqn:pySxpxSy} holds automatically.
The first direction follows since the $(S^x, S^y) = \left(\tilde{S}, \frac{q \gamma}{q_y} - \frac{q_x}{q_y} \tilde{S} \right)$ are determined uniquely and satisfy~\eqref{eqn:pxSx},~\eqref{eqn:pySy}.
Conversely, defining $(S^x, S^y) = (S^x, S^y) = \left(\tilde{S}, \frac{q \gamma}{q_y} - \frac{q_x}{q_y} \tilde{S} \right)$ in this manner ensures that they satisfy~\eqref{eqn:conformally-flat-last-sx} and~\eqref{eqn:sx-sy-second-equation}.
Differentiating~\eqref{eqn:conformally-flat-last-sx} with $\px, \py$ and matching coefficients shows that
\[
q_xq_y(\py S^x + \px S^y) = 4 q_x^2 q_y^2, \quad \implies~\eqref{eqn:pySxpxSy}
\]
since $\tilde{S}^x, \tilde{S}^y$ are constructed to satisfy~\eqref{eqn:conformally-flat-last-sx}.
\end{proof}

Combining the results presented in this section recovers all the metrics introduced in Theorem~\hyperref[thm:secondary-thm]{\ref{thm:secondary-thm}(i)}.
This will be crucial in proving Theorem \ref{thm:secondary-thm}, upon invoking the classification Theorem \ref{thm:solutions-to-monge-ampere} for local solutions $q$ of the Monge-Amp\`ere equation.
Namely, any $q$ as in Theorem~\hyperref[thm:solutions-to-monge-ampere]{\ref{thm:solutions-to-monge-ampere}(i)} produces the metrics of Proposition \ref{prop:d2q=0}, while any $q$ as in Theorem~\hyperref[thm:solutions-to-monge-ampere]{\ref{thm:solutions-to-monge-ampere}(ii)} has $q = \theta(x+y) + cx$ and the resulting metrics
are classified in Proposition \ref{prop:metrics-s2-through-s5}.
The remaining class of solutions introduced in Theorem~\hyperref[thm:solutions-to-monge-ampere]{\ref{thm:solutions-to-monge-ampere}(iii)} is, in some sense, too general to classify completely.
The collection of such functions contains the homogeneous conformal factors $q$, whose rich geometry we now investigate.
In our proof of Theorem~\hyperref[thm:secondary-thm]{\ref{thm:secondary-thm}(ii)}, we will show that $q$ must have this form when one of the local surface factors has constant curvature.

\section{Conformally cylindrical base}\label{sec:homogeneous} 

We specialize to the conformally cylindrical ansatz where $q$ is homogeneous, i.e., $q(\tau x,\tau y) = \tau q(x,y)$, which forces a geometric volume collapse at infinity and removes the Monge-Amp\`ere obstruction to solutions.
We call this metric property conformally cylindrical because it makes the base metric in the $(x,y)$-plane asymptotically cylindrical near the origin.
The base has domain $\Omega \ni (x,y)$ over which the total space is a torus fibration.
The conformally cylindrical property reduces the soliton PDE to an ODE, yielding explicit solutions and geometries of cohomogeneity two Ricci solitons as well as abstract existence results.
All previously known Einstein metrics and explicit solitons of the form~\eqref{eqn:PaperMetric} have conformally cylindrical base.

If $q$ is homogeneous of degree $1$, it satisfies the Hessian compatibility relation~\eqref{eqn:HessianCompatibility}.
By Euler's theorem, homogeneity is equivalent to $x q_x + y q_y = q$.
Applying $\py$ (respectively, $\px$) to this relation implies that
\begin{equation}\label{eqn:homogeneous-partials}
        x q_{xy} + y q_{yy} =0 \qquad \text{and} \qquad x q_{xx} + y q_{xy} =0.
\end{equation}
These equalities imply $q_{xx}q_{yy} = ( - \frac{y}{x} ) q_{xy} ( - \frac{x}{y} ) q_{xy} = q_{xy}^2$, so $q$ satisfies~\eqref{eqn:HessianCompatibility} and $x^2 q_{xx} = y^2 q_{yy}$.

In the cohomogeneity one case, Ricci solitons can be constructed by reducing the soliton equation to ODEs which can be well-described analytically and numerically, cf.~\cites{buzano-dancer-wang, dancer-hall-wang, dancer-wang}.
Recently, Nienhaus and Wink \cite{nienhaus-wink-4d} extended this analysis to the ODE system resulting from a higher-dimensional warped product metric to produce asymptotically conical expanding Ricci solitons with positive scalar curvature on $\bR^3 \times \bS^1$.
In our setting, the high degree of symmetry imposed on the metric~\eqref{eqn:PaperMetric} together with the conformally cylindrical base condition requires different techniques and existence arguments.

In general, $q$ being homogeneous means that the base metric is asymptotically cylindrical near the origin.
When $A,B$ are constants, the base metric is fully cylindrical. 
Near the origin, as $(A,B) \approx (A(0), B(0))$, the metric converges to a cylinder.
Since $q$ vanishes at the origin, we can describe the two-dimensional asymptotic cylinder as the cylinder over a level set of $q$.
For $A,B$ constant, the level sets of $q$ are all isometric because the domain increases linearly with the radius $r = \sqrt{x^2 + y^2}$ while the conformal factor scales inversely, demonstrating the cylindrical structure. 
Letting $N = \{(x,y) : q(x,y) = 1\}$ with the induced metric, the manifold near the origin will be $\bR \times N$ with the cylinder metric $dr^2 + g_N$.
Similarly, as $x$ or $y$ diverges to infinity, if $A(x)$ and $B(y)$ approach constant values along a direction to infinity, the metric will likewise converge to a cylinder.
As $(x,y) \to (0,0)$, the size of the $s$ or $t$ circle fibers grows linearly, so along a straight line to the origin, this fibration forms a cone with Euclidean volume growth.
When either $x$ or $y$ diverges to infinity, one of the circles in the $s$ or $t$ direction must collapse to 0.

We will prove that the resolvents $S^x, S^y$ inherit homogeneity properties from $q$.
We define $S := x S^x + y S^y$ and use the relations~\eqref{eqn:pySxpxSy} --~\eqref{eqn:pySy} to compute
\begin{align}
    \partial_i \partial_j (S - 2q^2) &= 4 (\partial_j \partial_j q)(x q_x + y q_y - q), & & \text{for } \; i,j \in \{ x, y \}, \label{eqn:pipjSh2q2} \\
    x \partial_x (S - 2q^2) + y \partial_y (S - 2q^2) &= (S - 2q^2) + 2 (x q_x + y q_y - q)^2. \label{eqn:xpx+ypySh2q2}
\end{align}
\begin{lemma}\label{lemma:Cx=Cy=0}
If $q$ is homogeneous and $q_{xy} \neq 0$, then $S^x, S^y$ are homogeneous of degree 1, and further satisfy
\begin{equation}\label{eqn:xSx+ySy=2q2}
x S^x + y S^y = 2q^2.
\end{equation}
Conversely, equation~\eqref{eqn:xSx+ySy=2q2} implies $q$ is homogeneous.
Additionally, if there exits a soliton metric \eqref{eqn:PaperMetric} with $q = \tilde{q} + c$ for $\tilde{q}$ a homogeneous function and $c$ a constant, then $q$ is homogeneous up to translating $x,y$.
\end{lemma}
\begin{proof}
First, if condition~\eqref{eqn:xSx+ySy=2q2} holds, then applying $\px \py$ to it and using using~\eqref{eqn:pipjSh2q2} shows that $x q_x + y q_y = q$, so $q$ is homogeneous. 
Conversely, suppose that $q = \tilde{q} + c$ as described above.
We first show that 
\begin{equation}\label{eqn:xSx-Cx+ySy-Cy=2q2}
x (S^x - C^x) + y (S^y - C^y) = 2 \tilde{q}^2, \qquad \text{for some constants } \; C^x, C^y,
\end{equation}
and then show that these constants vanish, so $c = C^x = C^y = 0$.
This will mean that $q = \tilde{q}, S^x$, and $S^y$ are homogeneous of degree $1$, as desired.
To obtain the expression~\eqref{eqn:xSx-Cx+ySy-Cy=2q2}, we use the homogeneity and relation~\eqref{eqn:pipjSh2q2}, which together imply that $S - 2 \tilde{q}^2$ is the sum of two single-variable functions in $x$ and $y$.
Equation~\eqref{eqn:xpx+ypySh2q2} also shows that $S - 2 \tilde{q}^2$ is homogeneous of degree $1$.
These two properties allow us to express 
\[ 
S - 2 \tilde{q}^2 := x S^x + yS^y - 2 \tilde{q}^2 = C^x x + C^y y, \quad \text{for some } \; C^x, C^y,
\]
which is~\eqref{eqn:xSx-Cx+ySy-Cy=2q2}. 
Finally, we differentiate the relation~\eqref{eqn:xSx-Cx+ySy-Cy=2q2} with $\py$; multiply by $y$; and apply equations~\eqref{eqn:xSx-Cx+ySy-Cy=2q2}, \eqref{eqn:pxSx}, \eqref{eqn:pySy}, as well as the homogeneity properties of $\tilde{q}$ (since $\partial_i q = \partial_i \tilde{q}$) to obtain
\[
xy \py (S^x - C^x) - x (S^x - C^x) + x^2 \px (S^x - C^x) = 0, \qquad \text{by } \; \px (S^x - C^x) = 2 \tilde{q}_x^2.
\]
This shows that $\tilde{S}^x := S^x - C^x$ is homogeneous of degree $1$; by symmetry, $\tilde{S}^y := S^y - C^y$ is as well.

To show that $C^x = C^y = 0$, we consider equation~\eqref{eqn:A''B''-PDE}. 
This fits into the framework of Corollary~\ref{cor:FGsystem} with $(F,G) = (A', B')$ and $u = (\tilde{q}+c)^{-2} \tilde{S}^x + C^x (\tilde{q}+c)^{-2}$ and $v = (\tilde{q}+c)^{-2} \tilde{S}^y + C^y (\tilde{q}+c)^{-2}$.
Note that
\begin{align*}
\partial^{\alpha} ( (\tilde{q} + c)^{-2} \tilde{S}^x) &= \sum_{\beta \leq \alpha} c_{\alpha, \beta} \partial^{\beta} ((\tilde{q}+c)^{-2}) \partial^{\alpha - \beta} \tilde{S}^x \\ 
&= \sum_{\beta \leq \alpha} \sum_{\gamma \leq \beta} c_{\alpha, \beta, \gamma, \delta} (\tilde{q} + c)^{-2 - |\beta| + |\gamma|} \partial^{\alpha - \beta} \tilde{S}^x \prod_{\delta_1 + \dots + \delta_k = \gamma} \partial^{\delta_i} \tilde{q}
\end{align*}
where each term $\partial^{\delta_i} \tilde{q}$ is homogeneous of degree $1 - |\delta_i|$ and $\partial^{\alpha - \beta} \tilde{S}^x$ is homogeneous of degree $1 - |\alpha| + |\beta|$.
An analogous formula holds for $\partial^{\alpha} ( C^x (\tilde{q} + c)^{-2})$.

We use these computations to deduce that $C^x = C^y = 0$, so $q = \tilde{q}, S^x = \tilde{S}^x$ and $S^y = \tilde{S}^y$ are homogeneous of degree one.
Corollary~\ref{cor:FGsystem} ensures that the solvability of equation~\eqref{eqn:A''B''-PDE} is equivalent to $u,v$ satisfying a triple of differential equations with appropriately matched coefficients so that the top-order terms contribute towards summands of the same homogeneity.
This is because in the notation of Corollary~\ref{cor:FGsystem}, if $u,v$ were homogeneous of degree $-1$, then the corresponding coefficient $\Pi_{r=1}^k \partial_x^{p_r} \partial^{q_r}_y u \cdot \Pi_{s=1}^{\ell} \partial_x^{p'_s} \partial_y^{q'_s} v$ would be homogeneous of degree $-N$, where $N := \sum_r (p_r + q_r) + \sum_s (p'_s + q'_s) + (k+\ell)$ is constant.
The terms $C^i (\tilde{q}+c)^{-2}$ contribute higher negative powers, so by the above, the largest power is $- (N + k+\ell)$.
Explicitly computing these expressions in terms of $\tilde{S}^i, \tilde{q}, C^i, c$, matching the homogeneous terms of each degree for each of them, and using the fact that $\partial^{\alpha} q \neq 0$ for any $\alpha$, we deduce that $c=0$.
The remaining terms then simplify to show that $C^x = C^y = 0$ as well, whereby the resolvents $S^x$ and $S^y$ are homogeneous of degree one.
\end{proof}
\begin{corollary}\label{corollary:homogeneous-SxSy-equivalent}
    If $q$ is homogeneous of degree $1$, then the properties~\eqref{eqn:pySxpxSy} -- \eqref{eqn:pySy} for $S^x, S^y$ are equivalent.
\end{corollary}
\begin{proof}
For homogeneous $q$, Lemma~\ref{lemma:Cx=Cy=0} proves that $S^x, S^y$ must also be homogeneous and satisfy~\eqref{eqn:xSx+ySy=2q2}.
Differentiating~\eqref{eqn:xSx+ySy=2q2} in $x,y$ and using Euler's theorem for the homogeneous functions $q, S^x, S^y$ yields
\[
    2 x (\px S^x) + y (\py S^x + \px S^y) = 4 xq_x^2 + 4 y q_xq_y \quad \text{ and }\quad 
    x (\py S^x + \px S^y) + 2 y (\py S^y) = 4 x q_xq_y + 4 yq_y^2.
\]
The equivalence of~\eqref{eqn:pySxpxSy} --~\eqref{eqn:pySy} follows from these relations.
\end{proof}
\begin{proposition}\label{prop:transform-A-B-c}
Given a homogeneous conformal factor $q$ and a pair $(A,B)$, the expressions~\eqref{eqn:Sx-system-solution-general} for $S^x, S^y$ are invariant under the transformation $(\tilde{A}, \tilde{B}) := (A(x) + cx^2, B(y) - cy^2)$ for any $c \in \bR$.
In particular, the triple $(q, A,B)$ satisfies the soliton equations if and only if $(q, \tilde{A}, \tilde{B})$ does. 
\end{proposition}
\begin{proof}
The property~\eqref{eqn:xSx+ySy=2q2} shows that the pair $(\tilde{A}, \tilde{B})$ satisfies equation~\eqref{eqn:A''B''-PDE}, with the same $S^x, S^y$, if $(A,B)$ does.
For the soliton equation~\eqref{eqn:RicciSoliton-Last-Sx}, we use the fact that $x^2 q_{xx} = y^2 q_{yy}$ by the homogeneity,~\eqref{eqn:homogeneous-partials}, to see that $q_{xx} A + q_{yy} B$ remains invariant.
Using Euler's theorem and~\eqref{eqn:xSx+ySy=2q2}, we then obtain
\[
     \left[ S^x q_x \tilde{A} + S^y q_y \tilde{B} \right] - \left[ S^x q_x A + S^y q_y B \right]  = c \left( x^2 S^x q_x - y^2 S^y q_y \right) = c \left( 2x q_x q^2 + x S^x q - 2q^3 \right).
\]
Subtracting the expressions for $(A,B)$ and $(\tilde{A},\tilde{B})$ shows that the right-hand side of~\eqref{eqn:RicciSoliton-Last-Sx} is invariant, so the equation is still satisfied for $(\tilde{A}, \tilde{B})$.
\end{proof}

\subsection{A family of singular steady solitons} \label{subsection:singular-solitons}
We produce a family of singular steady soliton metrics of type~\eqref{eqn:PaperMetric} with conformally cylindrical base, for which $q(x,y) = x^{\alpha} y^{1-\alpha}$.
We show that any local Ricci soliton metric satisfying the assumptions~\eqref{eqn:metric-assumptions} with conformally cylindrical base either has a surface factor $\Sigma_i$ of constant curvature, or must belong to this family. 
\begin{proposition}\label{prop:Fack(S7)}
All Ricci soliton metrics~\eqref{eqn:PaperMetric} with $q(x,y) = x^{\alpha} y^{1-\alpha}$ for $\alpha \neq \frac{1}{2}$ belong to a three-parameter family of steady singular solitons given by
\begin{equation}\label{eqn:singular-soliton-metrics}
\tag{\metricFaCK}
    \begin{split}
    g &= x^{-2 \alpha} y^{-2 (1-\alpha)} \left( \dfrac{dx^2}{F_{\alpha,c,k_1}(x)} +  \dfrac{dy^2}{F_{1-\alpha,-c,k_2}(y)} + F_{\alpha,c,k_1}(x) \, ds^2 + F_{1-\alpha,-c,k_2}(y) \, dt^2 \right), \\
    & \text{where } \; F_{\alpha,c,k}(t) := c t^2 + k t^{\mu_{\alpha} +1}, \qquad \mu_{\alpha} := 2 \alpha^2 / (2 \alpha - 1). 
    \end{split}
\end{equation}
These solitons specialize to the Ricci-flat Schwarzschild metrics from Example~\ref{ex:Einstein} when $\alpha = 0, 1$.
\end{proposition}\label{prop:singular-solitons}
\begin{proof}
    The cases $\alpha \in \{ 0,1 \}$ are treated in Proposition~\ref{prop:d2q=0}, so suppose $\alpha \neq 0,1$.
    Integrating~\eqref{eqn:pySxpxSy} --~\eqref{eqn:pySy} under these assumptions and using their homogeneity by Lemma~\ref{lemma:Cx=Cy=0}, we obtain $S^x = \mu_{\alpha} x^{-1} q^2 + C_0 y$ and $S^y = \mu_{1-\alpha} y^{-1} q^2 - C_0 x$, where we denoted $\mu_{\alpha} := 2 \alpha^2 / (2\alpha - 1)$ as above.
    In this case, equation~\eqref{eqn:A''B''-PDE} becomes
\begin{equation}\label{eqn:singular-PDE}
    (A'' - \mu_{\alpha} x^{-1} A') - (B'' - \mu_{1-\alpha} y^{-1} B') = C_0 x^{- 2 \alpha} y^{- 2(1-\alpha)} (A' y + B' x).
\end{equation}
If $C_0 \neq 0$, denote $\tilde{A}(x) := x^{- 2 \alpha} A'(x)$ and $\tilde{B}(y) := y^{-2(1-\alpha)} B'(y)$.
Taking $\px \py$ in~\eqref{eqn:singular-PDE} implies $x^{2 \alpha} \tilde{A}' = y^{2(1-\alpha)} \tilde{B}' = 2(2 \alpha -1)k$ for some constant $k$, so integrating this gives $A,B$ as
\[
A(x) = - k x^2 + c_1 x^{2 \alpha + 1} + k_{A,0}, \qquad B(y) = k y^2 + c_2 y^{2(1 -\alpha) + 1} + k_{B,0}.
\]
Using this in~\eqref{eqn:singular-PDE} gives a contradiction, since $c_1 c_2 \neq 0$ unless $A,B$ are both quadratic; thus, $C_0 = 0$.
Equation~\eqref{eqn:singular-PDE} therefore splits into ODEs with constant term $2c(1-\mu_{\alpha}) = - 2c(1-\mu_{1-\alpha})$, by $\mu_{\alpha} + \mu_{1-\alpha} = 2$.
These ODEs have solution $A(x) = F_{\alpha,c,k_1}(x) + k_{A,0}$ from~\eqref{eqn:singular-soliton-metrics} and likewise for $B$.
Equation~\eqref{eqn:RicciSoliton-Last-Sx} then has only single-variable terms except $\lambda q^{-2}$, so $\lambda = 0$.
Also, $k_{A,0} (\mu_{\alpha} + 1) = k_{B,0} (\mu_{1-\alpha} + 1) = 0$ and the vanishing is equivalent, since $\mu_{\alpha} +1 = 0$ makes $t^{\mu_{\alpha} +1} = 1$ in $F_{\alpha,c,k_1}(t)$.

Aside from the smooth Ricci-flat Schwarzschild metrics when $\alpha \in \{ 0,1 \}$, the soliton metrics given in equation~\eqref{eqn:singular-soliton-metrics} have singular cone angles when $A$ or $B$ vanish. 
Due to the symmetry of $g$ from~\eqref{eqn:singular-soliton-metrics}, we may assume that $\alpha > \frac{1}{2}$ and examine $A$ as $x \to 0^+$. 
Since $\mu_{\alpha} \geq 2$ for $\alpha > \frac{1}{2}$, the term $F_{\alpha, c, k_1} (x) = cx^2 + O(x^3)$ near $x = 0$ is locally quadratic, so the metric along the surface slice of constant $(y,t)$ is approximated by $g \approx \hat{g} = y^{-2(1-\alpha)}\left(c^{-1}x^{-2 -2\alpha}\,dx^2 + c x^{2 - 2\alpha}\, ds^2\right)$.
To write this $\hat{g}$ in the standard form~\eqref{eqn:generalized-cylinder}, we set $r = \int x^{-1-\alpha}\, dx$ and express $\hat{g} = y^{-2(1-\alpha)} \left( C_1'dr_1^2 + C_1''r_1^{2-\frac{2}{\alpha}}\,ds^2\right)$.
This metric $\hat{g}$ can only be smooth when $\alpha = 1$.
Otherwise, setting $\rho_1 = 1- r_1$ realizes $\hat{g}$ as equivalent to $d\rho^2 + (1 - \rho_1)^{2-\frac{2}{\alpha}} ds^2$, which has boundary singularities resembling those of the (\hyperref[cor:SexBuffet]{\metricSexBuffet}) metric, as illustrated in Figure~\ref{fig:geometricBoundary}.

For the asymptotics of the metric on the surface slice of constant $(x,s)$, as $y \to 0^+$, the term $y^{\mu_{1-\alpha} + 1} = y^{3 - \mu_{\alpha}}$ dominates in $F_{1-\alpha, -c, k_2}(y)$.
An analogous change of coordinates realizes the metric as asymptotic to $x^{-2\alpha}(C_2'd r_2^2 + C_2'' r_2^{\frac{6 (2 \alpha^2 - 1)}{2(3 \alpha - 2)^2 + 1}}dt^2)$, which likewise has boundary singularities for $\alpha < \frac{1}{\sqrt{2}}$, edge singularities for $\alpha \in ( \frac{1}{\sqrt{2}}, \infty) \setminus \{ 1 \}$ since the exponent of $r_2$ is positive $\neq 2$, and closes up smoothly when $\alpha = \frac{1}{\sqrt{2}}$ or $\alpha = 1$, the latter being the Schwarzschild metric.
The singular behavior as $(x,y) \to (0,0)$ is more complicated as these boundary singularities collide at the origin, the particular limiting behavior depending on the path taken.
\end{proof}
\begin{claim}
    When $q(x,y) = \sqrt{xy}$, there are no admissible soliton metrics of the form~\eqref{eqn:PaperMetric}.
\end{claim}
\begin{proof}
Clearing denominators by $q^2 = xy$ in~\eqref{eqn:A''B''-PDE}, applying $\px^2 \py^2$, and using~\eqref{eqn:p2ySx} gives $x A^{(3)} = y B^{(3)}$, so $A(x) = kx^2 \log x + Q_1(x), B(y) = ky^2 \log y + Q_2(y)$ for quadratics $Q_1(x) = \sum_{j=0}^2 a_j x^j$ and $Q_2$.
    We also obtain $S^x, S^y$ by integrating~\eqref{eqn:pySxpxSy} --~\eqref{eqn:pySy}, as
    \[
2 S^x = \left(  y \log(x/y) + \left( C_0 + 3 \right) y + C_2 \right), \qquad 2 S^y = \left( x \log (y/x) - \left( C_0 - 1 \right) x - C_1 \right).
\]
This results in a term $- kxy( \log x)^2$ in the expression~\eqref{eqn:A''B''-PDE} that does not cancel unless $k=0$.
The appearance of terms $x \log (x/y)$ and $y \log(x/y)$ in~\eqref{eqn:A''B''-PDE} forces $a_1 = b_1 = 0$; $b_2 = - a_2 = \ell$; and $C_1 = C_2 = 0$ or $\ell = 0$.
In either case, the equation resulting from~\eqref{eqn:RicciSoliton-Last-Sx} has a term $\left( a_0 x^{-2} -  b_0 y^{-2} \right) \log (x/y)$ not canceled by any other terms, unless $(a_0,b_0) = (0,0)$.
This is impossible for $\ell=0$ (then $A,B \equiv 0$), and for $\ell \neq 0$, there is another unmatched term $- \frac{1}{2} \ell xy(xy-1) \log(x/y)$ implying there are no solutions.
\end{proof}

\begin{proposition}\label{prop:homogeneous-rigidity}
    If $q$ is homogeneous with $D^2 q \neq 0$ and $A,B$ are not both quadratic, then $q = x^{\beta} y^{1-\beta}$ for some $\beta$ and the corresponding metric $g$ is~\eqref{eqn:singular-soliton-metrics}, given by Proposition~\ref{prop:Fack(S7)}. 
\end{proposition}
\begin{proof}
By Lemma~\ref{lemma:Cx=Cy=0}, $S^x, S^y$ are homogeneous of degree $1$, so the functions $\frac{S^x}{q^2}, \frac{S^y}{q^2}$ are homogeneous of degree $-1$.
First, we show that if neither of $A,B$ is a quadratic, they satisfy
    \begin{equation}\label{eqn:A-B-mu}
    \dfrac{A^{(4)} x}{\px (A'/x)} = \mu = \dfrac{B^{(4)}y}{\py (B'/y)}, \qquad \text{for some } \; \mu \neq 0.
    \end{equation}
Furthermore, if one of $A$ or $B$ is quadratic, then either both are, or $q = x^{\beta} y^{1-\beta}$ for some $\beta$ as in Proposition~\ref{prop:singular-solitons}.
The equations~\eqref{eqn:A''B''-PDE} and~\eqref{eqn:xSx+ySy=2q2} give a $2 \times 2$ system for the functions $\frac{S^x}{q^2}, \frac{S^y}{q^2}$.
\begin{equation}\label{eqn:Homogeneous-SolveForSxSyq2}
        \dfrac{S^x}{q^2} = \dfrac{2B' + y(A'' - B'')}{A' y + B' x}, \qquad \dfrac{S^y}{q^2} = \dfrac{2A' - x(A'' - B'')}{A' y + B'x}.
\end{equation}
If $A,B$ are not quadratics, the denominator does not vanish and neither do $\px (A'/x), \py (B'/y)$.
Euler's formula shows that if $u/v$ is homogeneous of degree $-1$, then
\begin{equation}\label{eqn:uvhomdeg-1}
        u \left( x \px v + y \py v \right) - \left( x \px u + y \py u \right) v = uv.
\end{equation}
Applying this to the expression for $S^y/q^2$ from~\eqref{eqn:Homogeneous-SolveForSxSyq2} shows that 
\begin{equation}\label{eqn:Syq2-after-Euler}
    \dfrac{S^y}{q^2} = \dfrac{1}{y} - \dfrac{1}{y} \dfrac{xA^{(3)} - yB^{(3)}}{A'' + B''},
\end{equation}
whereby equating the expressions~\eqref{eqn:Homogeneous-SolveForSxSyq2} and~\eqref{eqn:Syq2-after-Euler} for $S^y/q^2$ implies
\begin{equation}\label{eqn:Homogeneity-A''-B''}
    (A'' + B'') \left[ (A'y - B' x) - xy (A'' - B'') \right] + (A' A^{(3)} - B' B^{(3)}) xy + A^{(3)} B' x^2 - A' B^{(3)} y^2 = 0.
\end{equation}
Denote this expression by $\cL$.
The computation $\px \py \left( \cL / (xy) \right) = 0$ is equivalent to
\[
A^{(4)} x \py (B'/y) = \px (A'/x) B^{(4)} y,
\]
which implies~\eqref{eqn:A-B-mu} provided that $\px (A'/x), \py(B'/y)$ are not $0$. 
In this case, $\mu \neq 0$; otherwise $A^{(4)} = B^{(4)} = 0$ would require $A,B$ to be cubics $\sum a_i x^i, \sum b_j y^j$. 
Equation~\eqref{eqn:Syq2-after-Euler} would then imply $a_2 + b_2 = 0$ and $S^y = 2 q^2/(\alpha x + y)$. 
The relation $\py S^y = 2 q_y^2$ forces $q_y = q/(\alpha x + y)$, and likewise for $q_x$.
This requires $(q,A,B)$ to correspond to the \pbnski metric~\eqref{metric:pbnski}, ruled out by $D^2 q \neq 0$.

If $\px (A'/x) \py (B'/x) = 0$, by symmetry we can write $A(x) = c x^2 + a_0$, so either $B(y) = - cy^2 + b_0$ or $\dfrac{S^y}{q^2} = \dfrac{2 c +B''}{2 c y+B'}$ is defined.
Equation~\eqref{eqn:Homogeneity-A''-B''} simplifies to a second-order ODE for $B'(y)$ with solution $B'(y) = - 2cy + \tilde{c}_2 y^{\tilde{c}_1}$ for some $\tilde{c}_i$.
This gives $S^y = \tilde{c}_3 q^2/y$, so $S^x = (2 - \tilde{c}_3) q^2/x$.
By~\eqref{eqn:pySxpxSy} --~\eqref{eqn:pySy} and the homogeneity relations, this implies that $q = x^{\beta} y^{1-\beta}$, as treated in Proposition~\ref{prop:singular-solitons}.

If neither $A$ nor $B$ is quadratic, equation~\eqref{eqn:A-B-mu} must hold for $\mu \neq 0$.
The equations~\eqref{eqn:pxSx} and~\eqref{eqn:pySy} yield
\begin{equation}\label{eqn:Sxq-quadratic}
    \left( \dfrac{S^i}{q^2} \right)^2 + 2 \, \partial_i \left( \dfrac{S^i}{q^2} \right) = \left( \dfrac{S^i}{q} - 2 \dfrac{q_i}{q} \right)^2, \qquad i \in \{ x,y\},
\end{equation}
so $q_x/q, q_y/q$ can be computed using the quadratic formulas.
Using~\eqref{eqn:Syq2-after-Euler}, define
\begin{equation}\label{eqn:define-F-x,y}
F(x,y) := \dfrac{xA^{(3)} - yB^{(3)}}{A'' + B''}, \quad \text{so } \quad \; \dfrac{S^x}{q^2} = \dfrac{1+F}{x}, \quad \dfrac{S^y}{q^2} = \dfrac{1-F}{y}.
\end{equation}
Solving~\eqref{eqn:Sxq-quadratic} for $q_x/q$ and $q_y/q$ gives, for some sign $\epsilon = \pm 1$,
\begin{equation}\label{eqn:syq2-u-soln}
    \dfrac{q_x}{q} = \dfrac{1}{2x} \left( 1 + F + \epsilon \sqrt{F^2 + 2 x \px F - 1} \right), \quad \dfrac{q_y}{q} = \dfrac{1}{2y} \left( 1 - F - \epsilon \sqrt{F^2 + 2 x \px F - 1} \right).
\end{equation}
It follows from these formulas that $F^2 + 2 x \px F - 1 \geq 0$.

Equations~\eqref{eqn:A-B-mu} for $A, B$ can be integrated explicitly, with solutions
\begin{equation}\label{eqn:A''''-solutions}
    A(x) = a_0 + a_2 x^2 + a_+ x^{2 + \alpha} + a_- x^{2 - \alpha}, \quad \alpha^2 = \mu + 1, \qquad \text{for } \; \mu \neq - 1
\end{equation}
and likewise for $B(y)$.
For $\mu = 3$, the solutions specialize to $A(x) = a_0 + a_2 x^2 + a_+ x^4 + a_- \log x$; for $\mu = -1$, they specialize to $A(x) = a_0 + x^2( a_2 + a_+ (\log x)^2 + a_- \log x)$.
For $\mu + 1 <0$, $\alpha = i \tilde{\alpha}$ is purely imaginary, so $x^{2 \pm \alpha} = x^2 \left[ \cos (\tilde{\alpha} \log x) \pm i \sin (\tilde{\alpha} \log x) \right]$ still satisfies the power rule for derivatives.
In the case $\mu = - 1$, computing~\eqref{eqn:Syq2-after-Euler} and taking $(x,y) \mapsto (tx,ty)$ requires the $(\log t)^2$ term to vanish and the expression to equal its limit as $t \to + \infty$, whereby $a_+ = - b_+ = k$ and $4k(a_2 + b_2) = a_-^2 - b_-^2$.
Then, $F$ from~\eqref{eqn:define-F-x,y} is given as $F(x,y) = \frac{4k}{a_- + b_- + 2k \log(x/y)}$, so $F^2 + 2 x \px F = 0$, contradicting~\eqref{eqn:syq2-u-soln}.
Therefore, $\mu = -1$ admits no solutions.

In all other cases, the formula~\eqref{eqn:A''''-solutions} holds and may be differentiated to yield
\begin{equation}\label{eqn:A''-A'''-homogeneous}
A''(x) = 2 a_2 + \tilde{a}_+ x^{\alpha} + \tilde{a}_- x^{- \alpha}, \qquad A^{(3)}(x) = \alpha \left( \tilde{a}_+ x^{\alpha - 1} - \tilde{a}_- x^{-\alpha - 1} \right).
\end{equation}
where $\tilde{a}_{\pm} := (2 \pm \alpha)(1 \pm \alpha) a_{\pm}$, same for $\tilde{b}_{\pm}$.
For $\mu = 3$, we instead obtain $\tilde{a}_- = - a_-$, so $\tilde{a}_i \neq 0$ if $a_i \neq 0$ still.
If $A,B$ are not both quadratic, we have that $(a_+,b_+)$ and $(a_-,b_-)$ in~\eqref{eqn:A''''-solutions} are not both $(0,0)$; by the symmetry in $\pm \alpha$, let $(a_+,b_+) \neq (0,0)$.
Since $S^y/q^2$ is homogeneous of degree $-1$, the equality~\eqref{eqn:Syq2-after-Euler} implies that $(xA^{(3)} - yB^{(3)})/(A'' + B'')$ is homogeneous of degree $0$.
We can rewrite this expression using~\eqref{eqn:A''-A'''-homogeneous}; taking $(x,y) \mapsto (tx,ty)$ here and examining the powers of $t^{\pm \alpha}$ gives $\tilde{a}_+ \tilde{a}_- = \tilde{b}_+ \tilde{b}_-$, so $a_+ a_- = b_+ b_- = k$ for some $k$.
Comparing the constant terms gives $a_2 = -b_2 =: c$.

For the remainder of the proof, we will define $X := a_+ x^{\alpha}, Y := b_+ y^{\alpha}$, so $\px X = \frac{\alpha X}{x}, \py Y = \frac{\alpha Y}{y}$.
Using $a_2 + b_2 = 0$, $a_+ a_- = b_+ b_- = k$, and~\eqref{eqn:A''-A'''-homogeneous}, we rearrange $F$ into
\begin{align*}
    F(x,y) &= \dfrac{x A^{(3)} - yB^{(3)}}{A'' + B''} = \alpha \dfrac{(X - kX^{-1}) - (Y - k Y^{-1})}{X + k X^{-1} + Y + k Y^{-1}} = \alpha \dfrac{X-Y}{X+Y}, \\
    & \implies F^2 + 2 x \px F -1 = \alpha^2 \dfrac{(X-Y)^2}{(X+Y)^2} + \alpha^2 \dfrac{4XY}{(X+Y)^2} - 1 = \alpha^2 - 1
\end{align*}
This implies that $|\alpha| \geq 1$ and integrating~\eqref{eqn:syq2-u-soln} shows
\begin{equation}\label{eqn:q-general-solution}
    q = \sqrt{xy} (x/y)^{\frac{1}{2} \epsilon \sqrt{\alpha^2 - 1}} (xy)^{- \alpha / 2} \left(a_+ x^{\alpha} + b_+ y^{\alpha} \right), \qquad \epsilon = \pm 1.
\end{equation}
The last term is simply $X+Y \neq 0$, by $(a_+,b_+) \neq (0,0)$.
For $\alpha = \pm 1$, the function $q$ from~\eqref{eqn:q-general-solution} becomes linear, giving the \pbnski solution, so $|\alpha| > 1$ for $D^2 q \neq 0$.
Observe that $q$ specializes to $x^{\beta} y^{1-\beta}$ as in Proposition~\ref{prop:singular-solitons} when $a_+ b_+ = 0$ or $a_- b_- = 0$, so $k=0$ above.
We show that this the only solution.
For any other solution, we have $XY \neq 0$, so~\eqref{eqn:A''''-solutions} becomes
\begin{equation}\label{eqn:A,B-homogeneous-form}
    A(x) = a_0 + x^2 \left( c + X + k X^{-1} \right), \qquad B(y) = b_0 + y^2 \left( - c + Y + k Y^{-1} \right)
\end{equation}
and by~\ref{prop:transform-A-B-c}, we may take $c=0$.
In terms of $X,Y$,~\eqref{eqn:q-general-solution} gives $q^{-2} = (xy)^{-1} (y/x)^{\epsilon \sqrt{\alpha^2 - 1}} \frac{XY}{(X+Y)^2}.$
Having computed $q,A,B$ in~\eqref{eqn:q-general-solution} and~\eqref{eqn:A,B-homogeneous-form} as well as $S^x, S^y$ using~\eqref{eqn:Syq2-after-Euler} and $F = \alpha \frac{X-Y}{X+Y}$, we can explicitly write out the last soliton equation~\eqref{eqn:RicciSoliton-Last-Sx}.
Clearing the denominator $(X+Y)^2$ and using $\alpha \geq 1$, the resulting equation has single-variable monomials of highest degree $\alpha(2+\alpha)(X^2 + Y^2)$ which are not matched, so the equation has no solution for $XY \neq 0$. 
This means that $A,B$ are both quadratics, or $q$ is given by~\eqref{eqn:singular-soliton-metrics}.
\end{proof}

\subsection{Conformally flat and scalar flat solitons}\label{subsect:solitons-conformal-S2H2}
The conformally cylindrical framework is motivated by studying metrics conformal to a product of two surfaces of constant curvature.
We show that the curvatures of the two surfaces must be opposite, so the metric is locally conformally flat.
This furthermore implies that the conformal factor $q$ is homogeneous of degree $1$, and the Ricci soliton equation is equivalent to a third-order ODE for $f$ with $q(x,y) = yf(x/y)$.
For steady and shrinking solitons, such examples must have boundary or be singular, by the classification of \cite{catino-mantegazza-mazzieri}.
We produce such soliton metrics, as well as complete locally conformally flat expanding solitons not covered by the above classification.

The metric~\eqref{eqn:PaperMetric} in this context is expressed with $A$ and $B$ quadratics. 
Up to shifting, we can assume the linear term vanishes so $A = ax^2 + a_0$ and $B = by^2 + b_0$ where $-2a$ and $-2b$ are the Gauss curvatures of the surfaces.
Equation~\eqref{eqn:A''B''-PDE} explicitly yields
\begin{equation}\label{eqn:abSxSy-quadratics}
(a - b)q^2 = a xS^x - b y S^y
\end{equation}
and differentiating twice by $\px^2$ and $\py^2$ and using equations~\eqref{eqn:pySxpxSy} --~\eqref{eqn:p2ySx} further shows that 
\begin{equation}\label{eqn:CPSF-two-cases}
2 \left( a x q_x - by q_y \right) - (a-b) q =
- (a+b) \dfrac{q_x^2}{q_{xx}} =  (a+b) \dfrac{q_y^2}{q_{yy}}
\end{equation}
after rearranging the terms and dividing by $q_{xx}$ and $q_{yy}$.
This requires either $a + b = 0$ or $\frac{q_{xx}}{q_x^2} = -\frac{q_{yy}}{q_y^2}$.
By~\eqref{eqn:HessianCompatibility}, the latter case implies that $D^2q$ has zeroes, which is handled in Proposition~\ref{prop:d2q=0}.

When $a + b = 0$, the local model of the product without the conformal factor is either $\bR^4$, for $a = b = 0$, or $\bS^2 \times \bH^2$, by rescaling according to expression~\eqref{eqn:scale-rescale-translate}, so that the Gauss curvature of each factor is $\pm 1$. 
If a conformal factor $q(x,y)$ is admissible for $\{ a, b\} = \{ -1 , 1 \}$, then we may apply Proposition~\ref{prop:transform-A-B-c} and rescale $x,y$ to obtain conformally flat metrics with the same $q$,
\begin{equation}\label{eqn:ConformallyFlat}
    g = \frac{1}{q(x,y)^2}\left(dx^2 + dy^2 + ds^2 + dt^2 \right).
\end{equation}
This situation, corresponding to $a = b = 0$ was previously analyzed in Section~\ref{subsect:conformally-flat}; we now focus on $\{ a, b \} = \{ -1, 1 \}$, so the resulting metric is conformally scalar flat and given by
\begin{equation}\label{eqn:CPSFMetric}
    g = \frac{1}{q(x,y)^2}\left(\frac{dx^2}{a_0 - x^2} + (a_0 - x^2 )\, ds^2 +\frac{dy^2}{y^2 + b_0} + (y^2 + b_0)\, dt^2 \right)
\end{equation}
Here, equation~\eqref{eqn:abSxSy-quadratics} gives $xS^x + yS^y = 2q^2$, whereby $S^x, S^y$ are homogeneous by Lemma~\ref{lemma:Cx=Cy=0}.
By Corollary~\ref{corollary:homogeneous-SxSy-equivalent}, the conditions~\eqref{eqn:pySxpxSy} --~\eqref{eqn:pySy} are all equivalent to each other, so it suffices to ensure $\px S^x = 2 q_x^2$.
By Proposition~\ref{prop:transform-A-B-c}, we may consider $(A,B) = (a_0 - cx^2, b_0 + cy^2)$, where $c$ is arbitrary.
We can solve equations~\eqref{eqn:A''B''-PDE} and~\eqref{eqn:RicciSoliton-Last-Sx} as a $2 \times 2$ system in $S^x, S^y$, unless they are linearly dependent.
As computed in~\eqref{eqn:Sx-system-solution-general}, linear dependence is equivalent to $AB' q_x + A'B q_y - \frac{1}{2} A' B' q = 0$.
Substituting $A$ and $B$ and using $x q_x + y q_y = q$ transforms this equality into $\frac{q_x}{q_y} = \frac{b_0 x}{a_0 y}$, so $q = \sqrt{b_0 x^2 + a_0 y^2}$.
In this case, equation~\eqref{eqn:RicciSoliton-Last-Sx} becomes an identity with $\lambda = - 2a_0 b_0$, due to $x S^x + y S^y = 2q^2$.
We therefore obtain the soliton metric
\begin{equation}\label{metric:q=sqrtx2+y2}
\tag{\metricQIsSqrtYSqXSq}
g = \frac{1}{a_0y^2 + b_0 x^2} \left( \frac{dx^2}{a_0-x^2} + (a_0-x^2) \, ds^2 + \frac{dy^2}{y^2 + b_0} + (y^2 + b_0) \, dt^2 \right).
\end{equation}
Since $a_0 > 0$ and $c \geq 0$, this metric corresponds to a complete expanding soliton for $b_0 > 0$; a steady Riemannian Schwarzschild metric with surface factors degenerating to quadratics for $b_0 = 0$ (i.e., $q = \sqrt{a_0} y$); and a shrinking soliton for $b_0 < 0$.
By~\eqref{eqn:SxSy}, the soliton equation is satisfied by $V$ with
\[
V = \left( \tilde{c} - 2 \epsilon \sqrt{a_0 |b_0|} \cdot \arctan \left( \frac{\sqrt{|b_0|} x}{\sqrt{a_0} y} \right) \right) V_0, \qquad V_0 = (a_0 - cx^2) y \, \px - (b_0 + cy^2) x \, \py
\]
for $\epsilon = \on{sgn}(b_0)$.
Note that the vector field $V_0$ is a Killing field, so changing $\tilde{c}$ does not alter $\cL_V g$.

Otherwise, Corollary~\ref{corollary:homogeneous-SxSy-equivalent} guarantees that conditions~\eqref{eqn:pySxpxSy} --~\eqref{eqn:RicciSoliton-Last-Sx} for $q$ are reduced to $\px S^x = 2 q_x^2$, where $S^x$ is obtained by solving the $2 \times 2$ system~\eqref{eqn:A''B''-PDE} and~\eqref{eqn:RicciSoliton-Last-Sx} in $S^x, S^y$. 
Solving the resulting $2 \times 2$ system from equation~\eqref{eqn:Sx-system-solution-general}, we compute
\[
S^x = \dfrac{2 b_0 q_y q^2 + y \left( \lambda q + b_0 q q_y^2 - b_0 q^2 q_{yy} + a_0 q q_x^2 - a_0 q^2 q_{xx}  \right)}{b_0 x q_y - a_0 y q_x}
\]
The condition $\px S^x = 2 q_x^2$ therefore amounts to $\cF (q, Dq, D^2 q, D^3 q; a_0, b_0, \lambda) = 0$, where we define
\begin{equation}\label{eqn:f-q-dq-d2q-d3q-a,b,l}
\begin{split}
    \cF (q,Dq, D^2 q, D^3 q; a_0, b_0, \lambda) &:= - 3 y^2 \left( a_0 q_x^2 + b_0 q_y^2 \right)^2 - \lambda y \left[ y \left( a_0 q_x^2 + b_0 q_y^2 \right) + x^{-1} (a_0 y^2 + b_0 x^2) q q_{xy} \right] \\
    & \; \; \; - 5 a_0 b_0 x^{-1} y q^3 q_{xy} + x^{-2} ( a_0 y^2 + b_0 x^2 )^2 q q_{xy} \left( q_x q_y - q q_{xy} \right) \\
    & \; \; \; + x^{-1} (a_0 y^2 + b_0 x^2 ) \left( b_0 x q_y - a_0 y q_x \right) q^2 q_{xxy}.
\end{split}
\end{equation}
Since $q$ is homogeneous, we may write $q = yf( \frac{x}{y})$ and express the derivatives of $q$ in terms of $z := \frac{x}{y}$, where 
\begin{equation}\label{eqn:q-derivatives-z-f}
\px z = \frac{1}{y}, \quad \py z = - \frac{z}{y}, \qquad \partial^{\alpha}_x \partial^{\beta}_y q = y^{1 - (\alpha + \beta)} \sum_{k=0}^{\beta} c_k z^k f^{(k+\alpha)}(z).
\end{equation}
The condition $\cF(D^{0 \leq i \leq 3}q; a_0, b_0, \lambda) = 0$ is therefore transformed into the ODE
\begin{equation}\label{eqn:resulting-ODE}\tag{$\star_{a_0,b_0,\lambda}$}
\begin{split}
    0= & - 3 \left( a_0 (f')^2 + b_0(f-zf')^2 \right)^2 + \lambda \left[ - \left( a_0 (f')^2 + b_0 (f - zf')^2 \right) + (a_0 + b_0 z^2) f f'' \right] \\
    & + (4a_0 -b_0 z^2) b_0 f^3 f'' + \left( a_0 + b_0 z^2 \right)^2 f f'' \left( (f')^2 - f f'' \right) - b_0 z ( b_0 z f - (a_0 + b_0 z^2) f') f^2 f'' \\
    & - \left( a_0 + b_0 z^2 \right)
    \left( b_0 zf - (a_0 + b_0 z^2) f' \right) f^2 f^{(3)}.
\end{split}
\end{equation}
A solution $f$ gives a homogeneous factor $q = y f(x/y)$ such that the associated metric and vector field satisfy the soliton equation.
The function $f(x) = q(x,1)$ is the stereographic projection of the homogeneous function onto a line.
In order for $q$ to be well-defined on the boundary points, we need the limits $\lim_{z \to \pm \infty }\frac{f(z)}{z}$ to exist and be equal, which corresponds to rescaling the line to the circle and approaching the poles.
The equation is symmetric in $(x,y) \leftrightsquigarrow (-x,-y)$.
If the limits $\lim_{z \to \pm \infty} \frac{f(z)}{z}$ agree, the function $q = yf(x/y)$ is well-defined on all of $\bR^2$ since the points $(1,0)$ and $(-1,0)$ correspond to these limits. 

By expressing the ODE as a system of three first-order ODEs, we can use the Picard-Lindel\"of theorem to show that a unique solution exists for short time, given initial data, and can be extended maximally. 
The coefficients are given by smooth combinations of $f,f',f''$, so as long as these functions are smooth, the coefficients of the system will be smooth, and satisfy the conditions to apply short-time existence and uniqueness.
If the third derivative $f^{(3)}$ exists, there is no obstruction to local existence and uniqueness, so we know that $f$ is smooth where $f^{(3)}$ exists.

When $f$ solves~\eqref{eqn:resulting-ODE}, we therefore obtain Ricci soliton metrics of the form
\begin{equation}\label{metric:SStar}\tag{S$\bigstar$}
    g = \frac{1}{y^2f(x/y)^2}\left( \frac{dx^2}{a_0-x^2} + (a_0-x^2) \, ds^2 + \frac{dy^2}{y^2 + b_0} + (y^2 + b_0) \, dt^2 \right)
\end{equation}
In the following section, we analyze particular instances of~\eqref{metric:SStar} and exhibit families of complete expanding solitons as well as shrinking and steady solitons with boundary that specialize to known examples. 
The analytic properties of $f$ solving~\eqref{eqn:resulting-ODE} determine various geometric behaviors of the resulting solitons.
The soliton metric obtained in~\eqref{metric:q=sqrtx2+y2} is equivalent to the explicit solution of the ODE~\eqref{eqn:resulting-ODE} given by $f(z) = \sqrt{a_0 + b_0 z^2}$ and $\lambda = - 2 a_0 b_0$.

The above results can be collected into the following statement:
\begin{corollary}\label{corollary:q-a-b-quadratics}
    If $q$ is homogeneous and $A, B$ are both quadratic, then up to translating $x,y$, they have the form $(a_0 - cx^2, b_0 + cy^2)$.
    The conformal factor $q$ is locally expressible as $q = y f(x/y)$, where $f$ solves the equation~\eqref{eqn:resulting-ODE}.
    The existence of such $q$ is independent of the leading coefficient $c$ of $A,B$.
\end{corollary}

The study of homogeneous factors $q$ also produces a family of conformally Euclidean metrics $g = q^{-2} g_{\on{Euc}}$ satisfying the conditions of Lemma~\hyperref[lemma:conformally-flat]{\ref{lemma:conformally-flat}(iii)}.
This follows from taking $a_0 = b_0 = 1$ and $c=0$ in the above result, whereby for $f(z)$ any solution of the ODE (\hyperref[eqn:resulting-ODE]{$\star_{1,1,\lambda}$}) and arbitrary $x_0, y_0$, defining $q = (y-y_0) f \left( \frac{x-x_0}{y-y_0} \right)$ produces a valid conformal factor for a Ricci soliton metric $g = q^{-2} g_{\on{Euc}}$.

\subsection{Shrinking and steady solitons with boundary}\label{subsec:shrinkingBoundarySolitons}

By choosing appropriate initial conditions for $f$, we can produce examples of shrinking and steady solitons with boundary by solving the ODE~\eqref{eqn:resulting-ODE} to ensure that the solution $f$ has two zeroes.
These solitons correspond to Figure~\hyperref[fig:solitonsDomainGeometryEx]{\ref{fig:solitonsDomainGeometryEx}(b)} or~\hyperref[fig:solitonsDomainGeometryEx]{(c)}.
Figure~\ref{fig:geometricBoundary} illustrates their geometric behavior near the boundary, where the curvature blows up.

Ricci solitons with boundary can be viewed as self-similar solutions of the Ricci flow on manifolds with boundary. 
A key difficulty of this problem is prescribing meaningful boundary conditions for the Ricci flow.
The work of Shen \cite{shen} yields convergence results when the metric is rotationally symmetric or has positive Ricci curvature, and the boundary satisfies additional conditions.
The general case of surfaces was studied by Brendle \cite{brendle-boundary}.
Kunikawa-Sakurai \cite{kunikawa-sakurai} recently obtained analogs of classification results for Ricci solitons with boundary.
Higher-dimensional Ricci flow with boundary has been studied by Chow \cite{chow} and Gianniotis \cites{gianniotis-1, gianniotis-2}.
In our case, the creation of a boundary along which the curvature blows up is similar to the behavior studied in \cite{singular-solitons}.
This metric singularity of the curvature demonstrates that the boundary encloses a maximal domain on which the soliton equation can be solved.
The Riemannian metric on the interior does not descend to the topological boundary, but the metric completion topologically agrees with the boundary charts and gives the boundary its own metric, as described in Corollary~\ref{cor:completeness/Singularity}.

\begin{proposition}\label{prop:SexBuffet}
    For $a_0, b_0 > 0$ and $\lambda \geq 0$, there exist positive solutions $f$ to the ODE~\eqref{eqn:resulting-ODE} and some $L_1 < 0 < L_2$ such that $f \in C^\infty(L_1,L_2) \cap C^{0}[L_1, L_2]$ and $f(L_1) = f(L_2) = 0$.
\end{proposition}
\begin{corollary}[Shrinking and steady solitons with boundary]\label{cor:SexBuffet}
    For $a_0, b_0 > 0$, there exist shrinking and steady solitons with boundary $\mr{(\metricSexBuffet)}$ given by~\eqref{metric:SStar} for $f$ as in Proposition~\ref{prop:SexBuffet} defined in the domain $\{|x| < a_0\} \cap \{y > L_1x\} \cap \{y > L_2x\}$.
    Using Proposition~\ref{prop:transform-A-B-c}, we similarly have shrinking and steady solitons with boundary on the domain $\{\tilde{y} > \tilde{L_1}\tilde{x}\} \cap \{\tilde{y} > \tilde{L_2}\tilde{x}\}$ where  $(\tilde{x},\tilde{y},\tilde{s},\tilde{t}) = \left(\frac{x}{\sqrt{a_0}}, \frac{y}{\sqrt{b_0}},\sqrt{a_0}s, \sqrt{b_0}t\right)$ and $\tilde{L}_i = L_i \sqrt{a_0/b_0}$.
    The former are the conformally scalar flat examples, and the latter are conformally flat.
\end{corollary}
\begin{proof}[Proof of Proposition~\ref{prop:SexBuffet}]
    We will show that the solution $f$ to ODE~\eqref{eqn:resulting-ODE} with initial condition $f(0) = C_1 > 0, f'(0) = 0$, and $f''(0) = -C_2 < 0$ is strictly concave ($f'' < 0$) on the maximal domain of definition.
    Since $f'(0) = 0$, $f$ attains an interior maximum, so by concavity, it has two zeroes $L_1, L_2$ with $L_1 < 0 < L_2$.

    Assume that $f$ is not strictly concave, so there exists some point $\zeta$ of minimal absolute value $|\zeta| \neq 0$ where $f''(\zeta) = 0$. 
    By minimality, $f''(z) < 0$ for all $|z| < |\zeta|$.
    In this region, we know that $f$ is strictly positive; otherwise $f$ must have crossed zero and we would be done. 
    Since $f'' < 0$, $f'$ is decreasing, so $f'(z) < 0$ for $z > 0$ and $f'(z) > 0$ for $-| \zeta| < z < 0$.
    Notably, in the region  $|z| \leq |\zeta|$, $z$ and $f'$ have opposite, non-zero signs:
    \begin{equation}\label{eqn:signOff'=sgnz}
         \on{sgn} f'(z) = - \on{sgn} z, \qquad |z| < |\zeta|.
    \end{equation}  
    
    We first claim that $f^{(3)}(z)$ exists for $|z| < |\zeta|$, so $f$ exists up to $\zeta$ and is smooth.
    $f^{(3)}$ is given by
    \begin{equation}\label{eqn:f'''}
    \begin{split}
            f^{(3)}(z) = &\frac{1}{(a_0 + b_0 z^2)(b_0z f - (a_0 + b_0z^2)f')f^2}\bigg[ - 3 \left( a_0 (f')^2 + b_0(f-zf')^2 \right)^2 - \lambda \left( a_0 (f')^2 + b_0(f-zf')^2 \right) \\
    & + (4a_0 -b_0 z^2) b_0 f^3 f'' + \left( a_0 + b_0 z^2 \right)^2ff''((f')^2- ff'') + \lambda \left( a_0 + b_0 z^2 \right) f f'' \bigg] -\frac{b_0z}{a_0 + b_0z^2}f'',
    \end{split}
    \end{equation}
    so while $f, f', f''$ exist, $f^{(3)}$ blows up only if $b_0zf = (a_0 + b_0z^2)f'$, since the other terms in the denominator are positive.
    This cannot occur for $a_0,b_0 > 0$ because $f > 0$ and $\on{sgn} z = - \on{sgn} f'$ by the previous observation. 

    At $\zeta$, equation~\eqref{eqn:f'''} reduces to 
    \begin{equation}
         f^{(3)}(\zeta) = \frac{1}{(a_0 + b_0 \zeta^2)(b_0\zeta f - (a_0 + b_0\zeta^2)f')f^2}\bigg[ - 3 \left( a_0 (f')^2 + b_0(f-\zeta f')^2 \right)^2 - \lambda \left( a_0 (f')^2 + b_0(f-\zeta f')^2 \right)\bigg]
    \end{equation}
    and since $\lambda \geq 0$, then the bracketed term is strictly negative. 
    Therefore,
    \begin{equation}\label{eqn:signs}
    \on{sgn} f^{(3)}(\zeta) = \on{sgn} \left(  (a_0 + b_0\zeta^2)f'(\zeta) -b_0 \zeta f(\zeta)  \right).
    \end{equation}
    If $\zeta > 0$, then $f''(z) <0$ for $z < \zeta$, so $f^{(3)}(\zeta) \geq 0$.
    Conversely, if $\zeta < 0$, then $f''(z) <0$ for $z > \zeta$, so $f^{(3)}(\zeta) \leq 0$. 
    From equation~\eqref{eqn:signs}, we can rewrite the above two inequalities as 
    \[
    \on{sgn} \left( (a_0 + b_0 \zeta^2)f' - b_0 \zeta f \right) = \on{sgn} \zeta.
    \] 
    This sign equation violates the sign inequality from equation~\eqref{eqn:signOff'=sgnz}, contradicting the vanishing of $f''$.
    Therefore, $f$ is strictly concave.
    
    It remains to be shown that $f$ extends from $(-\ve,\ve)$ to $[L_1, L_2]$, for some $L_1 < 0 < L_2$ with $f(L_i) = 0$.
    This can only fail if one of the derivatives of $f$ blows up before $f$ vanishes.
    We show that $f^{(3)}$ does not blow up, whereby $f',f''$ do not either.
    This comes from equation~\eqref{eqn:f'''} for $f^{(3)}$: while $f \neq 0$ is still finite, this would mean that $b_0 z f = (a_0 + b_0 z^2)f'$, which cannot occur while $f > 0$.
    Therefore, $f^{(3)}$ exists while $f$ is positive, so $f$ continues to exist until it reaches zero on both sides of the origin, at points $L_1 < 0 < L_2$. 
\end{proof}
\begin{figure}
    \centering
    \includegraphics[scale = 0.25]{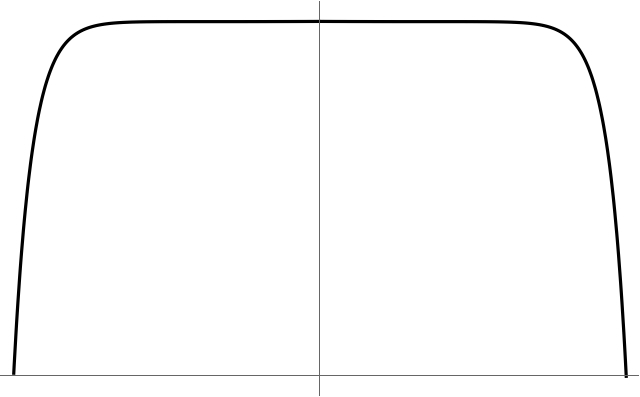}
    \caption{A numerically approximated solution $f$ to ODE (\hyperref[eqn:resulting-ODE]{$\star_{1,1,3}$}) with initial data $f'(0) = 0$ and $f''(0) = -1$ as given in Proposition~\ref{prop:SexBuffet}, showing the behavior of such solutions. 
    At the boundary points where $f = 0$, the derivatives of $f$ blow up.}
    \label{fig:numericalSB}
\end{figure}

\subsection{Completeness, boundary, singularities, and asymptotic geometry}\label{subsec:singularityGeometry}
The soliton metrics~\eqref{metric:SStar} with conformally cylindrical base, where $q = y f(x/y)$ for $f$ a solution of equation~\eqref{eqn:resulting-ODE}, are a priori only given on the chart described by the coordinates $(x,y,s,t)$ where~\eqref{eqn:CPSFMetric} is well-defined.
Such a metric may degenerate, determined by where $q,A,B$ vanish or blow up.
The points where $q$ vanishes can be at either finite or infinite distance away, representing incompleteness or a geometric end, respectively.
This dichotomy depends on the rate of vanishing of $q$ at such points.
By homogeneity, $q$ must vanish at the origin, so such a soliton cannot be compact. 
On the other hand, if $q$ blows up approaching a certain point, then there exist points at the boundary of the domain of definition $\Omega$ within finite distance, so the manifold is incomplete.

The domain of definition $\Omega$ of the metric $g$ in the $(x,y)$ plane is a maximal connected region where $A,B$ are positive and $q$ is non-zero: it is a connected component of $\{A(x) > 0\} \cap \{B(y) > 0\} \cap \{q(x,y) \neq 0\}$.
To understand the geometry of the soliton, we analyze the metric degeneration occurring along $\p \Omega$, including $(x,y) \to \infty$. 
The homogeneity of $q$ makes the base metric asymptotically cylindrical when $A(0) B(0) \neq 0$. 
The maximal domain $\Omega$ will be the conical region in the $(x,y)$ plane bounded by the rays where $q$ vanishes. 
Figure~\ref{fig:solitonsDomainGeometryEx} shows the possible configurations for $q$ and the resulting geometry of the constructed solitons:
\begin{figure}
    \centering
    \begin{tabular}[b]{c}
        \begin{tikzpicture}[scale = 0.55]
        \begin{axis}[clip = false,
        hide x axis, hide y axis,
        axis equal
        ]
        \addplot[color = black, samples = 10, domain = -5:5,dotted]{ 0};
        \addplot[color = gray,samples=100,domain = -3:3]{sqrt(9-(x)^2)};
        \addplot[color = gray,samples=100,domain = -3:3]{-sqrt(9-(x)^2)};
    
        \draw (0,0) node {$\bullet$};
        \addplot[color = black, dashed]coordinates
        {(-3,-4) (-3,6)};
        \addplot[color = black, dashed]coordinates
        {(3,-4) (3,6)};
        \addplot[color = black, dotted]coordinates
        {(0,-4) (0,6)};

        \addplot[name path = upper, samples = 10, domain = -3:3, color = white]{6    };
        \addplot[name path = lower, samples = 10, domain = -3:3, color = white]{-4    };
    
        \addplot [
                thick,
                color=black,
                fill=gray, 
                fill opacity=0.25
            ]
            fill between[
                of=lower and upper,
                soft clip={domain=-3:3},
            ];
        
        \end{axis}
    \end{tikzpicture} \\
        \small (a)
    \end{tabular}
    \begin{tabular}[b]{c}
        \begin{tikzpicture}[scale = 0.55]
        \begin{axis}[clip = false,
        hide x axis, hide y axis,
        axis equal
        ]

        \addplot[name path = upper, samples = 10, domain = -3:3, color = white]{6    };
        \addplot[name path = lowerL, samples = 10, domain = -3:0, color = white]{-x    };
        \addplot[name path = lowerR, samples = 10, domain = -3:3, color = white]{x    };
        \addplot [
                thick,
                color=black,
                fill=gray, 
                fill opacity=0.25
            ]
            fill between[
                of=lowerL and upper,
                soft clip={domain=-3:0},
            ];
            \addplot [
                thick,
                color=black,
                fill=gray, 
                fill opacity=0.25
            ]
            fill between[
                of=lowerR and upper,
                soft clip={domain=0:3},
            ];
            \addplot[color = black, samples = 10, domain = -5:5,dotted]{ 0};
        \addplot[color = gray,samples=100,domain = -3:3]{sqrt(9-(x)^2)};
        \addplot[color = gray,samples=100,domain = -3:3]{-sqrt(9-(x)^2)};
    
        \addplot[color = black, dashed]coordinates
        {(-3,-4) (-3,6)};
        \addplot[color = black, dashed]coordinates
        {(3,-4) (3,6)};
        \addplot[color = black,dotted]coordinates
        {(0,-4) (0,6)};
    
        \addplot[samples = 10, domain = -4:5, color = black]{x};
        \addplot[samples = 10, domain = -5:4, color = black]{-x};
        \end{axis}
    \end{tikzpicture} \\
        \small (b)
    \end{tabular}
    \begin{tabular}[b]{c}
        \begin{tikzpicture}[scale = 0.55]
        \begin{axis}[clip = false,
        hide x axis, hide y axis,
        axis equal
        ]

        \addplot[name path = upper, samples = 10, domain = 0:3, color = white]{x+x    };
        \addplot[name path = lower, samples = 10, domain = 0:3, color = white]{x    };
    
        \addplot [
                thick,
                color=black,
                fill=gray, 
                fill opacity=0.25
            ]
            fill between[
                of=lower and upper,
                soft clip={domain=0:3},
            ];
    
        \addplot[color = black, samples = 10, domain = -5:5,dotted]{ 0};
        \addplot[color = gray,samples=100,domain = -3:3]{sqrt(9-(x)^2)};
        \addplot[color = gray,samples=100,domain = -3:3]{-sqrt(9-(x)^2)};
    
        \addplot[color = black, dashed]coordinates
        {(-3,-4) (-3,6)};
        \addplot[color = black, dashed]coordinates
        {(3,-4) (3,6)};
        \addplot[color = black,dotted]coordinates
        {(0,-4) (0,6)};
    
        \addplot[samples = 10, domain = -4:5, color = black]{x};
        \addplot[samples = 10, domain = -2:3, color = black]{x+x};
        \end{axis}
    \end{tikzpicture} \\
        \small (c)
    \end{tabular}
    \begin{tabular}[b]{c}
        \begin{tikzpicture}[scale = 0.55]
        \begin{axis}[clip = false,
        hide x axis, hide y axis,
        axis equal
        ]

        \addplot[name path = upper, samples = 10, domain = -3:3, color = white]{6    };

        \addplot[name path = lower, samples = 10, domain = -3:3, color = black]{x    };

            \addplot [
                thick,
                color=black,
                fill=gray, 
                fill opacity=0.25
            ]
            fill between[
                of=lower and upper,
                soft clip={domain=-3:3},
            ];
            \addplot[color = black, samples = 10, domain = -5:5,dotted]{ 0};
        \addplot[color = gray,samples=100,domain = -3:3]{sqrt(9-(x)^2)};
        \addplot[color = gray,samples=100,domain = -3:3]{-sqrt(9-(x)^2)};
    
        \addplot[color = black, dashed]coordinates
        {(-3,-4) (-3,6)};
        \addplot[color = black, dashed]coordinates
        {(3,-4) (3,6)};
        \addplot[color = black,dotted]coordinates
        {(0,-4) (0,6)};
    
        \addplot[samples = 10, domain = -4:5, color = black]{x};
        \end{axis}
    \end{tikzpicture} \\
        \small (d)
    \end{tabular}
    \caption{The shaded regions above represent $\Omega$, the possible domains of definition for the~\eqref{metric:SStar} metrics with $b_0 > 0$.
    The conformally flat domains ($c = 0$) would be the same, without the constraint $|x| < \sqrt{a_0 / c}$ coming from $A > 0$.
    The origin bullet represents an asymptotically cylindrical end of the base metric.
    The solid lines denote $\{q = 0\}$, which are either ends or boundary points if the derivatives of $q$ are finite or infinite, respectively.
    If $b_0 < 0$, then the regions would be intersected with the half plane where $y > \sqrt{b_0}$. 
    }
    \label{fig:solitonsDomainGeometryEx}
\end{figure}
\begin{enumerate}[(a)]
\item $f \neq 0$ satisfying the closing condition $\lim_{z \to \pm \infty} \frac{f(z)}{z} \neq 0$; the origin is the only end of the manifold. 
    This represents the domain for the asymptotically cylindrical expanding soliton in example~\eqref{metric:q=sqrtx2+y2}.
    
    \item $f$ has two zeroes, so $q$ vanishes on the rays emanating from the origin, for example the (\hyperref[cor:SexBuffet]{\metricSexBuffet}) metrics for $c = 0$ (conformally flat) as well as $c \neq 0$ (conformally scalar flat).
    The boundary behavior where $q = 0$ is illustrated in Figure~\ref{fig:geometricBoundary}.

    \item same as (b), but $f(z)$ has two positive zeroes, so $y$ remains bounded.
    Solutions have been observed numerically, with the same boundary behavior as (\hyperref[cor:SexBuffet]{\metricSexBuffet}) metrics.
    The domain can also cross the $x$-axis, exemplified by the metric~\eqref{metric:q=sqrtx2+y2} for $b_0 < 0$, where the domain is bounded between the rays $y = \pm \sqrt{|b_0| / a_0} x$. 
   \item this case occurs for (\hyperref[cor:SexBuffet]{\metricSexBuffet}) metrics obtained as perturbations of~\eqref{metric:pbnski}.
   It can also occur for a solution $f$ to~\eqref{eqn:resulting-ODE} with a single zero and linear growth. 
   A critical case occurs when $f(0) = 0$ and the boundary becomes the $y$-axis, so the boundary line is vertical. 
   Another critical case is when $f \neq 0$ has sub-linear growth, so $q$ vanishes along the $x$-axis, creating an end of the manifold where the boundary line is horizontal.
   The latter case includes the Schwarzschild metric.

\end{enumerate}

To understand the metric near the boundary of $\Omega$, we need to consider the behavior of the metric $g$ along paths in the $(x,y)$ plane approaching $\p\Omega$. 
The study of such paths along the base $\Omega$, meaning fixing $s$ and $t$ to be constant, is sufficient because the base is totally geodesic since $g$ is toric.

When approaching the origin along a ray where $q \neq 0$, we may approximate $A,B$ by their values at $0$, so the base metric is approximated by the conformally flat one $\mr{(}$\hyperref[cor:SexBuffet]{$\mr{\metricSexBuffet}$}$\mr{)}$ with $c = 0$.
For homogeneous $q$, this metric is cylindrical.
The origin therefore represents an asymptotically cylindrical end in the base metric.
If $A$ and $B$ are constant, then the base metric is the genuine cylinder metric $\bR \times \bS^1$. 
For the $(s,t)$ directions of the torus fibration, the size of each torus corresponds to $1/q$, which grows linearly along the ray to the origin. 
The end at the origin is therefore asymptotic to the cone over a torus.
The radial direction in the cone corresponds to the cylindrical end $(x,y) \to (0,0)$ in the base metric.

We now examine the path as $y \to \infty$.
Any solitons in this category satisfying $f(0) \neq 0$ have domains where $y$ is unbounded, including the soliton metrics $\mr{(}$\hyperref[cor:SexBuffet]{$\mr{\metricSexBuffet}$}$\mr{)}$ of Corollary~\ref{cor:SexBuffet}. 
\begin{claim}\label{claim:length-of-paths}
The path along $y \to \infty$ has finite length when $A,B$ are quadratics, as in 
~\eqref{metric:q=sqrtx2+y2} and $\mr{(}$\hyperref[cor:SexBuffet]{$\mr{\metricSexBuffet}$}$\mr{)}$ with $c \neq 0$, and infinite length in the conformally flat case of~\eqref{metric:q=sqrtx2+y2} and $\mr{(}$\hyperref[cor:SexBuffet]{$\mr{\metricSexBuffet}$}$\mr{)}$ when $c = 0$. 
\end{claim}
\begin{proof}
We consider a sequence of points $(x_i,y_i)$ where $y_i \to \infty$ and write $q(x_i, y_i) = y_i q(x_i/y_i, 1)$.
For $A = a_0 - x^2$, $|x| \leq \sqrt{a_0}$, so $q(x_i/y_i, 1) \to L := q(0,1)$ and we can approximate the metric by $g \approx \frac{1}{y^2L^2} \left({d\tilde{x}^2} +  d\tilde{s}^2 + {d\tilde{y}^2} + d\tilde{t}^2\right)$.
The conformal factor is independent of $(x,s)$, so letting $w :=\log y$ exhibits this as a warped product metric with $\bS^2$, $g \approx  dt^2 + dw^2 + \frac{w^2}{L^2}g_{\bS^2}$.
Therefore, as $y \to \infty$, the new coordinate $t \to 0$ evidently has finite length.

When the pair $(A,B) = (a_0 - cx^2, b_0 + cy^2)$ become constants $(a_0, b_0)$ upon taking $c \downarrow 0$, the metric is of the form $g = \frac{1}{y^2 f(x/y)^2} (dx^2 + dy^2 + ds^2 + dt^2)$.
For this metric, curves of the form $(x_0,y)$ as $y \to \infty$ have infinite length. 
At a point $x_0$ where $q(x_0,1)$ is maximized, the path $\gamma(\tau) = (x_0,e^\tau,s_0,t_0)$ is a geodesic of infinite length, as this metric is comparable to the hyperbolic metric in $(x_0,y)$.
\end{proof}
This result shows that the conformally flat $\eqref{metric:SStar}$ metrics have another end where $(x,y) \to \infty$ where the tori shrink, so the total metric on this end is asymptotic to the base metric, which is genuinely cylindrical. 
For example, the metric~\eqref{metric:q=sqrtx2+y2} with $c = 0$ is a complete expanding Ricci soliton interpolating between two ends, one asymptotic to the cylinder $\bR \times \bS^1$, and the other asymptotic to $C(\bT^2)$, the cone over the two-torus. 
\begin{figure}
    \centering
    \begin{tabular}[b]{c}
        \begin{tikzpicture}[scale = 0.42]
        \begin{axis}[clip = false,
        hide x axis, hide y axis,
        axis equal
        ]

        \addplot[name path = upper, samples = 10, domain = -3:3, color = white]{6    };
        \addplot[name path = lowerL, samples = 10, domain = -3:0, color = white]{-x    };
        \addplot[name path = lowerR, samples = 10, domain = -3:3, color = white]{x    };
        \addplot [
                thick,
                color=black,
                fill=gray, 
                fill opacity=0.25
            ]
            fill between[
                of=lowerL and upper,
                soft clip={domain=-3:0},
            ];
            \addplot [
                thick,
                color=black,
                fill=gray, 
                fill opacity=0.25
            ]
            fill between[
                of=lowerR and upper,
                soft clip={domain=0:3},
            ];
        \addplot[color = black, samples = 10, domain = -5:5,dotted]{ 0};
        \addplot[color = gray,samples=100,domain = -3:3]{sqrt(9-(x)^2)};
        \addplot[color = gray,samples=100,domain = -3:3]{-sqrt(9-(x)^2)};
    
        \addplot[color = black, dashed]coordinates
        {(-3,-3) (-3,6)};
        \addplot[color = black, dashed]coordinates
        {(3,-3) (3,6)};
        \addplot[color = black,dotted]coordinates
        {(0,-3) (0,6)};

        \addplot[samples = 10, domain = -3:5, color = black]{x};
        \addplot[samples = 10, domain = -5:3, color = black]{-x};

        \addplot[color = blue, ->]coordinates{(1,2) (1,6)};
        
        \end{axis}
    \end{tikzpicture} \\
        \small (a)
    \end{tabular}
    \begin{tabular}[b]{c}
           \begin{tikzpicture}[scale = 0.42]
        \begin{axis}[clip = false,
        hide x axis, hide y axis,
        axis equal
        ]

        \addplot[name path = upper, samples = 10, domain = -5:5, color = white]{6    };
        \addplot[name path = lowerL, samples = 10, domain = -5:0, color = white]{-x    };
        \addplot[name path = lowerR, samples = 10, domain = 0:5, color = white]{x    };
        \addplot [
                thick,
                color=black,
                fill=gray, 
                fill opacity=0.25
            ]
            fill between[
                of=lowerL and upper,
                soft clip={domain=-5:0},
            ];
            \addplot [
                thick,
                color=black,
                fill=gray, 
                fill opacity=0.25
            ]
            fill between[
                of=lowerR and upper,
                soft clip={domain=0:5},
            ];
            \addplot[color = black, samples = 10, domain = -5:5,dotted]{ 0};
        \addplot[color = gray,samples=100,domain = -3:3]{sqrt(9-(x)^2)};
        \addplot[color = gray,samples=100,domain = -3:3]{-sqrt(9-(x)^2)};
    
        
        \addplot[color = black,dotted]coordinates
        {(0,-3) (0,6)};

        \addplot[samples = 10, domain = -3:5, color = black]{x};
        \addplot[samples = 10, domain = -5:3, color = black]{-x};

        \addplot[color = blue, ->]coordinates{(1,2) (1,6)};
        
        \end{axis}
    \end{tikzpicture}\\
         \small (b) 
    \end{tabular}
        \begin{tabular}[b]{c}
    \begin{tikzpicture}[scale = 0.42]
        \begin{axis}[clip = false,
        hide x axis, hide y axis,
        axis equal
        ]

        \addplot[name path = upper, samples = 10, domain = -3:3, color = white]{6    };
        \addplot[name path = lowerL, samples = 10, domain = -3:0, color = white]{-x    };
        \addplot[name path = lowerR, samples = 10, domain = -3:3, color = white]{x    };
        \addplot [
                thick,
                color=black,
                fill=gray, 
                fill opacity=0.25
            ]
            fill between[
                of=lowerL and upper,
                soft clip={domain=-3:0},
            ];
            \addplot [
                thick,
                color=black,
                fill=gray, 
                fill opacity=0.25
            ]
            fill between[
                of=lowerR and upper,
                soft clip={domain=0:3},
            ];
            \addplot[color = black, samples = 10, domain = -5:5,dotted]{ 0};
        \addplot[color = gray,samples=100,domain = -3:3]{sqrt(9-(x)^2)};
        \addplot[color = gray,samples=100,domain = -3:3]{-sqrt(9-(x)^2)};
    
        \addplot[color = black, dashed]coordinates
        {(-3,-3) (-3,6)};
        \addplot[color = black, dashed]coordinates
        {(3,-3) (3,6)};
        \addplot[color = black,dotted]coordinates
        {(0,-3) (0,6)};

        \addplot[samples = 10, domain = -3:5, color = black]{x};
        \addplot[samples = 10, domain = -5:3, color = black]{-x};

        \addplot[color = blue,->]coordinates
        {(1.3,4) (1.3,1.7)};
        \addplot[color = blue,samples=100, ->,domain = -0.5:1.3]{1.5};
        
        \end{axis}
    \end{tikzpicture}\\
    \small (c)
    \end{tabular}
    \begin{tabular}[b]{c}
    \begin{tikzpicture}[scale = 0.42]
        \begin{axis}[clip = false,
        hide x axis, hide y axis,
        axis equal
        ]

        \addplot[name path = upper, samples = 10, domain = -3:3, color = white]{6    };
        \addplot[name path = lowerL, samples = 10, domain = -3:0, color = white]{-x    };
        \addplot[name path = lowerR, samples = 10, domain = -3:3, color = white]{x    };
        \addplot [
                thick,
                color=black,
                fill=gray, 
                fill opacity=0.25
            ]
            fill between[
                of=lowerL and upper,
                soft clip={domain=-3:0},
            ];
            \addplot [
                thick,
                color=black,
                fill=gray, 
                fill opacity=0.25
            ]
            fill between[
                of=lowerR and upper,
                soft clip={domain=0:3},
            ];
            \addplot[color = black, samples = 10, domain = -5:5,dotted]{ 0};
        \addplot[color = gray,samples=100,domain = -3:3]{sqrt(9-(x)^2)};
        \addplot[color = gray,samples=100,domain = -3:3]{-sqrt(9-(x)^2)};
    
        \addplot[color = black, dashed]coordinates
        {(-3,-3) (-3,6)};
        \addplot[color = black, dashed]coordinates
        {(3,-3) (3,6)};
        \addplot[color = black,dotted]coordinates
        {(0,-3) (0,6)};

        \addplot[samples = 10, domain = -3:5, color = black]{x};
        \addplot[samples = 10, domain = -5:3, color = black]{-x};

        \addplot[color = blue,samples=100, ->,domain = 1:3]{4};
        
        \end{axis}
    \end{tikzpicture}\\
    \small (d)
    \end{tabular}
    \begin{tabular}[b]{c}
    \begin{tikzpicture}[scale = 0.42]
        \begin{axis}[clip = false,
        hide x axis, hide y axis,
        axis equal
        ]

        \addplot[name path = upper, samples = 10, domain = -3:3, color = white]{6    };
        \addplot[name path = lowerL, samples = 10, domain = -3:0, color = white]{-x    };
        \addplot[name path = lowerR, samples = 10, domain = -3:3, color = white]{x    };
        \addplot [
                thick,
                color=black,
                fill=gray, 
                fill opacity=0.25
            ]
            fill between[
                of=lowerL and upper,
                soft clip={domain=-3:0},
            ];
            \addplot [
                thick,
                color=black,
                fill=gray, 
                fill opacity=0.25
            ]
            fill between[
                of=lowerR and upper,
                soft clip={domain=0:3},
            ];
            \addplot[color = black, samples = 10, domain = -5:5,dotted]{ 0};
        \addplot[color = gray,samples=100,domain = -3:3]{sqrt(9-(x)^2)};
        \addplot[color = gray,samples=100,domain = -3:3]{-sqrt(9-(x)^2)};
    
        \addplot[color = black, dashed]coordinates
        {(-3,-3) (-3,6)};
        \addplot[color = black, dashed]coordinates
        {(3,-3) (3,6)};
        \addplot[color = black,dotted]coordinates
        {(0,-3) (0,6)};

        \addplot[samples = 10, domain = -3:5, color = black]{x};
        \addplot[samples = 10, domain = -5:3, color = black]{-x};

        \addplot[color = blue,samples=100, ->,domain = 1:2.8]{3};
        
        \end{axis}
    \end{tikzpicture}\\
    \small (e)
    \end{tabular}
    \begin{tabular}[b]{c}
    \begin{tikzpicture}[scale = 0.42]
        \begin{axis}[clip = false,
        hide x axis, hide y axis,
        axis equal
        ]

        \addplot[name path = upper, samples = 10, domain = -3:3, color = white]{6    };
        \addplot[name path = lowerL, samples = 10, domain = -3:0, color = white]{-x    };
        \addplot[name path = lowerR, samples = 10, domain = -3:3, color = white]{x    };
        \addplot[name path = bottom, samples = 10, domain = -2:2, color = white]{2    };
        \addplot[samples = 10, domain = -4:4, color = black]{2    };
        \addplot [
                thick,
                color=black,
                fill=gray, 
                fill opacity=0.25
            ]
            fill between[
                of=lowerL and upper,
                soft clip={domain=-3:-2},
            ];
            \addplot [
                thick,
                color=black,
                fill=gray, 
                fill opacity=0.25
            ]
            fill between[
                of=lowerR and upper,
                soft clip={domain=2:3},
            ];
            \addplot [
                thick,
                color=black,
                fill=gray, 
                fill opacity=0.25
            ]
            fill between[
                of=bottom and upper,
                soft clip={domain=-2:2},
            ];
            \addplot[color = black, samples = 10, domain = -5:5,dotted]{ 0};
        \addplot[color = gray,samples=100,domain = -3:3]{sqrt(9-(x)^2)};
        \addplot[color = gray,samples=100,domain = -3:3]{-sqrt(9-(x)^2)};
    
        \addplot[color = black, dashed]coordinates
        {(-3,-3) (-3,6)};
        \addplot[color = black, dashed]coordinates
        {(3,-3) (3,6)};
        \addplot[color = black,dotted]coordinates
        {(0,-3) (0,6)};

        \addplot[samples = 10, domain = -3:5, color = black]{x};
        \addplot[samples = 10, domain = -5:3, color = black]{-x};

        \addplot[color = blue,->]coordinates
        {(0,5) (0,2.1)};
        
        \end{axis}
    \end{tikzpicture}\\
    \small (f)
    \end{tabular}
    \caption{Figure (a) shows a conformally scalar flat soliton, (e.g.~$\mr{(}$\hyperref[cor:SexBuffet]{$\mr{\metricSexBuffet}$}$\mr{)}$ with $c \neq 0$) where the path $y \to \infty$ has finite length.
    In the conformally flat case (b), (e.g.~$\mr{(}$\hyperref[cor:SexBuffet]{$\mr{\metricSexBuffet}$}$\mr{)}$ with $c = 0$), this length is infinite.
    Figure (c) shows paths of either finite or infinite length depending on the value of $f'$ near the boundary.
    For metrics where $f$ has bounded derivative where it vanishes, these lengths are infinite, so the boundary line is an end. 
    For metrics $\mr{(}$\hyperref[cor:SexBuffet]{$\mr{\metricSexBuffet}$}$\mr{)}$ when $c = 0$ or $c \neq 0$, these paths have finite length approaching the boundary where the curvature blows up.
    Figure (d) shows a path of finite length where the manifold closes up spherically, smoothly if $a_0 = 1$ or with a conical singularity otherwise.
    We can also approach a point where $A,q$ both vanish as shown in Figure (e), which can either be finite or infinite length.
    Figure (f) shows a path along with the metric closes up as in~\eqref{metric:q=sqrtx2+y2} solitons with $b_0 < 0$.
    The $(y,t)$ surface is hyperbolic as shown above in equation~\eqref{eqn:metric:x2-b0}, smoothly if and only if $b_0 = -1$, otherwise this is a conical singularity.}
    \label{fig:offToRadialy=infty}
\end{figure}

We can also characterize the case when vertical paths of finite length close up smoothly.
\begin{corollary}\label{cor:closingUpSmoothingAty=infty}
    For $g$ a metric of the form~\eqref{metric:SStar}, if $a_0 = 1$, then the metric in the region $y \to \pm \infty$ within $\Omega$ closes up smoothly.
\end{corollary}
\begin{proof}
    It suffices to consider $f = \sqrt{a_0 + b_0 z^2}$, as in the case of~\eqref{metric:q=sqrtx2+y2} metrics, since by rescaling, any metric $g$ with homogeneous $q$ is arbitrarily well approximated by this model metric in the region $\{x^2 \leq a_0\} \cap \{y^2 \geq L\}$ for $L$ sufficiently large.
    To study vertical paths of constant $x = x_0 \in (- \sqrt{a_0}, \sqrt{a_0})$ and $|y| \to + \infty$, let $\hat{g}$ be the restriction of $g$ to the surface slice of fixed $(x,s) = (x_0,s_0)$. 
    We reparametrize this using $\rho_{x_0}(y)$, defined as $\int \frac{dy}{\sqrt{(y^2 + b_0) (y^2 + b_0 x_0^2)}} = i b_0^{-1/2} F \left( i \sinh^{-1} \frac{y}{\sqrt{b_0} |x_0|} ; x_0 \right)$, where $F(x; k)$ denotes the incomplete elliptic integral of the first kind with modulus $k$, cf.~\cite{elliptic}*{\S 3.1}.
    It follows from \cite{elliptic}*{\S 3.10} that $\rho_{x_0}(y)$ depends smoothly on $(x_0,y)$, and $y \to \infty$ corresponds to $\rho \to K(x_0)$, the complete elliptic integral with modulus $x_0$.
    We can write $\hat{g} = \frac{1}{a_0} (d \rho^2 + f(\rho)^2 dt^2)$ with $f(\rho) = \sqrt{ (y^2 + b_0) / (y^2 + b_0 x_0^2)}$, so $f(\rho) \to 1$ in $C^k$ uniformly in $x_0$ as $y \to \infty$.
    If $a_0 = 1$, $\Sigma_1$ is the round unit sphere as described in equation~\eqref{eqn:metric:a0-x2}.
    Thus, as $y \to \infty$, $g$ converges to $(dr^2 + r^2 g_{\tilde{\bS}^2} + dt^2)$ in $C^k$, for $r = y^{-1}$, where $g_{\tilde{\bS}^2}$ is the spherical metric in $(x,s)$-coordinates.
    By \cite{verdiani-ziller}*{Theorem 2} and~\eqref{eqn:metric:a0-x2}, the limiting metric of $g$ is smooth since the cone angle near $r=0$ is $2 \pi$, provided that $a_0 = 1$.
    At $r=0$, the topology is given locally as $\bR^3 \times \bS^1$, so $g$ closes up smoothly.  
\end{proof}

We can also consider a path approaching a point $(x_0,y_0)$ with $q(x_0,y_0) = 0$ and $A(x_0), B(y_0) \neq 0$, as in Figure~\hyperref[fig:offToRadialy=infty]{\ref{fig:offToRadialy=infty}(c)}.
Depending on the decay of $q$, the length of this path may be infinite, so $(x_0,y_0)$ represents an end of the manifold, or finite, so the metric is incomplete with a boundary or singular. 
Let us approximate the base metric along this path as $\frac{1}{q(x,y)^2}\left(\frac{dx^2}{A(x_0)} + \frac{dy^2}{B(y_0)}\right)$.
We denote $r(x,y) = d((x,y),(x_0,y_0))$ and consider a path that ends at $(x_0,y_0)$; this has length uniformly equivalent to $\int_\gamma q^{-1}$ as $q \to 0$. 
The path integral converges if $q = O(r^{1+\ve})$ for some $\ve > 0$, and diverges for the metrics~\eqref{metric:pbnski} where $q$ is linear.
\begin{figure}
    \centering
    \begin{tabular}[b]{c}
    \includegraphics[scale=0.2]{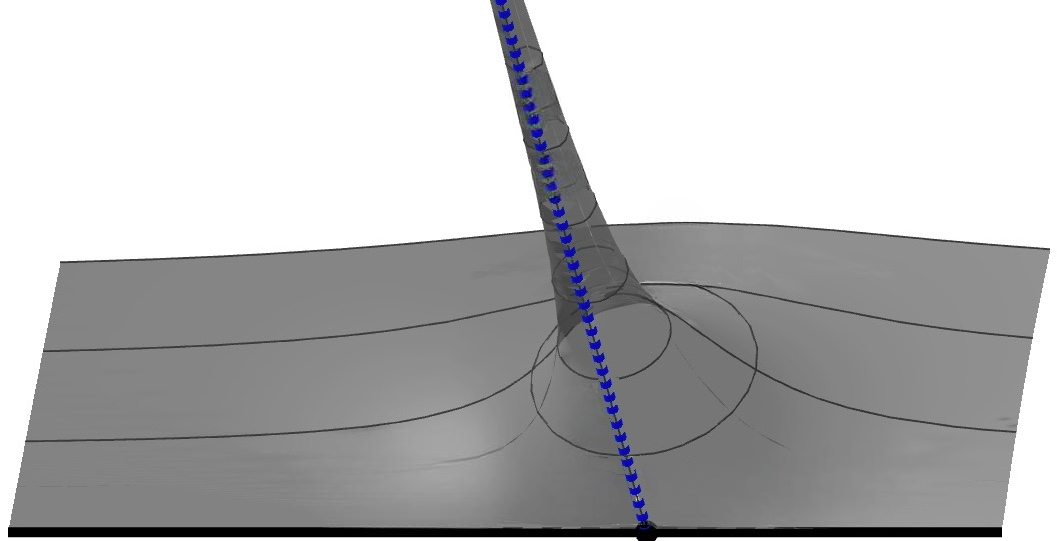}\\
    (a)
    \end{tabular}
    \begin{tabular}[b]{c}
        \includegraphics[scale = 0.08]{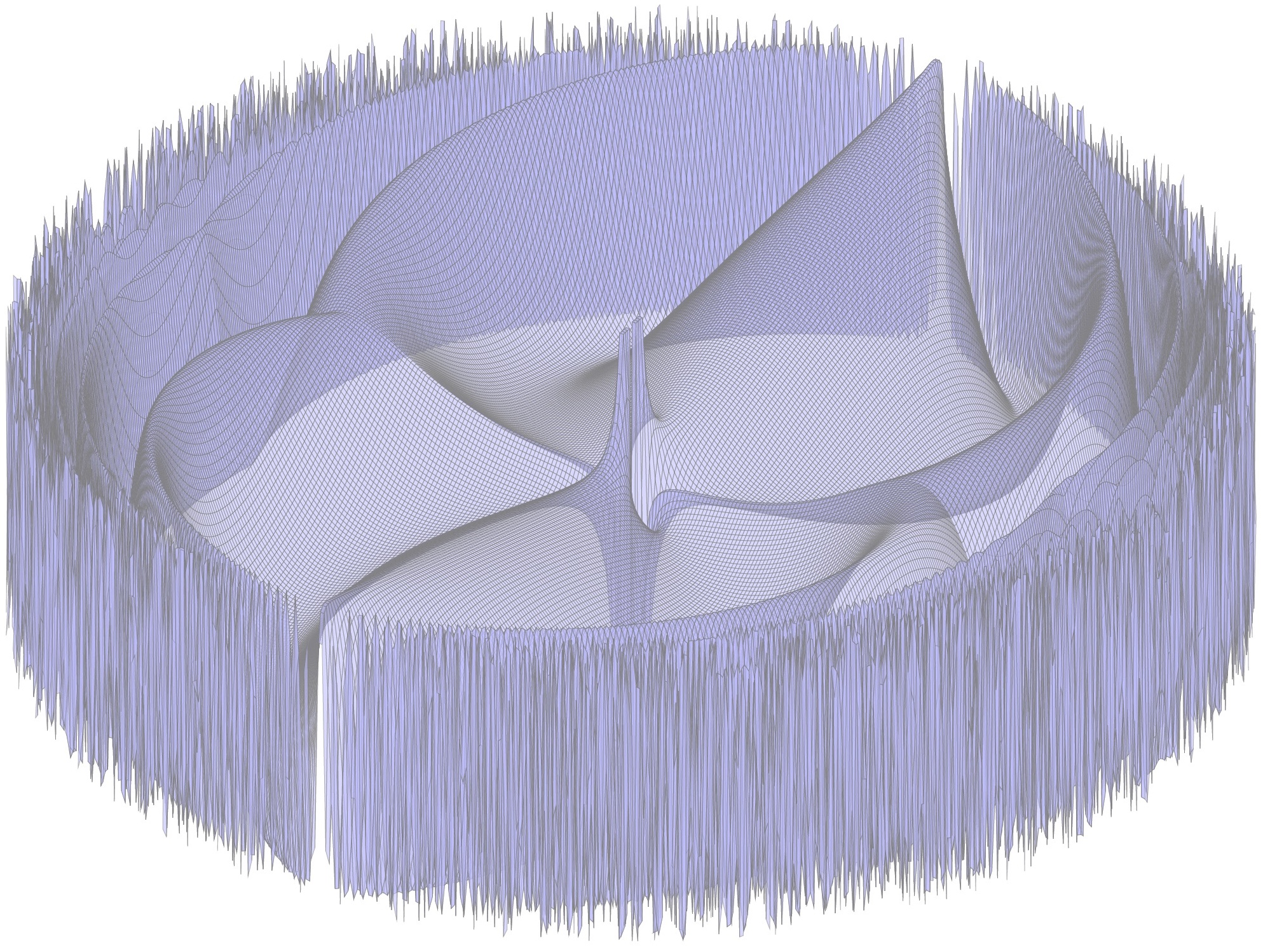}
        \\
    (b)
    \end{tabular}
    \caption{Two geometric realizations of the boundary for the $\mr{(}$\hyperref[cor:SexBuffet]{$\mr{\metricSexBuffet}$}$\mr{)}$ metrics, viewed (a) from the boundary and (b) from the interior.
    These visualizations exhibit the boundary of the cylinder as the limit of neighboring circles growing to infinity.
    In Figure (a), the neighboring circles expand by stereographic projection; the blue line represents a curve as in~\hyperref[fig:offToRadialy=infty]{\ref{fig:offToRadialy=infty}(c)} of finite length.
    In Figure (b), the concentric circles expand by increasingly condensed corrugation near the boundary to preserve the metric in the radial direction.
    This behavior is similar to the singularity in the Ooguri-Vafa metric where the circular fibers blow up in finite radius.
    Our metric singularities can therefore be seen as a cohomogeneity two analogue. 
    }
    \label{fig:geometricBoundary}
\end{figure}

\begin{proposition}\label{prop:newCompleteExpanders}
    The expanding solitons~\eqref{metric:q=sqrtx2+y2} with $a_0 = 1$ and $b_0 > 0$ are complete. 
    The origin $(x,y) = (0,0)$ represents and end asymptotic to the cone over a torus $C(\bT^2)$.
    If $(A,B) = (1 - x^2, y^2 + b_0)$, the origin is the only end. 
    If $(A,B) = (1,b_0)$, there is an additional end at $x^2 + b_0y^2 \to \infty$, asymptotic to a cylinder $\bR \times \bS^1$.
\end{proposition}
\begin{proof}
    Since $b_0 > 0$, $q$ has no zeroes except for the origin, so the domain $\Omega$ of the base metric is either $\bR^2\setminus \{(0,0)\}$ when $A,B$ are constants, or $\bR^2\setminus \{(0,0)\} \cap \{|x| \leq \sqrt{a_0} \}$ when $(A,B)$ are quadratics; the latter case is exhibited in Figure~\hyperref[fig:solitonsDomainGeometryEx]{\ref{fig:solitonsDomainGeometryEx}(a)}.

    First, suppose $(A,B) = (a_0, b_0)$ are constant, so the base metric is the cylinder over a level set $\{a_0 x^2 + b_0y^2 = 1\}$, as seen by the coordinate transformation $r^2 = a_0 x^2 + b_0 y^2$ and $\theta = \tan^{-1}(a_0y/b_0x)$. 
    The tori expand (resp.~contract) linearly along the two ends of the cylinder, as $q \to 0$ (resp. $q \to \infty$).
    As $q \to \infty$, the tori shrink, creating an overall cylindrical end $\bR \times \bS^1$. 
    At the other end, where $q \to 0$, the tori expand linearly, so metrically this end is $C(\bT^2)$, asymptotic to the cone over a torus. 

    In the case $(A,B) = (a_0-x^2, y^2 + b_0)$, we have $(A,B) \approx (a_0, b_0)$ near the origin, so this end is asymptotic to $C(\bT^2)$, as in the above case.
    Since $|x| \leq \sqrt{a_0}$, there are two behaviors to analyze: where $x \to \sqrt{a_0}$ with $y$ bounded, and $y \to \infty$. 
    First, at points where $A$ vanishes and $x^2 = a_0$ for $a_0=1$, the metric closes up spherically to a smooth point by \cite{verdiani-ziller}*{Theorem 2} and \eqref{eqn:metric:a0-x2}.
    The length of a path with $(x,s,t)$ fixed and $y \to \infty$, as illustrated in Figure~\hyperref[fig:offToRadialy=infty]{\ref{fig:offToRadialy=infty}(a)}, is $\lesssim \int_1^\infty \frac{dy}{y^2} < + \infty$.
    By Corollary~\ref{cor:closingUpSmoothingAty=infty}, the vertical paths close up smoothly for $a_0 = 1$. 
\end{proof}
Note that by the previous observations and Proposition~\ref{prop:transform-A-B-c}, we have soliton metrics $\{ g_{b_0, c} \}_{b_0, c \geq 0}$ with the same behavior for the quadratics $(A,B) = (1 - cx^2, b_0 + c y^2)$, where $c \geq 0$.
For $b_0 > 0$, say $b_0 = 1$, we can take $c \downarrow 0$ to obtain a family of solitons $\{ g_{1, c} \}_{c > 0}$ metrics approaching the metric $g = \frac{1}{x^2 + y^2} \left( dx^2 + dy^2 + ds^2 + dt^2 \right)$ with cylindrical base.
This metric is a complete expanding soliton with two volume collapsing ends: a cylindrical end, $\bR \times \bS^1$, and a conical end over a two-torus, $C(\bT^2)$. 
Another geometric limit arises when $b_0 \to 0^+$; this makes the~\eqref{metric:q=sqrtx2+y2} metrics with domain $\Omega$ as in Figure~\hyperref[fig:solitonsDomainGeometryEx]{\ref{fig:solitonsDomainGeometryEx}(a)} degenerate to the Schwarzschild metric, with domain $\Omega$ as in Figure~\hyperref[fig:solitonsDomainGeometryEx]{\ref{fig:solitonsDomainGeometryEx}(d)}.
The limiting geometry at $b_0 = 0$ is an ALF Ricci-flat metric, notably exhibiting different asymptotic behavior from the expanding solitons $\{ g_{b_0, c} \}$ with $b_0 \to 0^+$.

We can generalize Proposition~\ref{prop:newCompleteExpanders} to other~\eqref{metric:SStar} metrics with similar properties.
\begin{corollary}\label{cor:completeness/Singularity}
    The metrics~\eqref{metric:pbnski},~\eqref{metric:q=sqrtx2+y2} with $b_0 \geq 0$, and more generally metrics of the form $\eqref{metric:SStar}$, where $f'$ is bounded at the zeroes of $f$, are complete if $a_0 = 1$ and $b_0 \in \bR_{\geq 0} \cup \{ - 1 \}$.
    The metrics~\eqref{metric:q=sqrtx2+y2} with $b_0 < 0$ and $\mr{(}$\hyperref[cor:SexBuffet]{$\mr{\metricSexBuffet}$}$\mr{)}$ are length spaces with boundary lying over $\{q = 0\} \setminus (0,0)$, along which the curvature blows up.
\end{corollary}
\begin{proof}
Metrics~\eqref{metric:pbnski} of the form have bounded derivative of $f$ near the vanishing of $q$, so the rays in $\p \Omega$ where $q$ vanishes are at infinite distance and are therefore ends of the manifold. 
This condition will apply to create an end to any such metric~\eqref{metric:SStar} with this property, making it complete. 

For metrics where $q$ has no zeroes except at the origin, as is the case for~\eqref{metric:q=sqrtx2+y2} with $b_0 > 0$, Proposition \ref{prop:newCompleteExpanders} shows that the origin becomes an end with cubic volume growth. 
Corollary~\ref{cor:closingUpSmoothingAty=infty} shows that the geometry as $y \to \infty$ is approximated by the three-dimensional warped product $dw^2 + \frac{w^2}{L^2} g_{\bS^2}$, which closes up with the proper cone angle of $2 \pi$ for $a_0 = 1$.
The $(x,y,s,t)$ coordinates close up by adding in topology of $\bR \times \bS^3$, which satisfies the Heine-Borel property as paths with $y \to \infty$ have finite length.
The metric is therefore complete.

If we approach a point where $A$ vanishes and $q$ does not, such as in Figure~\hyperref[fig:offToRadialy=infty]{\ref{fig:offToRadialy=infty}(d)}, then the metric is uniformly equivalent to the spherical metric scaled by the value of $q$ near the point, so this path has finite length.
By \cite{verdiani-ziller}, the metric closes up smoothly precisely when $a_0 = 1$ from equation~\eqref{eqn:metric:a0-x2}.

A path to a boundary point $p = (x_0, y_0) \in \p\Omega$ where $f(x_0/y_0) = 0$ (so $q(p) = 0$) and $A(x_0) = 0$, as in Figure~\hyperref[fig:offToRadialy=infty]{\ref{fig:offToRadialy=infty}(e)}, can either have finite or infinite length.
Along the path in Figure~\hyperref[fig:offToRadialy=infty]{\ref{fig:offToRadialy=infty}(e)}, setting $\rho = x_0 -x$ expresses the metric as $g \approx \frac{1}{q^2}\left(\frac{d\rho^2}{2\rho} + 2\rho\, ds^2 + g_{\Sigma_2(y,t)}\right)$.
As we approach $\rho = 0$ at a point where $q = 0$, the length of this curve is equivalent to $\int_0^\ve \frac{1}{q\sqrt{r}}$, so this is finite when $q$ decays at least as steeply as $\sqrt{r}$, and infinite otherwise. 
This $p \in \partial \Omega$ represents a boundary point in the former case and an end in the latter.

If $b_0 < 0$, along a path where $y \to \sqrt{b_0}$ and $q$ does not vanish, such as in Figure~\hyperref[fig:offToRadialy=infty]{\ref{fig:offToRadialy=infty}(f)}, the metric closes up smoothly if and only if $b_0 = -1$ from equation~\eqref{eqn:metric:x2-b0}, else this limit point is a conical singularity. 
If $b_0 = -1$, then the smooth region where $\{y = \sqrt{-b_0}\}\cap \{q \neq 0\}$ separates the two boundary regions where $q$ vanishes. 

The incompleteness of the metric (\hyperref[cor:SexBuffet]{$\mr{\metricSexBuffet}$}) follows from Proposition~\ref{prop:SexBuffet}: the solutions $f$ exhibited there have derivatives blowing up when they vanish, (see Figure~\ref{fig:numericalSB}), so the corresponding functions $q$ have rapid decay, meaning that the length of the path considered above is finite and the metrics are incomplete at $q = 0$.

Each (\hyperref[cor:SexBuffet]{$\mr{\metricSexBuffet}$}) metric defined over $\Omega$ can be completed to a Ricci soliton with boundary over $\ol{\Omega}$. 
The region $\{ q = 0 \} \setminus (0,0)$ defines the topological boundary of the manifold because the length of a path approaching any $p \in \partial \Omega \setminus (0,0)$ while keeping $s,t$ constant is finite.
Along a ray $\{ x = C y\}$ on which $q = 0$, so $f(C) = 0$, there are coordinates $(d,r,s,t)$, where $r$ is along the ray and $d$ is the Euclidean distance from this line.
These coordinates give a diffeomorphism to a region in $\bR^+ \times \bR^3$, exhibiting $\{ x = C y \}$ as the topological boundary.
Figure~\ref{fig:geometricBoundary} shows the geometry of the boundary projected to the $(d,s)$ coordinates; it is topologically equivalent to two copies of $\bT^2 \times \bR$, each corresponding to a zero of $f$.

The paths described above characterize all the equivalence classes of Cauchy sequences in this metric; any path approaching such a $p \in \partial \Omega$, with $s,t$ entries not eventually constant, has infinite length, due to the torus directions expanding to lines near the boundary.
Therefore, the metric completion realizes the total space with boundary as a length space.
A metric ball of size $3\ve$ around a boundary point will contain nearby boundary points.
These are attained via a path radially traveling inward up to distance $\ve$, rotating in the $s$ and $t$ directions by at most $\ve$, and then returning to the boundary while fixing $s,t$.
These metric charts give the same topology as the atlas coming from the boundary charts, which relies only on the manifold structure, not the metric.
The curvature blows up on the boundary, showing that these coordinates are maximal and the boundary is a genuine obstruction to continuing the soliton equation on any larger domain.
\end{proof}

The explicit metrics~\eqref{metric:q=sqrtx2+y2} where $b_0 < 0$ are model examples of the boundary behavior (\hyperref[cor:SexBuffet]{\metricSexBuffet}) described in Corollary~\ref{cor:completeness/Singularity} where $f'$ blows up at the zeroes of $f$, since $f = \sqrt{a_0 - |b_0| z^2}$.
Notably, $q$ has zeroes along the lines $|b_0| x^2 = a_0y^2$, and $y > \sqrt{|b_0|}$, so the domain of definition of the base metric is as in~\hyperref[fig:solitonsDomainGeometryEx]{\ref{fig:solitonsDomainGeometryEx}(b)} intersected with the half-plane $y > \sqrt{|b_0|}$.
Over the cone where $|b_0|x^2 = a_0 y^2$ and $y \geq |b_0|$, the curvature blows up as expected from \cite{catino-mantegazza-mazzieri}*{Corollary 1.3}.
The boundary $\partial \Omega$ is reach
ed in finite distance along the base metric, leading to the boundary behavior as described in Figure~\ref{fig:geometricBoundary}.

\section{Proof of Theorems~\ref{thm:main-theorem} and~\ref{thm:secondary-thm}}\label{section:proofs}

We now use the rigidity results developed in the previous sections to complete the proofs of Theorems~\ref{thm:main-theorem} and~\ref{thm:secondary-thm}.
We first provide a complete classification of Ricci soliton metrics with conformally cylindrical base under the assumptions~\eqref{eqn:metric-assumptions}. 
Then, we show that if one of the surface factors has constant curvature but is not a piece of a hyperbolic cusp, then either the metric has conformally cylindrical base, or the surface is flat and the metric is among the ones obtained in Section~\ref{section:rigidity-results}.

\begin{proof}[Proof of Theorem~\ref{thm:main-theorem}]
   If the conformal factor $q$ is homogeneous, we consider three possibilities:
   \begin{enumerate}[(i)]
       \item $D^2 q = 0$;
       \item $D^2 q \neq 0$ and $A,B$ are not both (possibly degenerate) quadratics;
       \item $A,B$ are both (possibly degenerate) quadratic.
   \end{enumerate}
   In the first case, the classification of soliton metrics with $D^2 q =0$ from Proposition~\ref{prop:d2q=0} shows that $q$ must be the Schwarzschild metric, the \pbnski metric, or a product metric of compatible Einstein or soliton surfaces.
   In the second case, Proposition~\ref{prop:homogeneous-rigidity} shows that $g$ must be a member of the family of singular steady solitons \eqref{eqn:singular-soliton-metrics}.
   Finally, if $A$ and $B$ are (possibly degenerate) quadratics, Corollary~\ref{corollary:q-a-b-quadratics} shows that the metric is conformal to $\bS^2 \times \bH^2$ or $\bR^4$ by a factor $q = y f(x/y)$, where $f$ is any solution of~\eqref{eqn:resulting-ODE}.
\end{proof}

To prove Theorem~\ref{thm:secondary-thm}, we first establish a pair of sixth-order differential equations in the conformal factor $q$ that force it to be homogeneous, up to translating $x,y$, if it produces a Ricci soliton. 
For this, we use the classification of solutions to the degenerate Monge-Amp\`ere equation $\det D^2 q = 0$ from Theorem~\ref{thm:solutions-to-monge-ampere}.
\begin{lemma}\label{lemma:sixth-order-condition}
Consider a Ricci soliton metric of the form~\eqref{eqn:PaperMetric} for which $B(y) = y^2 + b_0$ and the function $v := \frac{q_{xy}}{q_{yy}}$ associated to the conformal factor $q$ satisfies the pair of equations
\begin{equation}\label{eqn:sixth-order-condition}
\px \left( \frac{v_{yyy} v_y}{v_{yy}^2} \right) = 0, \qquad \py \left( \frac{v_{yyy} v_y}{v_{yy}^2} \right) = 0.
\end{equation}
Then, after possibly translating $x,y$, $q$ is homogeneous and the metric is classified by Theorem~\ref{thm:main-theorem}.
\end{lemma}
\begin{proof}
Recall from our general Theorem \ref{thm:solutions-to-monge-ampere} that the quantity $v := \frac{q_{xy}}{q_{yy}} = \frac{q_{xx}}{q_{xy}}$ satisfies $v_x = v v_y$, as obtained in ~\eqref{eqn:qxy-qyy-v-transport}.
Equation~\eqref{eqn:sixth-order-condition} implies that $v_{yyy} v_y = \alpha v_{yy}^2$ for some constant $\alpha$, where either $\alpha \in \{ 0, 1\}$ or we have the pair of equations 
\begin{equation}\label{eqn:pair-of-equations}
\py^2 (v_y^{1-\alpha}) = 0 \qquad \text{and} \qquad v_y^{1-\alpha} = \tilde{h}_1(x) y + \tilde{h}_0(x)
\end{equation}
for some functions $\tilde{h}_1, \tilde{h}_0$.
If $\alpha = 0$, then $v_{yyy} = 0$ implies that $v = h_2(x) y^2 + h_1(x) y + h_0(x)$.
In this case, $v_x = v v_y$ requires $h_2 = 0$ and $h'_1 = h_1^2$, so $v(x) = - \frac{y - y_0}{x - x_0}$.
In the framework of Theorem~\hyperref[thm:solutions-to-monge-ampere]{\ref{thm:solutions-to-monge-ampere}(iii)}, this requires $\psi(t) = y_0 + x_0v$, so $q(x,y) = (x-x_0) \Theta ( - \frac{y - y_0}{x - x_0} ) + c$ for some $c$.
By Lemma~\ref{lemma:Cx=Cy=0}, we have $c=0$, so $q$ is homogeneous after translating $x,y$.
If $\alpha = 1$ or $\alpha = 2$ in \eqref{eqn:pair-of-equations}, we would have, respectively,
\[
v(x,y) = h_0(x) e^{h_1(x) y} + h_2(x) \qquad \text{or} \qquad v(x,y) = \frac{\log |h_1(x) y + h_0(x)| + h_2(x)}{h_1(x)}
\]
for some functions $h_i$; neither expression can satisfy $v_x = v v_y$ for $h_1 \neq 0$.
We may therefore consider $\alpha \not\in \{ 0,1,2\}$, in which case the equations~\eqref{eqn:pair-of-equations} imply
\begin{equation}\label{eqn:v-from-condition}
v(x,y) = ( h_1(x) y + h_0(x))^{\frac{2-\alpha}{1-\alpha}} + h_2(x) , \qquad \text{for } \; \alpha \not\in \{ 0,1,2\}
\end{equation}
for some functions $h_i$.
For the function~\eqref{eqn:v-from-condition} to satisfy $v_x = vv_y$, we need 
\[
\frac{2-\alpha}{1-\alpha} h'_1 y ( h_1 y + h_0)^{\frac{1}{1-\alpha}} + h'_2 = \frac{2-\alpha}{1-\alpha} h_1 (h_1 y + h_0)^{\frac{1}{1-\alpha}} v \implies v = \frac{h'_1}{h_1} + \frac{1-\alpha}{2-\alpha} h'_2 (h_1 y + h_0)^{- \frac{1}{1-\alpha}}.
\]
For non-trivial functions $h_i$, the latter equality forces $\frac{2-\alpha}{1-\alpha} = - \frac{1}{1-\alpha}$, whereby $\alpha = 3$.
Combining the above expressions forces $h'_1 = 0$ and $2h'_2 = h_1$, so 
\begin{equation}\label{eqn:v-from-quadratic-psi}
v(x,y) = \frac{x + \sqrt{x^2 + 4y}}{2}
\end{equation}
up to scaling and translating $x,y$.
If~\eqref{eqn:v-from-quadratic-psi} holds, then $v$ satisfies the equation $v^2 = y + xv$ and $\psi(t) = t^2$.
We let $f(t) := \int \Theta \, dt$ be an antiderivative of $ \Theta(t)$, so $\int \psi' \Theta' = 2 ( t f' - f)$ and
\[
q(x,y) = x \Theta(v) - ( \textstyle{\int} \psi' \Theta') (v) = 2f(v) - \sqrt{x^2 + 4y} f'(v).
\]
By~\eqref{eqn:qx,qy-from-MA-1}, $q(x,y) = C(x-x_0)$ if and only if $f$ is linear and $q(x,y) = C(y-y_0)$ if and only if $f$ is quadratic; the function $q$ is again homogeneous after translation in both cases.

Otherwise, we rearrange equation~\eqref{eqn:A''B''-PDE} as $q^2 A'' - S^x A' = 2 (q^2 - y S^y)$ and differentiate twice to obtain
\begin{equation}\label{eqn:a-before-MA}
(q q_{yy} + q_y^2) A'' - 2 q_x q_{yy} A' = 2q q_{yy} - 2q_y^2 - 4 y q_y q_{yy}.
\end{equation}
We use the relations of ~\eqref{eqn:qx,qy-from-MA-1} and~\eqref{eqn:qx,qy-from-MA-2} to express $q$ and its derivatives in terms of $f$ as
\[
q - x q_x - y q_y = 2f - 2 v f' + v^2 f'', \qquad q_{ij} = - f^{(3)} (v) \frac{v_i v_j}{v_y}, \quad i,j \in \{ x, y \},
\]
where $v_x = v v_y = \frac{v}{\sqrt{x^2 + 4y}}$.
Writing $\tilde{A}(x) := A(x) + x^2$ and $\sqrt{x^2 + 4y} = 2v - x$ makes equation~\eqref{eqn:a-before-MA} into
\begin{equation}\label{eqn:equation-for-constants-v-A''-A'}
\left( \frac{- f f^{(3)} + v \, (f' f'')'}{(f' - v f'') f^{(3)}} - x \frac{ (f' f'')'}{2 (f' - vf'') f^{(3)}} \right) \tilde{A}''(x) + \tilde{A}'(x) = - 2 \frac{2f - 2vf' + v^2 f''}{f' - vf''},
\end{equation}
which has the form $\left( F_1(v) - x F_2(v) \right) \cdot \tilde{A}''(x) + \tilde{A}'(x) = G(v)$, for functions $F_i, G$.
Differentiating by $\py^2$ yields
\[
(F'_1 - x F'_2) G'' = (F''_1 - x F''_2) G'. 
\]
This condition requires that either $G$ be constant, or $F'_1/G'$ and $F'_2/G'$ be constant functions, so $F_i = k_i G + m_i$ for $(k_1, k_2) \neq (0,0)$ and $\tilde{A}''(x) = \frac{1}{k_1 - k_2 x}$.
The latter forces $\tilde{A}'(x) = - \frac{1}{k_2} \log |k_1 - k_2 x| + k_3$, which does not satisfy the original condition.
Therefore, $G$ is constant and $(F'_1 - x F'_2) \cdot \tilde{A}''(x) = 0$, so either $\tilde{A}''(x) = 0$, or $F_i = k_i$ are constant.
In the first case, $A$ and $B$ are both quadratic, so $q$ is homogeneous up to translating $x,y$ by Section \ref{subsect:solitons-conformal-S2H2}.
Otherwise, the constants $F_i = k_i$ coming from equation~\eqref{eqn:equation-for-constants-v-A''-A'} satisfy
\[
k_1 = F_1 = - \frac{f}{f' - v f''} - 2 v F_2 = - \frac{f}{f' - vf''} - 2 v k_2.
\]
Therefore, $f$ satisfies the ODE 
\begin{equation}\label{eqn:firstOdeForf}
f + (k_1 - 2 k_2 v) f' + (2 k_2 v^2 - k_1 v) f'' = 0.
\end{equation}
Also, the expression for $g_0 = G(v)$ is
\begin{equation}\label{eqn:g0-secod-ODE}
g_0 = \frac{2f - 2v f' + v^2 f''}{f' - vf''} = \frac{2(k_1 - 2 k_2 v) f' + 2 (2k_2 v^2 - k_1 v) f'' - 2 vf' + v^2 f''}{f' - vf''},
\end{equation}
unless $f$ is quadratic or linear.
Equation~\eqref{eqn:g0-secod-ODE} produces a first-order ODE in $f'$ that can be integrated via the method of partial fractions to yield
\[
f'(v) = C_1 v \left( (4k_2 + 1) v - (2k_1 - g_0) \right)^{\frac{1}{4 k_2 + 1}}
\]
Differentiating~\eqref{eqn:firstOdeForf} shows $(1 - 2 k_2) f' + 2 k_2 v f'' + (2 k_2 v^2 - k_1 v) f''' = 0$.
Combing this with the above computation for $f'$ via partial fractions demonstrates that it does not satisfy the equation unless $k_2 = \frac{1}{2}$.
In this case, equation~\eqref{eqn:firstOdeForf} is solved as $f(t) = C_1(t-k_1) + C_2(t-k_1) \log |t-k_1|$, which satisfies equation~\eqref{eqn:g0-secod-ODE} if and only if $C_2 = 0$.
This means that $f$ is linear and $q = C(y-y_0)$ is homogeneous after translation.
\end{proof}
Before proving Theorem \ref{thm:secondary-thm}, we establish a final auxiliary ingredient, treating the last possible source of degeneracy in the soliton equations.
Corollary~\ref{corollary:apxr+bpyr=c} shows that a function $q$ satisfies $\px \py \log \frac{q_x}{q_y} = 0$ if and only if $\frac{q_y}{q_x} = \frac{f_2'(y)}{f_1'(x)}$; equivalently, if and only if it is locally expressible as $\theta (f_1(x) + f_2(y))$, for some $f_1, f_2$.
\begin{lemma}\label{prop:q(f+g)}
Consider a soliton metric $g$ of the form~\eqref{eqn:PaperMetric}, where $q$ satisfies $D^2 q \neq 0$ and is locally expressible in the form $\theta (f_1(x) + f_2(y))$ for some single-variable functions $\theta, f_1,f_2$.
Then, up to translations and scaling of $x,y$, the metric $g$ is one of~\eqref{metric:pbnski},$\mr{(}$\hyperref[prop:metrics-s2-through-s5]{\mr{\metricStarStarBar}}$\mr{)}$, or \eqref{eqn:singular-soliton-metrics}.
\end{lemma}
\begin{proof}
Using this expression for $q$, we compute its second derivatives in~\eqref{eqn:HessianCompatibility} to obtain
    \begin{equation}\label{eqn:Hessian-q(f+g)-1}
    f_1'' f_2'' (\theta')^2 + (f_1')^2 f_2'' \theta' \theta'' + f_1'' (f_2')^2 \theta' \theta'' = 0,
    \end{equation}
where we denote $\theta := \theta(f_1(x)+f_2(y))$ for brevity, likewise for $\theta',\theta''$.
If $f_1''=0$, then $D^2 q \neq 0$ (so $f_1' f_2' \theta' \theta'' \neq 0$) requires $f_2''=0$, so $f_1,f_2$ are linear and $q=\theta(x+y)$ up to linear transformations.
Otherwise, dividing~\eqref{eqn:Hessian-q(f+g)-1} by $f_1'' f_2'' \theta' \theta''$ and applying $\px \py$ gives $\frac{d^2}{dt^2} (\theta'/\theta'') \big\rvert_{t=f_1(x)+f_2(y)} = 0$, so $\theta'/\theta''$ is linear.
If it equals a constant $k$, then $\theta(t) = c_1 e^{t/k} + c_0$ and $\frac{(f_1')^2}{f_1''}, \frac{(f_2')^2}{f_2''}$ must equal constants $k_x, k_y$ with $k_x + k_y = -k$; the solution to such an ODE is $f_1(x) = c_3 - k_x \cdot \log(x+c_2)$ and likewise for $f_2$.
Up to the above transformations, this means $q(x,y) = x^{\alpha} y^{1-\alpha}$ for some $\alpha$.
Next, if $\theta'/\theta''$ has non-zero leading coefficient $\frac{1}{\alpha-1}$ for $\alpha \neq 1$, then $\alpha = 0$ leads (up to rescaling) to $\theta(t) = \log t$ and $f_1,f_2 = \exp(-)$.
For $\alpha \neq 0,1$, $\theta(t) = (c_{1,1} t + c_{1,0})^{\alpha} + c'_0$ transforms~\eqref{eqn:Hessian-q(f+g)-1} an uncoupled pair of ODEs for $f_1(x),f_2(y)$, of the form $\frac{1}{\alpha - 1} f_i + (f_i')^2/f_i'' = \text{\sffamily{const}}.$
This gives $f_1(x) = c_{2,1}(x+c_{2,0})^{1/\alpha} - c'_{2,0}$ and likewise for $f_2(y)$, so~\eqref{eqn:Hessian-q(f+g)-1} becomes $q = (x^{\alpha} \pm y^{\alpha})^{1/\alpha}$.

The conformal factors $q(x,y) = x^{\alpha} y^{1-\alpha}$ and $(x^{\alpha} \pm y^{\alpha})^{1/\alpha}$ obtained above are homogeneous of degree $1$, so the metric $g$ has one of the forms classified in Theorem \ref{thm:main-theorem}.
We conclude that $q(x,y) = x^{\alpha} y^{1-\alpha}$ forces $g$ to have the form~\eqref{eqn:singular-soliton-metrics}, while $q(x,y) = (x^{\alpha} \pm y^{\alpha})^{1/\alpha}$ produces no solutions for $\alpha \neq 1$.
Finally, for $q(x,y) = \log(e^x + e^y)$, integrating~\eqref{eqn:pySxpxSy} --~\eqref{eqn:pySy} gives $S^x = \frac{2 e^y}{e^x + e^y} + 2q + C_0 y + C_1$, likewise for $S^y$.
By the same procedure as above, we obtain $C_0 = 0$ and $A = e^{k x}, B = e^{k y}$.
Substituting these functions into equations~\eqref{eqn:A''B''-PDE} and~\eqref{eqn:RicciSoliton-Last-Sx} shows that there are no admissible solutions in this case.
\end{proof}
We now prove Theorem~\ref{thm:secondary-thm}, which consists of two parts.
Suppose that the surface factor $\Sigma_2$ with local coordinates $(y,s)$ has constant curvature, meaning that $B(y)$ is a (possibly degenerate) quadratic. 
Rescaling and translating $y$ allows us to consider just two cases, of constant zero or non-zero curvature for the surface $\Sigma_2$, for which we take $B(y) = \pm y^2 + b_0, B(y) = y$, or $B(y) = b_0$.
We first handle the case where the constant Gauss curvature of the second surface factor vanishes.
\begin{proof}[Proof of Theorem~\ref{thm:secondary-thm} part I: $K = 0$]
We first consider the case $K=0$, meaning that $B(y) \in \{ y, b_0 \}$.
    If $A$ and $B$ are both constant, the metric is described by Lemma~\ref{lemma:conformally-flat} and is given by~\eqref{eqn:conformally-flat-metric}; thus, in what follows, we suppose that $A$ is non-constant.
First, if $B$ is constant, then equation~\eqref{eqn:A''B''-PDE} yields $S^x = q^2 F(x)$ for some $F(x)$, so equations~\eqref{eqn:pySxpxSy},~\eqref{eqn:pxSx}, and $\py$ applied to~\eqref{eqn:pxSx} can be combined as
        \[
        \begin{bmatrix} - 2 q_x q_{yy} & q_y^2 + q q_{yy} & 0 \\
        - 2 q_x^2 & 2 qq_x & q^2 \\
        - 2 q_x q_{xy} & q_x q_y + q q_{xy} & q q_y
        \end{bmatrix} \begin{bmatrix}
            1 \\ F \\ F'
        \end{bmatrix} = \begin{bmatrix}
            0 \\ 0 \\ 0
        \end{bmatrix}.
        \]
        The $3 \times 3$ matrix has determinant $2 q q_x q_y^2 (q_x q_y - q q_{xy})$, which must vanish for there to be a solution $F$.
        Since $q$ is not a single variable function, $q q_x q_y^2 \neq 0$, so we conclude that $q_xq_y = q q_{xy}$ and $\px \py \log q = 0$.
        By Lemma~\ref{prop:q(f+g)}, $q = e^{x+y}$ and $g$ is the metric~\eqref{eqn:Soliton-q=exp(x+y)-metric} with $k_2 = 0$.

        To treat the cases $B(y) = y$, we will use the result of Theorem~\ref{thm:solutions-to-monge-ampere}, which classifies the local forms of solutions to the Monge-Amp\`ere equation $\det D^2 q = 0$.
        For $D^2 q$ as in Theorem~\hyperref[thm:solutions-to-monge-ampere]{\ref{thm:solutions-to-monge-ampere}(i)}, all the soliton metrics of this form are classified in Proposition \ref{prop:d2q=0}.
        For $q$ is as in Theorem~\hyperref[thm:solutions-to-monge-ampere]{\ref{thm:solutions-to-monge-ampere}(ii)}, all the soliton metrics with $q = \theta(x+y) + cx$ are classified in Proposition \ref{prop:metrics-s2-through-s5}.

        It therefore remains to consider conformal factors of the form $q = x \Theta(v) - ( \int \psi' \Theta')(v)$ as in Theorem~\hyperref[thm:solutions-to-monge-ampere]{\ref{thm:solutions-to-monge-ampere}(iii)}.
        When $B(y) = y$, we have $B'' = B^{(3)} = 0$, so the left-hand side of equation~\eqref{eqn:second-compatibility-equation} becomes 
        \begin{equation}\label{eqn:second-compatibility-linear}
            \frac{2}{A'} (A' q_x - q_y)^2 = A^{(3)} q^2 D +  \frac{qA''}{D} \left[ \lambda + \frac{q^2 A''}{2} -  q (q_{xx} A + q_{yy} y) + 3 A q_x^2 - y q_y^2 + 2 \left( A' y - \frac{A}{A'} \right) q_x q_y - 2 A'  q q_x \right]
        \end{equation}
        where $D := A q_x + A' y q_y - \frac{1}{2} A' q$.
        If $A^{(3)} \neq 0$, using the expressions for $q$ and $Dq, D^2 q$ from ~\eqref{eqn:qx,qy-from-MA-1} and~\eqref{eqn:qx,qy-from-MA-2} in equation~\eqref{eqn:second-compatibility-linear} produces a system for $S^x, S^y$, which satisfy
        \[
        \px S^x = 2 \Theta^2 + 2 v^2 \Theta'^2 - 2 v \Theta \Theta', \qquad \py S^y = 2 \Theta'^2, \qquad \py S^x + \px S^y = - 4 \Theta \Theta' + 4 v \Theta'^2.
        \]
        If $A^{(3)} \neq 0$, the resulting system of equations for the functions $\Theta, \psi$ that define $v$, and thus $q$, admits no solutions. 
        For $A^{(3)} = 0$, translating $x$ allows us to consider only the cases $A(x) \in \{ a x^2 + a_0, x \}$.
        By direct substitution into~\eqref{eqn:second-compatibility-linear}, the first case produces no solution if $a \neq 0$.
        The remaining case of $(A,B) = (x,y)$ was treated in Proposition~\ref{prop:metrics-s2-through-s5}, resulting in one of the metrics~\eqref{eqn:new-soliton} or $\mr{(}$\hyperref[prop:metrics-s2-through-s5]{\mr{\metricStarStarBar}}$\mr{)}$ above.
\end{proof}
The result of Theorem~\ref{thm:secondary-thm} a posteriori implies that $B(y) = y^2 + b_0$ with $b_0 \neq 0$ forces $A(x) = a_0 - (x - x_0)^2$ for some constants $a_0, x_0$.
    We now prove the final case of Theorem~\ref{thm:secondary-thm} where $K \neq 0$.
    This situation requires a different approach, more algebraic in nature: we will prove that the soliton equations involving $A, B$ impose a system of differential relations on the conformal factor $q$, which must be algebraically equivalent to another system of differential equations that includes \eqref{eqn:sixth-order-condition}.
    This will enable us to apply Lemma \ref{lemma:sixth-order-condition} to deduce that the function $q$ is homogeneous, and therefore classified by Theorem \ref{thm:main-theorem}.
    
    We introduce the key algebraic tools used in the second part of the proof.
    The local version of Bando's theorem \cite{real-analyticity} establishes that for any Ricci soliton metric $g$ as in \eqref{eqn:PaperMetric}, the function $q$ must be real analytic in the interior of the domain $\Omega$ where the soliton metric~\eqref{eqn:PaperMetric} is defined. 
    By the Weierstrass preparation theorem, the ring of germs of two-variable analytic functions $C^{\omega}(x,y)$ is Noetherian.
        The Hilbert basis theorem shows that the polynomial ring $C^{\omega}(x,y) [t_1, \dots, t_N]$ is also Noetherian.
        Therefore, the Lasker–Noether theorem implies that every ideal of $C^{\omega}(x,y) [t_1, \dots, t_N]$ can be decomposed uniquely as an intersection of finitely many primary ideals.
        Recall that the sum and the intersection of any $\fp$-primary ideals in a Noetherian ring (i.e.,~ones sharing the same associated prime $\fp$) are again $\fp$-primary.

        We will also use the Nullstellensatz for real analytic ideals, see for example \cite{nullstellensatz}*{Lemma 2.5}.
        The reader is referred to the book of Milnor, \cite{milnor}*{Ch. 2}, for the local analytic form of the Nullstellensatz used here, and its application in differential geometry, combined with the Weierstrass preparation theorem.
See also \cite{exterior-differential}*{Ch. VIII \S 5} for the application of such commutative algebra techniques, notably the Nullstellensatz, to the algebra of differential relations in the manner considered here.
\begin{proof}[Proof of Theorem~\ref{thm:secondary-thm} part II: $K \neq 0$]
   We finally treat the case $K \neq 0$, corresponding to $B(y) = \pm y^2 + b_0$, where $b_0 \neq 0$ unless $B$ is a hyperbolic cusp.
   We may write for simplicity $B(y) = y^2 + b_0$; the case $B(y) = b_0 - y^2$ is identical.
   We again use the classification of local solutions to the Monge-Amp\`ere equation $\det D^2 q = 0$ from Theorem~\ref{thm:solutions-to-monge-ampere} to write $q = x \Theta(v) - ( \int \psi' \Theta')(v)$ as in Theorem~\hyperref[thm:solutions-to-monge-ampere]{\ref{thm:solutions-to-monge-ampere}(iii)}.
   As noted above, the metrics resulting from $q$ having the forms~\hyperref[thm:solutions-to-monge-ampere]{\ref{thm:solutions-to-monge-ampere}(i)} or~\hyperref[thm:solutions-to-monge-ampere]{\ref{thm:solutions-to-monge-ampere}(ii)} are treated in Propositions~\ref{prop:d2q=0} and~\eqref{prop:metrics-s2-through-s5}, respectively.
   
   Substituting $\frac{1}{2} w(x,y) = 1 - \frac{S^y}{q^2} y$ from~\eqref{eqn:A''B''-PDE}, we solve equation~\eqref{eqn:RicciSoliton-Last-Sx} for $A$ as
        \begin{equation}\label{eqn:solve-for-A}
    A = \frac{- \lambda q - q^3 + S^y q y + 2 q^2 q_y y - (S^y q_y + q q_y^2 - q^2 q_{yy}) B}{S^x q_x + q q_x^2 - q^2 q_{xx}} + \frac{q^2 q_x}{S^x q_x + q q_x^2 - q^2 q_{xx}} A'.
        \end{equation}
        We write $A = \frac{\cN + q^2 q_x A'}{S^x q_x + q q_x^2 - q^2 q_{xx}}$, for $\cN$ the numerator of the first term in~\eqref{eqn:solve-for-A}.
        Applying $\partial_y$ to~\eqref{eqn:A''B''-PDE} implies
        \begin{equation}\label{eqn:a-from-pde}
            A' = 2 \frac{y (2 q q_y^2 - 2 q_y S^y) + S^y q}{q \py S^x - 2 q_y S^x}.
        \end{equation}
        The condition $\py A' = 0$ in~\eqref{eqn:a-from-pde} is then equivalent to
        \begin{equation}\label{eqn:sx-sy-from-a'}
        \begin{split}
        & \py S^x \left( 4 q^2 q_y (q q_y + y q_{yy}) + ( qq_y - 2 y (q q_{yy} + q_y^2)) S^y \right)- 2 S^x \left( 4 q q_y^3 - (q q_{yy} + q_y^2) S^y + 2 y q_y^2 (q q_{yy} - q_y^2) \right) \\
    &= 8 y q q_x q_y q_{yy} (qq_y - S^y).
        \end{split}
        \end{equation}
        On the other hand, the condition $\px A' - \frac{S^x}{q^2} A' = 2 - 2\frac{S^y}{q^2}y$ from~\eqref{eqn:A''B''-PDE} can be rearranged, using~\eqref{eqn:a-from-pde}, into
        \begin{equation}\label{eqn:a''-a'-equation}
        \begin{split}
            0 &=  (\partial_y S^x)^2 \left( -2q^2 + 2yq q_y + y S^y \right) + 2 q_y \, \partial_y S^x \left( 2 q_x (q^2 - yq q_y) + 3 q S^x - 2 y \tfrac{S^y}{q} S^x \right) \\
&+ 4q_y^2 (S^x)^2 \left( -1 + y \tfrac{S^y}{q^2} \right) + S^x S^y \left( 2y \tfrac{q_y}{q} - 1 \right) + 4y q_x^2 q q_y (q q_y - q_y^2 - S^y).
        \end{split}
        \end{equation}
        Computing $A' = \px A$ in two ways from~\eqref{eqn:solve-for-A}, and using $A'' = 2 + \frac{S^x}{q^2} A' - 2\frac{S^y}{q^2} y$ by~\eqref{eqn:A''B''-PDE}, we obtain
        \begin{equation}\label{eqn:compute-in-two-ways-quadratic}
\px \cN - \cN \frac{3 q_x^3 + S^x q_{xx} - q^2 q_{xxx}}{S^x q_x + q q_x^2 - q^2 q_{xx}} + q (2 q q_{xx} + q_x^2) A' + 2 q^2 q_x - 2 q_x S^y y = 0.
        \end{equation}
        Moreover, equation~\eqref{eqn:solve-for-A} defining a single-variable function $A(x)$ requires $\py A = 0$, meaning
        \begin{equation}\label{eqn:single-variable-in-x}
        \begin{split}
            & \py \cN (S^x q_x + q q_x^2 - q^2 q_{xx}) - \cN ( q_x^2 q_y + q_x \py S^x + S^x q_{xy} + 2 q q_x q_{xy} - 2 q q_y q_{xx} - q^2 q_{xxy} ) \\
    & - 2 q q_x^2 \left( y (2 q q_y^2 - 2 q_y S^y) + S^y q \right) + A' q^2 ( q_x^3 q_y + q^2 q_x q_{xxy} - q^2 q_{xy} q_{xx} - q q_x^2 q_{xy}) = 0.
    \end{split}
        \end{equation}
        Note that equations~\eqref{eqn:compute-in-two-ways-quadratic} and~\eqref{eqn:single-variable-in-x} together imply~\eqref{eqn:sx-sy-from-a'}, since $\py A' = \px (\py A) = 0$. 
        Here, 
        \begin{align*}
            \cN &= - \lambda q - q^3 + (qy - q_y B) S^y + 2 q^2 q_y y - q (q_y^2 - q q_{yy}) B, \\
            \px \cN &= - \lambda q_x - 3 q^2 q_x + ( q_y B - q y) (\py S^x - 2 q q_{xy}) + S^y (q_x y - q_{xy} B) + 8 q q_x q_y y - ( 5 q_x q_y^2 - 2 q q_x q_{yy} - q^2 q_{xyy}) B, \\
            \py \cN &= - \lambda q_y - q^2 q_y + 4 q q_y^2 y + (q - q_y y - q_{yy} B) S^y + 4 q^2 q_{yy} y - (3 q_y^3 - q^2 q_{yyy}) B.
        \end{align*}
        In the computation of $\px \cN$, we used \eqref{eqn:pySxpxSy} to write $\px S^y = 4 q_x q_y - \py S^x$.
        This substitution ensures that the triple of equations~\eqref{eqn:a''-a'-equation}, \eqref{eqn:compute-in-two-ways-quadratic}, and \eqref{eqn:single-variable-in-x} is expressed using only the three unknowns $(S^x, S^y, \py S^x)$ and has a unique solution of the form $(S^x , S^y, \py S^x) = (\cE_1, \cE_2, \cE_3)$, for three expressions $\cE_i$ only involving $q, Dq, D^2 q, D^3 q, \lambda, b_0$.
        The expression~\eqref{eqn:a''-a'-equation} is quadratic in $\py S^x$, so we may express it explicitly in terms of $S^x, S^y, D^i q$ via the quadratic formula; equations~\eqref{eqn:compute-in-two-ways-quadratic} and~\eqref{eqn:single-variable-in-x} define similar quadratics in $\py S^x$ upon substituting $A'$ from~\eqref{eqn:a-from-pde}.
        Solving these relations for $\py S^x$ and equating the three resulting expressions produces the unique solution $(S^x, S^y) = (\cE_1, \cE_2)$, with $\py S^x = \cE_3$.
        The equations
        \[
        \px \cE_1 = 2 q_x^2, \qquad \py \cE_2 = 2 q_y^2, \qquad \py \cE_1 = \cE_3, \qquad \cE_3 + \px \cE_2 = 4 q_x q_y
        \]
        coming from~\eqref{eqn:pySxpxSy} -- \eqref{eqn:pySy} and the solvability of the system therefore produce a quadruple of fourth-order differential equations for $q$.
        If $b_0 \neq 0$, these equations are linearly independent.

        We now claim that for $b_0 \neq 0$, the resulting system of differential equations in $\{ D^i q \}_{0 \leq i \leq 4}$ is equivalent to $q$ being homogeneous and satisfying~\eqref{eqn:resulting-ODE} for the given values of $\lambda, b_0$, and some $a_0$.
        In what follows, we let $M$ denote the free $C^{\omega}(x,y)$-module generated by the formal symbols corresponding to $q$ and its derivative of order $\leq 6$, which we represent by the indeterminates $t_{i,j} := \partial_x^i \partial_y^j q$.
        In this formulation, differential relations among the derivatives of $q$ correspond to submodules of $M$.
        Let $\cI_0 \subset M$ be the ideal formed by all the differential equalities of order $\leq 6$ satisfied by a solution $q$ to $\det D^2 q = 0$, so
        \[
        \cI_0 := \left\{ D^{\alpha} ( q_{xx} q_{yy} - q_{xy}^2) \; : \; |\alpha| \leq 4 \right\}.
        \]
        Let $N_0$ be the submodule generated by $\cI_0$.
        We also define 
\begin{equation}\label{eqn:ell}
\begin{split}
    \cL ( q, Dq, D^2 q ; b_0, \lambda) &:= 2 q_y q^2 + y ( b_0^{-1} \lambda q + q q_y^2 - q^2 q_{yy}) + x q q_x q_y - x q^2 \frac{q_{xx}}{q_x} q_y.
\end{split}
\end{equation}
Finally, we define the expressions $\cR^x (D^{0 \leq i \leq 3} q; b_0, \lambda)$ and $\cR^y (D^{0 \leq i \leq 3} q ; b_0, \lambda)$ by
\begin{align}
\begin{split}\label{eqn:culry-rx-definition}
    \cR^x &:= - \px \left( \frac{x q_y}{q_x} \right) \\
    & \; \; \; \; + \px \frac{\px \cL - \cL \frac{q_{xx}}{q_x} \pm \sqrt{\left( \px \cL - \cL \frac{q_{xx}}{q_x} \right)^2 - 4 \cL \frac{q_x q_y + x (q_x q_{yy} + q_y q_{xx})}{q_x} \left( 3 q_x^2 - q q_{xx} - \frac{q^2 q_{xxx}}{q_x} + \frac{q^2 q_{xx}^2}{q_x^2}\right)}}{2 q_x \left( 3 q_x^2 - q q_{xx} - \frac{q^2 q_{xxx}}{q_x} + \frac{q^2 q_{xx}^2}{q_x^2} \right)}, 
\end{split} \\
\begin{split}\label{eqn:culry-ry-definition}
    \cR^y &:= - x \py \left( \frac{q_y}{y q_x} \right) \\
    & \; \; \; \; + \py \frac{\px \cL - \cL \frac{q_{xx}}{q_x} \pm \sqrt{\left( \px \cL - \cL \frac{q_{xx}}{q_x} \right)^2 - 4 \cL \frac{q_x q_y + x (q_x q_{yy} + q_y q_{xx})}{q_x} \left( 3 q_x^2 - q q_{xx} - \frac{q^2 q_{xxx}}{q_x} + \frac{q^2 q_{xx}^2}{q_x^2}\right)}}{2 y q_x \left( 3 q_x^2 - q q_{xx} - \frac{q^2 q_{xxx}}{q_x} + \frac{q^2 q_{xx}^2}{q_x^2} \right)}.
\end{split}
\end{align}
The expressions $\cR^x$ and $\cR^y$ have the following significance: let $u(x,y)$ be a function defined so that
\begin{equation}\label{eqn:u-equality}
\px \dfrac{2 b_0 q_y q^2 + y \left( \lambda q + b_0 q q_y^2 - b_0 q^2 q_{yy} + u(x,y) \cdot q q_x^2 - u(x,y) \cdot q^2 q_{xx}  \right)}{b_0 x q_y - u(x,y) \cdot y q_x} = 2 q_x^2.
\end{equation}
The function $u = a_0$ is equal to some constant if and only if $\cR^x = \cR^y = 0$, in which case the equality~\eqref{eqn:u-equality} becomes $\cF(D^{0 \leq i \leq 3} q; a_0, b_0, \lambda) = 0$, for this $a_0$ and the expression $F$ defined in~\eqref{eqn:f-q-dq-d2q-d3q-a,b,l}.
The first equivalence comes from rewriting the fraction in~\eqref{eqn:u-equality} as
\[
b_0 \frac{2 q_y q^2 + y ( b_0^{-1} \lambda q + q q_y^2 - q^2 q_{yy}) + x q q_x q_y - x q^2 \frac{q_{xx}}{q_x} q_y}{b_0 x q_y - u(x,y) \cdot y q_x} + \frac{q^2 q_{xx}}{q_x} - q q_x
\]
and expressing equation~\eqref{eqn:u-equality} as a quadratic in $b_0^{-1} (b_0 x q_y - u \cdot y q_x)$.
We may thus solve the equation further, in terms of $u$, which shows that the conditions $\cR^x = 0$ and $\cR^y =0$ are equivalent to $\px u = 0$ and $\py u = 0$, respectively.
In this case, we have $u = a_0$ for some constant, so equation~\eqref{eqn:u-equality} becomes equivalent to the derivation of $\cF(D^{0 \leq i \leq 3} q; a_0, b_0, \lambda) = 0$ in~\eqref{eqn:f-q-dq-d2q-d3q-a,b,l} for this $a_0$.

Note that a differential equality of the form $\partial_i (a + \sqrt{b}) = \partial_i c$ can be rearranged into the differential polynomial equation $(\partial_i b)^2 = 2 b (\partial_i c - \partial_i a)^2$.
Using this rearrangement, let us denote $\tilde{\cR}^x, \tilde{\cR}^y$ to be the expressions obtained by expressing the relations $\cR^x = 0$ and $\cR^y = 0$ as polynomials of $\{ D^{0 \leq i \leq 4} q ; b_0, \lambda \}$ in this manner, thereby removing the radicals.
Likewise, we denote by $\tilde{\cT}^x$ and $\tilde{\cT}^y$ the polynomials corresponding to the two expressions of~\eqref{eqn:sixth-order-condition} under this rearrangement.
We may then consider the polynomial ideals
\begin{align*}
    \cI &:= \cI_0 \oplus \la \px \cE_1 - 2 q_x^2, \py \cE_2 - 2 q_y^2, \py \cE_1 - \cE_3, \cE_3 + \px \cE_2 - 4 q_x q_y \rg, \\
    \cJ &:= \cI_0 \oplus \la \tilde{\cR}^x, \tilde{\cR}^y, \tilde{\cT}^x, \tilde{\cT}^y \rg,
\end{align*}
and define $N_{\cI}$ and $N_{\cJ}$ to be the submodules of $M$ generated by the conditions contained in $\cI$ and $\cJ$, respectively.
The submodule $N_{\cI}$ encodes all the conditions on $q$ for the existence of a Ricci soliton metric~\eqref{eqn:PaperMetric} with $B(y) = y^2 + b_0$.
The submodule $N_{\cJ}$ contains the functions $q$ that are homogeneous upon translating $x,y$ and, when expressed as $q = (y-y_0) f \left( \frac{x-x_0}{y-y_0} \right)$, satisfy equation~\eqref{eqn:resulting-ODE} for some $a_0$.
Indeed, such functions satisfy the Monge-Amp\`ere equation $\det D^2 q = 0$ by~\eqref{eqn:homogeneous-partials}, hence also all the relations of $\cI_0$.
By the computation of Lemma~\ref{lemma:sixth-order-condition}, they satisfy the condition~\eqref{eqn:sixth-order-condition}, so the polynomials $\tilde{\cT}^x, \tilde{\cT}^y$ vanish.
Finally, the computation~\eqref{eqn:u-equality} shows that the polynomials $\tilde{\cR}^x$ and $\tilde{\cR}^y$ vanish on a homogeneous function $q$ if and only if it satisfies~\eqref{eqn:resulting-ODE} for some $a_0$.

We will prove that $\cI = \cJ$ generate the same set of differential conditions on $q$.
As noted above, any submodule of the free $C^{\omega}(x,y)$-module $M$ is finitely generated and Noetherian.
Hence, ascending chains of submodules stabilize.
Our work in Section~\ref{subsect:solitons-conformal-S2H2} shows that any function $q$ satisfying the relations in $\cJ$ produces a Ricci soliton metric with the terms given above, so $N_{\cI} \subseteq N_{\cJ}$ at the level of differential constraints.
Equivalntly, $\cV(N_{\cJ}) \subseteq \cV(N_{\cI})$ on the level of vanishing loci of modules, where $\cV(N)$ denotes the set of all analytic functions $q$ satisfying all differential relations in $N$.
By the real Nullstellensatz, applied in the form of \cite{nullstellensatz}*{Lemma 2.5}, this implies $\cI \subseteq \sqrt[\bR]{\cJ}$, where $\sqrt[\bR]{\cJ} := \{ f : f^{2k} + \sum_i g_i^2 \in \cJ \}$ denotes the real radical of $\cJ$.
Moreover, the Lasker-Noether theorem implies that every submodule of a finitely generated module in a Noetherian ring admits a primary decomposition, and the above containment shows that the associated primes of $N_{\cJ}$ must contain those of $N_{\cI}$.

On the other hand, the generators of $\cI$ are independent and the ideal $\la \px \cE_1 - 2 q_x^2, \py \cE_2 - 2 q_y^2, \py \cE_1 - \cE_3, \cE_3 + \px \cE_2 - 4 q_x q_y \rg$ is coprime with $\la \{ t_{i,j} \} \rg_{5 \leq i+j \leq 6} = \la \{ \partial_x^i \partial_y^j q  \} \rg_{5 \leq i+j \leq 6}$, so the differential conditions added to define of $N_{\cI}$ are of lower order than those defining $N_{\cJ}$. 
Therefore, $N_{\cI} \subseteq N_{\cJ}$ is only possible if $N_{\cI} = N_{\cJ}$, meaning that the ideals $\cI = \cJ$ are equivalent under the algebraic prolongation of the differential system defined by $\det D^2 q = 0$, thus defining the same differential conditions on $q$. 
In particular, any $q$ producing a Ricci soliton with $B(y) = y^2 + b_0$ must satisfy the differential relations $\tilde{\cT}^x$ and $\tilde{\cT}^y$, whereby Lemma \ref{lemma:sixth-order-condition} shows that it is homogeneous after translating $x,y$.
This completes the proof.
    \end{proof}
    The assumption that $b_0 \neq 0$ in this result is necessary and cannot be dropped: indeed, the expression~\eqref{eqn:ell} for $\cL$ is not defined for $b_0 \neq 0$, and no analogous definition of the ideal $\cI$ is possible as in the above argument.
    Moreover, when $b_0 = 0$, the function $A(x)$ is not necessarily quadratic, as evidenced by the metrics \eqref{eqn:singular-soliton-metrics}: for $\lambda = 0$ and any $\alpha \not\in \{ 0, \frac{1}{2}, 1\}$, we can take $B(y) = F_{1-\alpha, 1,0}(y) = y^2$ and $A(x) = F_{\alpha, -1, k}(x)$ for $k \neq 0$, which is not quadratic.
    We hypothesize that for the cusp metric with $B(y) = y^2$, the Ricci soliton system may admit solutions that do not necessarily have conformally cylindrical base, so $q$ is not necessarily homogeneous.


\end{document}